\documentclass[a4paper,10pt]{article}
\usepackage[english]{babel}  % Deutsch (neue Rechtschreibung)
\usepackage{amsmath}  % Allgemeiner Formelsatz
\usepackage{amssymb}  % Noch ein paar mathematische Symbole
\usepackage{amsthm}   % Mehr Optionen f\"ur theorem-Umgebungen
\usepackage[utf8]{inputenc}  % universelle Eingabecodierung
\usepackage[left=2cm,right=2cm,top=2cm,bottom=2cm]{geometry}
\usepackage{bbm}
\usepackage{tkz-euclide}
\usepackage{mathtools}
\usepackage{enumitem}
\usepackage{setspace}
\usepackage{float}
\usepackage{cancel}
\allowdisplaybreaks

\usepackage{hyperref}
\hypersetup{
  colorlinks = true, %Colours the links instead of ugly red boxes
  urlcolor = blue, %Colour for external hyperlinks
  linkcolor = blue, %Colour of internal links
  citecolor = blue %Color of cites
}

  \definecolor{asparagus}{rgb}{0.53, 0.66, 0.42}

\newtheorem{thm}{Theorem}[section]
\newtheorem*{thm*}{Theorem}
\newtheorem{prop}[thm]{Proposition}

\newtheorem*{cor*}{Corollary}
\newtheorem{lem}[thm]{Lemma}
\newtheorem*{lem*}{Lemma}
\theoremstyle{definition}
\newtheorem{defin}[thm]{Definition}
\newtheorem{rmk}[thm]{Remark}

\newcommand{\id}{\mathbbm{1}}
\renewcommand{\div}{{\rm div}_x}
\newcommand{\supp}{\mathrm{supp}}
\newcommand{\Grad}{\nabla_x}

\newcommand{\res}{\mathop{\hbox{\vrule height 7pt width 0.5pt depth 0pt
			\vrule height 0.5pt width 6pt depth 0pt}}\nolimits}

\newcommand{\FF}{\mathbb{F}}

\newcommand{\RR}{\mathbb{R}}
\renewcommand{\SS}{\mathbb{S}}
\newcommand{\NN}{\mathbb{N}}

\newcommand{\cI}{\mathcal{I}}
\newcommand{\cM}{\mathcal{M}}
\newcommand{\cN}{\mathcal{N}}
\newcommand{\cR}{\mathcal{R}}

\newcommand{\mmu}{\boldsymbol{\mu}}
\newcommand{\hh}{\textbf{h}}
\newcommand{\uu}{\textbf{u}}

\newcommand{\vv}{\textbf{v}}
\newcommand{\ff}{\textbf{f}}

\newcommand{\ee}{\textbf{e}}

\newcommand{\R}{\textbf{R}}
\newcommand{\fR}{\mathfrak{R}}
\newcommand{\p}{\textbf{p}}
\newcommand{\vp}{\varphi}
\newcommand{\Ppsi}{\boldsymbol{\Psi}}
\newcommand{\pp}{\boldsymbol{\varphi}}
\newcommand{\pphi}{\boldsymbol{\phi}}
\newcommand{\rr}{\textbf{r}}
\renewcommand{\ss}{\textbf{t}}
\newcommand{\eps}{\varepsilon}
\newcommand{\oomega}{\boldsymbol{\omega}}

\newcommand{\bU}{\mathbf{U}}

\newcommand{\Ma}{\mathrm{Ma}}
\newcommand{\dist}{\mathrm{dist}}

\newcommand{\dx}{\, {\rm d}x}
\newcommand{\dtau}{\, {\rm d}\tau}
\newcommand{\ds}{\, {\rm d}s}
\newcommand{\dt}{\, {\rm d}t}
\newcommand{\bM}{{\rm M}}

\numberwithin{equation}{section}

\title{Brinkman's law as $\Gamma$-limit of compressible low Mach Navier-Stokes equations and application to randomly perforated domains}
\author{
	Peter Bella\footnote{
		Technische Universit\"at Dortmund,
		Fakult\"at f\"ur Mathematik, 
		44227 Dortmund, Germany. E-mail: peter.bella@udo.edu
	}  
	\and
	Friederike Lemming\footnote{Technische Universit\"at Dortmund,
		Fakult\"at f\"ur Mathematik, 
		44227 Dortmund, Germany. E-mail: friederike.lemming@tu-dortmund.de}
	\and
	Roberta Marziani\footnote{Dipartimento di Ingegneria dell'Informazione e  Scienze Matematiche, 53100 Siena, Italy.
		E-mail: roberta.marziani@unisi.it}
	\and
	Florian Oschmann\footnote{Institute of Mathematics, Czech Academy of Sciences, \v Zitn\'a 25, 115 67 Praha 1, Czech Republic. Email: oschmann@math.cas.cz}
}

\begin{document}

\maketitle
\begin{abstract}
We consider the time-dependent compressible Navier-Stokes equations in the low Mach number regime inside a family of domains $(\Omega_\eps)_{\eps > 0}$ in $\RR^3$. Assuming that $\lim_{\eps \to 0} \Omega_\eps = \Omega \subset \RR^3$ in a suitable sense, we show that in the limit
%Under appropriate assumptions on the convergence of these domains, we show that as $\eps \to 0$ and $\Omega_\eps \to \Omega \subset \RR^3$, 
the fluid flow inside $\Omega$ is governed by the incompressible Navier-Stokes-Brinkman equations, provided the latter one admits a strong solution. The abstract convergence result is complemented with a stochastic homogenization result for 
randomly perforated domains in the critical regime.

\end{abstract}

\tableofcontents
%%%%%%%%%%%%%%%%%%%%%%%%%%%%%%%%%%%%%%%%%%%%%%%%%%%%%%%%%%%%%%%%%%%%%%%%%
\section{Introduction}\label{sec:introduction}
%%%%%%%%%%%%%%%%%%%%%%%%%%%%%%%%%%%%%%%%%%%%%%%%%%%%%%%%%%%%%%%%%%%%%%%%%

% \red List to Peter
% \begin{itemize}
%     \item ${\rm Cap}\Omega_b^\epsilon\to 0$ in Lemma 3.1?
% \end{itemize}
% \black

% In this paper, we consider a family of bounded randomly perforated domains $(\Omega_\eps)_{\eps>0}$ in $\RR^3$, and assume that the size of the holes scales as $\eps^{3}$. We study the homogenization problem for a time-dependent compressible viscous fluid in $\Omega_\eps$ for a low Mach number limit. The fluid is modeled by the time-dependent Navier-Stokes system in the low Mach number regime:
Homogenization of compressible and incompressible Navier-Stokes equations has gained a lot of interest during the last decades. The simplest setting one may consider is  a smoothly bounded domain $\Omega \subset \RR^3$ with small holes (representing a container with tiny obstacles), precisely for $\eps>0$ let $\Omega_\eps$ be defined by
\begin{align}\label{firstDom}
    \Omega_\eps = \Omega \setminus \bigcup_{i \in I_\eps} B_{\eps^\alpha} (\eps z_i), \quad \alpha \geq 1,\quad z_i \in 2 \mathbb{Z}^3, \quad I_\eps = \{ i: \ \overline{B_\eps(\eps z_i)} \subset \Omega \}.
\end{align}

For given $T>0$, we consider in $(0,T) \times \Omega_\eps$ the compressible Navier-Stokes equations in the low Mach number regime $\Ma(\eps) \to 0$:
\begin{equation}\label{NavierSystem}
	\begin{cases}
	\partial_t\rho_\eps +\div(\rho_\eps \uu_\eps)=0 & \text{in } (0,T)\times \Omega_\eps,\\[1em]
\partial_t(\rho_\eps\uu_\eps)+\div(\rho_\eps \uu_\eps\otimes \uu_\eps)+\dfrac{1}{\Ma^2(\eps)}\Grad p(\rho_\eps) = \div\SS(\Grad \uu_\eps) + \rho_\eps \ff & \text{in } (0,T)\times \Omega_\eps,\\[1em]
\SS(\Grad \uu_\eps)=\mu\left( \Grad \uu_\eps+\Grad^T \uu_\eps-\frac{2}{3}\div(\uu_\eps)\id \right)+\eta \div(\uu_\eps)\id& \text{in } (0,T)\times \Omega_\eps,
	\end{cases}
\end{equation}
where $\rho_\eps$ and $\uu_\eps$ are the fluid's density and velocity, respectively, $p$ is the pressure, $\ff$ the force, and $\SS$ denotes the Newtonian viscous stress tensor with constant viscosity coefficients $\mu >0$, $\eta\geq 0$. Further, we impose no-slip boundary conditions
\begin{align}\label{no-slip}
	\uu_\eps=0 \qquad\text{ on }(0,T)\times\partial\Omega_{\eps}.
\end{align}
For the pressure $p$ and the force $\ff$ we assume
\begin{equation}\label{assumptions0}
		\begin{gathered}
	\ff\in L^\infty(0,T; L^\infty(\RR^3;\RR^3)),\quad 
p\in C([0,\infty))\cap C^2((0,\infty)),\quad p'(\rho)>0 \text{ if } \rho > 0, \quad p(0)=0,\\
\exists \gamma \geq 1, \,0<\underline{p}\le\overline{p} <\infty\colon \ \ \underline{p}\leq \liminf_{\rho\to\infty}\frac{p(\rho)}{\rho^\gamma}\leq\limsup_{\rho\to\infty}\frac{p(\rho)}{\rho^\gamma}\leq \bar p  .
		\end{gathered}
	\end{equation}
A prototypical example of such forces is gravity $\ff = - (0, 0, 1)^T$. As for the pressure, one might think of the adiabatic pressure law $p(\rho) = a \rho^\gamma$ for some $a>0$.\\

The flow behavior in the limit $\eps \to 0$ is strongly affected by the parameter $\alpha$, that is, by the size of the holes $B_{\eps^\alpha} (\eps z_i)$ in \eqref{firstDom}. In particular, if $1 \leq \alpha < 3$, then the holes create huge friction on the flow such that the latter eventually stops; a proper rescaling of the fluid's velocity leads to Darcy's law for the limiting system. If instead $\alpha>3$, then the holes do not influence the flow considerably and the limit flow is governed by the same Navier-Stokes equations we started with. Eventually, if $\alpha = 3$, which is the case we are concerned with in the present work, we are in the \emph{critical regime}. Indeed, in the latter, the holes create some friction that in the limit $\eps \to 0$ gives rise to an additional Brinkman term which resembles the so called Stokes drag law. 
Moreover, since we consider the case of a low Mach number ${\rm Ma}(\eps)\to 0$, the fluid shall become \emph{incompressible} in the limit $\eps \to 0$. Gathering all these considerations, we shall expect for $\alpha = 3$ and $\Ma(\eps) \to 0$ a limiting system of the form
\begin{equation}\label{systemstarkeloesung}
	\begin{cases}
	\div\uu=0&\text{ in }(0,T)\times\Omega,\\[1em]
\bar{\rho} (\partial_t\uu + (\uu\cdot\Grad\uu) + \Grad\Pi+\mu \bM\uu=\mu\Delta_x \uu+\bar{\rho} \ff&\text{ in }(0,T)\times\Omega,\\[1em]
\uu=0&\text{ on }(0,T)\times\partial\Omega,
	\end{cases}
\end{equation}
where now $\bar{\rho} > 0$ is constant, $\bM \in \RR_{\rm sym}^{3 \times 3}$ is a given positive definite symmetric matrix representing the Brinkman term, and $\Pi$ the associated pressure (Lagrange multiplier) to the divergence-free condition. \\

In this paper we consider a more general setting, 
namely, we replace the perforated domains in \eqref{firstDom} for $\alpha=3$
 with a
 family of domains $(\Omega_\eps)_{\eps > 0}$ which satisfies mild assumptions 
  and at the same time keeps the relevant features of the critical regime $\alpha=3$. In particular, we will assume $(\Omega_\eps)_{\eps > 0}$  converge in the sense of Mosco to a fixed domain $\Omega$ 
  (see Definition \ref{def:omega_eps}). 
  In a nutshell, our informal main result reads as follows:
\begin{thm}[See Theorem~\ref{mainTheorem}]\label{thm}
Let $\eps>0$ and $\Omega_\eps, \Omega \subset D \subset \RR^3$ be smoothly bounded domains confined to an overall bounded domain $D$. If $\lim_{\eps \to 0} \Omega_\eps = \Omega$ (in the sense of Definition \ref{def:omega_eps}) and $\Ma(\eps)$ with $\lim_{\eps \to 0} \Ma(\eps) = 0$ complies with this convergence of domains, and if \eqref{systemstarkeloesung} exhibits a strong solution $(\bar{\rho}, \uu)$, then a sequence of weak solutions $(\rho_\eps, \uu_\eps)$ to \eqref{NavierSystem} converges to $(\bar{\rho}, \uu)$ as $\eps \to 0$.
\end{thm}
We then complement our analysis by applying Theorem \ref{thm} to a specific choice of domains. More precisely, we consider a family of randomly perforated domains and show that, if the size of the holes is of order $\eps^3$, they converge in the sense of Mosco (Theorem \ref{theorem-random-holes}). We stress that the random setting is a generalization of the periodic one given in \eqref{firstDom}. \\

\noindent \textbf{State of the art.}
Theorem \ref{thm} shares similarities with the one proven in \cite{BellaFeireislOschmann2023b}. The main crucial difference is that in our case an additional Brinkman term $\mu \bM\uu$ pops up in the limit. The latter can be traced back to the fact that our reference domains $(\Omega_\eps)_{\eps>0}$ 
are essentially perforated domains with holes of critical size (that is, $\alpha = 3$ in \eqref{firstDom}), whereas the holes in \cite{BellaFeireislOschmann2023b} are much tinier (i.e., $\alpha > 3)$. Such Brinkman term 
appeared for the first time in the homogenization of the Dirichlet problem in \cite{CioranescuMurat1982} for holes of size $\eps^3$. Later, Allaire found in \cite{Allaire1989, Allaire1990a, Allaire1990b} (for all dimensions $d \geq 2$ and periodically distributed holes) that there are three different regimes concerning the size of the holes $\eps^\alpha$, $\alpha\geq 1$, for the homogenization of the incompressible Stokes equations, which are (for $d=3$) precisely the regimes $\alpha \in [1,3)$, $\alpha=3$, and $\alpha>3$ mentioned above. %For small holes, meaning $\alpha> 3$ if $d=3$, the holes are not big enough to have an impact on the fluid such that the limit of the fluid motion is still governed by the incompressible Stokes equation. For large holes, $\alpha< 3$, the velocity slows down and finally stops; a proper rescaling of the fluid velocity then gives rise to Darcy's law. For the critical size of holes, the holes are big enough to give some friction on the fluid, but not big enough to stop it, giving rise to Brinkman's law considered in the present paper.
The first result of homogenization of \emph{compressible} fluids for a \emph{critically} sized perforation was given by the first and last author in \cite{BellaOschmann2022}. In the latter the authors proved that the Brinkman term also appears for the limit of the stationary Navier-Stokes equation in the low Mach number regime. Our aim in here is to generalize their result in two directions: to the time-dependent setting, and also to more general domains in the sense of Mosco (domain) convergence.\\

Concerning results on stochastic homogenization, that is, randomly perforated domains, we recall the result of   Giunti, H\"ofer, and Vel\'azquez \cite{GiuntiHoeferVelazquez2018}, where the authors proved that the Brinkman term arises from the Poisson equation in a critically randomly perforated domain. Later, Giunti and H\"ofer \cite{GiuntiHoefer2019} derived a similar result for the stationary incompressible Stokes equations. Results for compressible fluids under the presence of randomly distributed, but \emph{tiny} holes, can be found in \cite{BellaOschmann2023}. $\Gamma$-convergence for nonlinear Dirichlet problems on randomly perforated domains was studied in \cite{Scardia-Zemas-Zeppieri}. Similar problems in the context of stochastic homogenization from a variational point of view were treated for example in \cite{BMZ, Marziani, Marziani-Solombrino}.

Besides the already mentioned references given above, there are many works dealing with homogenization of compressible as well as incompressible models, heat-conducting fluids, and transition in various ways. We refer the interested reader to the literature given in the references below.\\

Regarding incompressible models, Tartar gave the first result in \cite{Tartar1980}. He obtained a Darcy law, which was proposed more than 120 years earlier by Darcy \cite{Darcy1856}. Later, Mikeli\'c \cite{Mikelic1991} gave results for the evolutionary system, and obtained a Darcy law with memory effect. An incompressible system with non-constant density was investigated in \cite{BasaricOschmannPan2025, LuPanYang2025, Pan2025}, where the limiting systems are Darcy's law, the unperturbed Navier-Stokes equations, and Brinkman's law, respectively. Non-Newtonian fluids were considered in \cite{LuOschmann2024, LuQian2024}. Moreover, critical perforations in heat conducting fluids were investigated in \cite{FeireislNamlyeyevaNecasova2016}. The case of particles changing size when time continues was recently treated in \cite{WiedemannPeter2025}.\\

The literature for compressible fluids focuses mostly on the case of tiny (subcritical) holes, or the case when the holes' mutual distance is proportional to their size such that rescaling arguments apply. More precisely, Masmoudi considered in his seminal paper \cite{Masmoudi2002} the latter case and obtained a compressible Darcy's law; convergence rates for the same type of problem were found recently in \cite{HoeferNecasovaOschmann2024}. In \cite{HoeferKowalczykSchwarzacher2021} the authors obtained an incompressible Darcy's law, but starting from a compressible system with a low Mach number. The case of tiny holes were investigated in \cite{DieningFeireislLu2017, FeireislLu2015, NecasovaPan2022, OschmannPokorny2023} for different size of holes, values of the adiabatic exponent $\gamma$, and dimensions. Heat conducting compressible fluids were considered in \cite{BasaricChaudhuri2024, LuPokorny2021, PokornySkrisovsky2021} for a deterministic distribution of holes, and in \cite{Oschmann2022} for randomly distributed ones.\\

In addition, it is worth mentioning that also other models of homogenization have been studied such as \cite{KhrabustovskyiPlum2022}, who looked at homogenization from an operator theoretical point of view, and \cite{deAnnaSchlomerkemperZarnescu2025}, who studied nematic liquid crystals.  

\paragraph{Organization of the paper.} The problem under consideration and our main result Theorem~\ref{mainTheorem} can be found in Section~\ref{sec:setting}. The proof of our main result is then carried out in Section~\ref{sec:proofMT}. Finally, Section~\ref{sec:applic} is devoted to the application of Theorem~\ref{mainTheorem} to randomly perforated domains.

\paragraph{Notation.} Throughout the paper, we will make use of the following notations:
\begin{enumerate}[label=$(\alph*)$]
%\item$d$ is an integer $d\ge3$.
\item $\id$ denotes the identity matrix in $\RR^{3\times 3}$.
\item We let $\chi_E$ be the characteristic function of the set $E$.
\item $a\lesssim b$ denotes $a\le Cb$ for some constant $C>0$ which is independent of $a, b,$ and $\eps$.
    \item We distinguish vector quantities from scalar ones by using bold symbols, i.e., $a$ is scalar whereas $\bf a$ is a vector.
    \item ${\bf a}\otimes \bf b$ and ${\bf a}\cdot \bf b$ denote the tensor product and scalar product between the vectors $\bf a$ and $\bf b$, respectively. $A:B = \sum_{i,j=1}^d A_{ij} B_{ij}$ denotes the scalar (Frobenius inner) product between the matrices $A$ and $B$.
    \item We denote with $\varphi$, $\pp$ respectively $\phi$, $\pphi$ test functions depending on the pair $(t,x)$ and only on $x$, where $t\in (0,T)$ is the time variable and $x\in \RR^3$ the space variable.
    \item $\ee_k$ for $k\in\{1,2,3\}$ denotes the $k$-th element of the canonical basis in $\RR^3$.
    \item Lebesgue and Sobolev spaces will be denoted as usual by $L^p$ and $W^{k,p}$, respectively.
    \item $C_{\rm weak}(0,T;L^2(\Omega;\RR^3))$ denotes the space of functions $\uu(t,\cdot)\in L^2(\Omega;\RR^3)$ for all $t\in[0,T]$ such that 
    \begin{equation*}
t\mapsto\int_\Omega\uu(t,x)\cdot\pphi(x)\dx
    \end{equation*}
    is continuous for all $\pphi\in C^\infty_c(\Omega;\RR^3)$;
    \item We denote by $\div$ and $\nabla_x$ the divergence and gradient in the variable $x$ respectively. We denote with $\partial_t$ the partial derivative with respect to $t$.
    \item $\RR^{3 \times 3}_{\rm sym}$ is the space of $3 \times 3$ symmetric real matrices.
\item $\cM^+(\Omega;\RR^{3 \times 3}_{\rm sym})$ denotes the space of finite matrix valued non-negative measures on $\Omega$ ranging to $\RR_{\rm sym}^{3 \times 3}$.
\item $\mathcal{H}^{2}$ denotes the $2$-dimensional Hausdorff measure, and $\mathcal{H}^{2}\res E$ the $2$-dimensional Hausdorff measure restricted to the set $E$.
\end{enumerate}

%%%%%%%%%%%%%%%%%%%%%%%%%%%%%%%%%%%%%%%%%%%%%%%%%%%%%%%%%%%%%%%%%%%%%%%%%
\section{Setting of the problem and main result}\label{sec:setting}
In this section we describe the setting of our problem and state our main result. To start, for $\eps>0$ we consider a family of domains $(\Omega_\eps)_{\eps>0}$ confined to a bounded subset $D\subset \RR^3$. We let also $\Omega \subset D$ be a smoothly bounded domain. On $(0,T)\times\Omega_\eps $ (with $T>0$) we consider the system of equations in \eqref{NavierSystem} coupled with no-slip boundary conditions \eqref{no-slip}, whereas on $(0,T)\times\Omega$ we consider the system of equations \eqref{systemstarkeloesung}.

\subsection{Suitable solution concepts}
We will give some concepts of weak, dissipative, and strong solutions to the systems \eqref{NavierSystem} and \eqref{systemstarkeloesung}, respectively, suitable for our purposes.

 \begin{defin}[Finite energy weak solution to \eqref{NavierSystem}]\label{def:FEWS}
 	Let $\gamma\ge1$. We say that $(\rho_\eps,\uu_\eps)$ is a \textit{finite energy weak solution} to system \eqref{NavierSystem} with initial data
 	\begin{align*}
 	\rho_\eps(0,\cdot)=\rho_{0,\eps}  \in L^\gamma(\Omega_\eps), \ \rho_{0, \eps} \geq 0, &&	(\rho_\eps \uu_\eps)(0,\cdot)=(\rho\uu)_{0,\eps} \in L^\frac{2\gamma}{\gamma+1}(\Omega_\eps; \RR^3),
 	\end{align*}
 	if the following hold:
 	\begin{enumerate}[label=(\roman*)]
 		\item \textit{Integrability.} We have
 		\begin{equation}\label{integrability}
        \begin{gathered}
 			\rho_\eps\in L^\infty (0,T;L^\gamma(\Omega_\eps)),\qquad	\rho_\eps\geq 0, \qquad \int_{\Omega_\eps} \rho_\eps(\tau, \cdot) \dx = \int_{\Omega_\eps} \rho_{0, \eps} \dx \ \text{ for almost any } \ \tau \in [0, T],\\
 		\uu_\eps\in L^2(0,T;W_0^{1,2}(\Omega_\eps;\RR^3)), \qquad \rho_\eps\uu_\eps\in L^\infty(0,T;L^{\frac{2\gamma}{\gamma+1}}(\Omega_\eps;\RR^3)).
        \end{gathered}
 		\end{equation}
 		\item \textit{Equation of continuity.} It holds
 		\begin{align}\label{continuity}
 			\int_0^T\int_{\Omega_\eps} \rho_\eps\partial_t \varphi + \rho_\eps\uu_\eps\cdot\Grad\varphi \dx\dt=-\int_{\Omega_\eps}\rho_{0,\eps} \varphi(0,\cdot)\dx,
 		\end{align}
 		for any $\varphi\in C_c^1([0,T)\times {\overline{\Omega}_\eps})$.
 		\item \textit{Momentum equation.} It holds
 		\begin{equation}\label{momentum}
 			\begin{split}
 		\int_0^{\tau}&\int_{\Omega_\eps} \rho_\eps \uu_\eps \cdot \partial_t \pp + \rho_\eps \uu_\eps \otimes \uu_\eps : \Grad \pp + \frac{1}{\Ma^2(\eps)}p(\rho_\eps)\div\pp \dx\dt\\
 	&=\int_0^{\tau}\int_{\Omega_\eps} \SS(\Grad\uu_\eps):\Grad\pp - \rho_\eps \ff \cdot \pp \dx \dt
 	{+\int_{\Omega_\eps} (\rho_\eps\uu_\eps) (\tau,\cdot) \cdot \pp(\tau,\cdot) \dx}
 	-\int_{\Omega_\eps}(\rho\uu)_{0,\eps}\cdot\pp(0,\cdot)\dx,
 			\end{split}
 		\end{equation}
 		for any $\tau\in [0,T]$ and any $\pp\in C^1_c([0,T)\times \Omega_\eps;\RR^3)$.
 		\item \textit{Energy inequality.} For any $\tau \in [0,T]$, there holds
 		\begin{equation}\label{energy}
 			\begin{split}
 				\int_{\Omega_\eps}&\left[ \frac{1}{2}\id_{\{\rho_\eps>0\}}\frac{|\rho_\eps\uu_\eps|^2}{\rho_\eps}+\frac{1}{\Ma^2(\eps)}P(\rho_\eps) \right](\tau,\cdot)\dx+\int_0^\tau\int_{\Omega_\eps}\SS(\Grad\uu_\eps):\Grad\uu_\eps\dx\dt\\
 			&\leq \int_{\Omega_\eps} \frac{1}{2} \id_{\{\rho_{0, \eps}>0\}} \frac{|(\rho\uu)_{0,\eps}|^2}{\rho_{0,\eps}}+\frac{1}{\Ma^2(\eps)}P(\rho_{0,\eps}) \dx+\int_0^\tau \int_{\Omega_\eps}\rho_\eps \ff\cdot \uu_\eps\dx\dt,
 			\end{split}
 		\end{equation}
 		where 
 		\begin{equation}\label{pressPot}
 		P'(\rho)\rho-P(\rho)=p(\rho), \quad P(1)=0.
 		\end{equation}
 	\end{enumerate}
 \end{defin}
 
 \begin{rmk}
     We remark that from \eqref{pressPot}, we may write $P(\rho)$ in compact form as
     \begin{align*}
         P(\rho) = \rho \int_1^\rho \frac{p(s)}{s^2} \ds.
     \end{align*}
 \end{rmk}
 \begin{defin}[Dissipative solution to \eqref{systemstarkeloesung}]\label{def:DissSol}
 	Given $\bar{\rho}>0$, we say that $$\uu\in C_{\rm weak}([0,T];L^2(\Omega;\RR^3))\cap L^2(0,T; W_0^{1,2}(\Omega; \RR^3))$$ is a \emph{dissipative solution} to \eqref{systemstarkeloesung} with initial datum
 	\begin{equation*}
 		\uu(0,\cdot)=\uu_0 \in L^2(\Omega), \quad \div \uu_0 = 0,
 	\end{equation*}
 	if the following hold:
 	\begin{enumerate}
 		\item \textit{Incompressibility.}  
 		\begin{equation}\label{incompressibility-diss}
 			\div\uu=0\quad \text{ a.e. in } (0,T)\times\Omega.
 		\end{equation}
 		\item \textit{Momentum equation.} There exists $\fR\in L^\infty(0,T;\cM^+(\Omega;\RR^{3 \times 3}_{\rm sym}))$ such that
 			\begin{equation}	\label{momentum-diss}
 				\begin{split}
 					& \left[ \int_\Omega\bar{\rho}\uu\cdot\pp \dx \right]_{t=0}^{t=\tau} = \int_0^{\tau}\int_\Omega \bar{\rho}\uu\cdot \partial_t\pp \dx \dt - \mu \int_0^{\tau}\int_\Omega( \bM \uu)\cdot\pp\dx\dt \\
 					&+ \int_0^{\tau}\int_\Omega (\bar{\rho} \uu\otimes\uu):\Grad\pp - \mu\Grad\uu:\Grad\pp + \bar{\rho} \ff\cdot\pp + \fR:\Grad\pp \dx\dt,
 				\end{split}
 			\end{equation}
 				for almost any $\tau\in[0,T]$ and any $\pp\in C^1_c([0,T]\times\Omega;\RR^3)$ with $\div\pp=0$.
 				\item \textit{Energy inequality.}
 				For $\fR$ as above it holds\begin{equation}\label{energy-diss}
 					\begin{split}
 						&\int_\Omega \frac{1}{2}{\bar{\rho}|\uu|^2} (\tau,\cdot)\dx + \int_\Omega \frac12 {\rm trace}[\fR](\tau,\cdot) \dx +\mu\int_0^{\tau}\int_\Omega |\Grad\uu|^2\dx\dt + \mu \int_0^\tau\int_\Omega( \bM\uu)\cdot\uu\dx\dt \\
 						&\leq \int_{\Omega}\frac{1}{2}\bar{\rho} |\uu_0|^2\dx+\int_0^\tau \int_{\Omega}\bar{\rho}\ff\cdot \uu\dx\dt,	\end{split}
 				\end{equation}
 				for almost any $\tau\in [0,T]$.
 			\end{enumerate}
 			We say that $(\bar{\rho},\uu)$ is a dissipative solution  if $\uu$ is a dissipative solution with respect to $\bar{\rho}$.
 		\end{defin}
    \begin{rmk} 
Note that in 
\eqref{momentum-diss}, \eqref{energy-diss} 
and in the following there is a little abuse of notation. Indeed we write $\fR:\nabla_x\pp\dx $ and ${\rm trace}[\fR](\tau,\cdot)\dx$, however in general $\fR$ is not absolutely continuous with respect to Lebesgue measure.
        
    \end{rmk}
 \begin{defin}[Strong solution to \eqref{systemstarkeloesung}]\label{def:StrongSol}
 	Let $\bar{\rho}>0$ be a given constant. We say that $\uu$ is a \textit{strong solution} to \eqref{systemstarkeloesung} with initial datum
 	\begin{equation*}
 		\uu(0,\cdot)=\uu_0\in \bigcap_{1 \leq p < \infty} W^{1,p}(\Omega;\RR^3), \quad \div \uu_0 = 0,
 	\end{equation*}
 provided 
 \begin{equation*}
 	\uu\in C(0,T;W^{2,p}(\Omega; \RR^3))\cap C^1(0,T;L^q(\Omega; \RR^3)) \text{ for all } p,q\in [1,\infty) ,
 \end{equation*}
and $\uu$ satisfies \eqref{systemstarkeloesung} pointwise. We say that the pair $(\bar{\rho},\uu)$ is a strong solution if $\uu$ is a strong solution with respect to $\bar{\rho}$.
 \end{defin}

\begin{defin}[Well-prepared initial data] We say that the initial data $(\rho_{0,\eps}, (\rho \uu)_{0, \eps})$ from Definition~\ref{def:FEWS} are \emph{well-prepared} if there is a constant $\bar{\rho}>0$ and a velocity field
 	\begin{align}\label{u0}
 		\uu_0\in \bigcap_{1 \leq p < \infty} W^{1,p}(\Omega ;\RR^3), && \div\uu_0=0 \ \text{ in } \Omega, && \uu_0=0 \ \text{ on }\partial\Omega,
 	\end{align}
 	
 	 such that
\begin{align}\label{ID1}
\lim_{\eps\to 0}	\frac{1}{\Ma^2(\eps)}\int_{\Omega_\eps}\left[ P(\rho_{0,\eps})-P'(\bar{\rho})(\rho_{0,\eps} -\bar{\rho})-P(\bar{\rho}) \right]\dx= 0,
\end{align}
and 
\begin{equation}\label{ID2}
	(\rho\uu)_{0,\eps}\to \bar{\rho} \uu_0\quad\text{in }L^\frac{2\gamma}{\gamma+1}(\Omega;\RR^3),
\end{equation}
\begin{equation}\label{ID3}
\lim_{\eps\to 0}\int_{\Omega_\eps} \id_{\{\rho_{0,\eps} > 0\}} \frac{ |(\rho \uu)_{0,\eps}|^2}{\rho_{0,\eps}}\dx= \int_\Omega \bar{\rho} |\uu_0|^2\dx.
\end{equation}
 \end{defin}

 \subsection{Admissible domains and the main result}
 To generalize the class of domains given in \eqref{firstDom}, we now introduce the class of admissible domains $(\Omega_{\eps})_{\eps>0}$.
 \begin{defin}[Assumptions on $(\Omega_{\eps})_{\eps>0}$]\label{def:omega_eps}  Let $\eps > 0$ and
$((\Omega_{\eps})_{\eps>0},\Omega,D)$ be a family of bounded domains 
$\Omega_\eps,\Omega\subset D\subset\RR^3$ with $\Omega$ of class $C^{2+\nu}$ for some $\nu>0$.
 \begin{enumerate}[label=(M\arabic*)]
 \item\label{M0} We say that $((\Omega_{\eps})_{\eps>0},\Omega,D)$ satisfies \ref{M0} if for any ${\pphi}\in W_0^{1,2}(D;\RR^3)$ and ${\pphi_\eps}\in W_0^{1,2}(\Omega_\eps;\RR^3)\subset W_0^{1,2}(D;\RR^3)$ such that $ {\pphi_\eps}\rightharpoonup{\pphi}$ weakly in $W_0^{1,2}(D;\RR^3)$, there holds ${\pphi}\in W_0^{1,2}(\Omega;\RR^3)$.
 	\item\label{M1} We say that $((\Omega_{\eps})_{\eps>0},\Omega,D)$ satisfies \ref{M1} if   for any ${\pphi}\in C_c^\infty (\Omega;\RR^3)$ with $\div\pphi=0$ there exists a sequence of solenoidal functions $({\pphi_\eps})_{\eps >0}$ such that
\begin{equation}\label{M1-i}
{\pphi_\eps}\in W_0^{1,2}(\Omega_\eps,\RR^3),\quad \div{\pphi_\eps}=0\quad \text{ in }\Omega_\eps;
\end{equation}
\begin{equation}\label{M1-ii}
{\pphi_\eps}\rightharpoonup{\pphi} \quad\text{ weakly in }W_0^{1,2}(D;\RR^3); %\quad \text{ and }\quad {\pphi_\eps}\to{\pphi} \quad\text{ strongly in }W_0^{1,p}(D;\RR^3)\text{ for all } 1\leq p <2;
\end{equation}
\begin{equation}\label{M1-iii}
{\Grad \pphi_\eps} \to {\Grad \pphi} \quad\text{ strongly in }L^\frac32(D;\RR^3);
%{\pphi_\eps}\to{\pphi}\quad\text{ strongly in }L^p(D;\RR^3)\text{ for all }1\leq p< \infty;
\end{equation}
\begin{equation}\label{M1-iv}
\limsup_{\eps\to 0} \| \Grad ({\pphi_\eps}-{\pphi})\|_{L^r(D; \RR^3)}\Ma(\eps)^{\frac{2}{\gamma}}=0\quad \text{ with } r:=\frac{3\gamma}{2\gamma - 3}.
\end{equation}
In addition, there exists a matrix $\bM\in \RR_{\rm sym}^{3 \times 3}$  with the following properties:

Let $\vv\in L^2(0,T;W_0^{1,2}(\Omega;\RR^3))$ and let $\vv_\eps\in L^2(0,T;W_0^{1,2}(\Omega_\eps;\RR^3))$ satisfy
\begin{equation}\label{v1}
\vv_\eps\rightharpoonup\vv\quad\text{ weakly in } L^2(0,T;W^{1,2}(D;\RR^3)),
\end{equation}
\begin{equation}\label{v2}
	\div\vv_\eps=\partial_t g_\eps+\div\hh_\eps,
\end{equation}
for some $g_\eps\in L^\infty(0,T;L^\gamma(\Omega_\eps))$, $\hh_\eps\in L^2(0,T;L^{\frac{6\gamma}{\gamma+6}}(\Omega_{\eps};\RR^3))$ with the following properties:   $g_\eps\to\bar g$ in $L^\infty(0,T;L^\gamma(D))$ for some $\bar{g} \in \RR$, and $\|\hh_\eps\|_{L^2(0,T;L^{\frac{6\gamma}{\gamma+6}}(\Omega_\eps;\RR^3))}\lesssim {\rm Ma}(\eps)^{\min\{\frac{2}{\gamma},1\}}$. Assume further that there exist sets $M_\eps\subset (0,T)\times \Omega_\eps$ such that
\begin{align}\label{v3}
    \|g_\eps\chi_{M_\eps}\|_{L^\infty(0,T; L^\infty(D))}&\lesssim 1,\\\label{v4}
    \| g_\eps (1 - \chi_{M_\eps}) \|_{L^\infty(0,T;L^\gamma(D))} &\lesssim {\rm Ma} (\eps)^{\frac{2}{\gamma}}.
\end{align}
Then there holds for a.e. $\tau\in[0,T]$
\begin{align}
\lim_{\eps\to 0}\int_0^\tau\int_{\Omega_\eps}\Grad{\pphi_\eps}:\Grad\vv_\eps\dx\dt &= \int_0^\tau\int_\Omega\Grad {\pphi} :\Grad\vv\dx\dt+\int_0^\tau\int_\Omega (\bM{\pphi})\cdot\vv\dx\dt, \label{M1-v}\\
	 \liminf_{\eps\to 0} \int_0^\tau\int_{\Omega_\eps} |\Grad \vv_\eps |^2\dx\dt &\ge \int_0^\tau\int_{\Omega} |\Grad \vv|^2 \dx\dt+\int_0^\tau\int_\Omega (\bM\vv)\cdot\vv \dx\dt. \label{M1-vi}
\end{align}
 \end{enumerate}
 If $((\Omega_\varepsilon)_{\varepsilon>0},\Omega,D)$ satisfies \ref{M0}, \ref{M1} we say that $\Omega_\varepsilon$ converge to $\Omega$ in the sense of Mosco.\end{defin}

%
%\begin{rmk}\label{remark-div(v)-div(v_eps)}
%	If   $\lim_{\eps\to 0}\Ma(\eps)= 0$ then necessarily $\div (\vv)=0$. Indeed, for $\psi\in C_c^\infty((0,T)\times D)$ we have using \eqref{v1} and \eqref{v2} \cb{better: assume \eqref{v2} and show it here as this is specific just for our function $\uu_\eps$; also, notice that the convergences do not happen in $\Omega$, but in $D$!}
%	\begin{align*}
%		\int_0^T\int_{\Omega} \div (\vv) \psi \dx\dt 
%		&=\lim_{\eps\to 0} \int_0^T\int_{\Omega_\eps} \div (\vv_\eps) \psi \dx \dt \\
%		&=\lim_{\eps\to 0} \int_0^T\int_{\Omega_\eps} (\partial_t g_\eps) \psi \dx \dt+\lim_{\eps\to 0}\int_0^T\int_{\Omega_\eps}\div(\hh_\eps) \psi \dx \dt\\
%		&=-\lim_{\eps\to 0} \int_0^T\int_{\Omega_\eps} g_\eps \partial_t\psi \dx \dt-\lim_{\eps\to 0}\int_0^T\int_{\Omega_\eps}\hh_\eps \cdot\Grad\psi \dx \dt.
%	\end{align*}
%Thus
%	\begin{align*}
%		&\lim_{\eps\to 0} \int_0^T\int_{\Omega_\eps} g_\eps \partial_t\psi \dx \dt
%		=\int_0^T\int_{\Omega} \bar g \partial_t\psi \dx \dt
%		=-\int_0^T\int_{\Omega} (\partial_t\bar g) \psi \dx \dt
%		=0,\\[1em]
%		&\lim_{\eps\to 0}\left|\int_0^T\int_{\Omega_\eps}\hh_\eps \cdot\Grad\psi \dx \dt\right|\lesssim \lim_{\eps\to 0}\| \hh_\eps\|_{L^2(0,T;L^{\frac{6\gamma}{\gamma+6}}(\Omega_\eps;\RR^3))}\lesssim \lim_{\eps\to 0}\Ma (\eps)^{\{ \frac{2}{\gamma},1 \}}=0.
%	\end{align*}    
    %Thus we have $\div(\vv)=0$. This together with \eqref{v1} also imply
	%\begin{align*}
	%	\div(\vv_\eps)\rightharpoonup 0 \quad\text{weakly in }L^2(0,T;L^2(D;\RR)),
	%\end{align*}
%and hence $\|	\div(\vv_\eps)\|_{L^2(0,T;L^2(D))}\lesssim 1$.
%\end{rmk}
In Section \ref{sec:applic} we will consider an example of $((\Omega_\eps)_{\eps>0},\Omega,D)$ where $\Omega_\eps$ is obtained by removing random holes from $\Omega$ and $D=\Omega$, hence we will write $((\Omega_\eps)_{\eps>0},\Omega)$ in place of $((\Omega_\eps)_{\eps>0},\Omega,D)$.
We are now in the position to state our main result.
\begin{thm}\label{mainTheorem}
	Let $((\Omega_\eps)_{\eps>0},\Omega,D)$ be a family of bounded domains as in Definition~\ref{def:omega_eps}. Let $(\rho_\eps,\uu_\eps)_{\eps>0}$ be a family of finite energy weak solutions in the sense of Definition~\ref{def:FEWS} to the Navier-Stokes system \eqref{NavierSystem} under assumptions \eqref{assumptions0} in $(0,T)\times\Omega_{\eps}$, emanating from the well-prepared initial data $(\bar{\rho}, \uu_0)$ satisfying \eqref{u0}--\eqref{ID3} for some $\bar{\rho}>0$. Finally, let 
	\begin{align}\label{assumptions1}
		\gamma >\frac{3}{2}, &&
	\lim_{\eps\to 0}	\Ma(\eps)= 0.
	\end{align}
    Then the following holds.
    \begin{enumerate}[label=$(\alph*)$]
        \item\label{a} Assume $((\Omega_\eps)_{\eps>0},\Omega,D)$ obeys \ref{M0}. Then
	\begin{align*}
		\rho_\eps\to\bar{\rho} &\text{ strongly in }L^\infty(0,T;L^\gamma(D)),\\
		\uu_\eps\rightharpoonup\uu&\text{ weakly in } L^2(0,T;W_0^{1,2}(D,\RR^3)),
	\end{align*}
with $\uu(0,\cdot)=\uu_0$. In particular, $\uu\in L^2(0,T;W_0^{1,2}(\Omega;\RR^3))$. 
\item \label{b} Assume in addition that $((\Omega_\eps)_{\eps>0},\Omega,D)$ satisfies also \ref{M1}. Then $(\bar{\rho},\uu)$ is a dissipative solution to \eqref{systemstarkeloesung} in the sense of Definition~\ref{def:DissSol}.
\item \label{c} Assume that there exists a strong solution to the system \eqref{systemstarkeloesung} in the sense of Definition~\ref{def:StrongSol} with initial datum $\uu_0$. Then the latter coincides to $(\bar{\rho},\uu)$ and $\uu_\eps\to \uu$ strongly in $L^2((0,T)\times D;\RR^3)$.
    \end{enumerate}
\end{thm}

%%%%%%%%%%%%%%%%%%%%%%%%%%%%%%%%%%%%%%%%%%%%%%%%%%%%%%%%%%%%%%
%%%%%%%%%%%%%%%%%%%%%%%%%%%%%%%%%%%%%%%%%%%%%%%%%%%%%%%%%%%%%%

\section{Proof of Theorem \ref{mainTheorem}}\label{sec:proofMT}

In this section we provide the proof of Theorem~\ref{mainTheorem}. For simplicity we divide the proof into two steps contained in Propositions \ref{prop1} and \ref{prop2} below. More precisely, in Proposition \ref{prop1} we show \ref{a} and that finite energy weak solutions to \eqref{NavierSystem} converge to dissipative solutions of \eqref{systemstarkeloesung}, from which we deduce \ref{b}, while in Proposition \ref{prop2} we show the validity of the  weak-strong uniqueness principle which in turn implies \ref{c}.
\begin{prop}[Convergence to dissipative solutions of \eqref{systemstarkeloesung}]\label{prop1}
		Let $((\Omega_\eps)_{\eps>0}, \Omega, D)$ be a family of bounded domains that satisfies \ref{M0}. Let also $(\rho_\eps,\uu_\eps)_{\eps>0}$ be a family of finite energy weak solutions to the Navier-Stokes system \eqref{NavierSystem} in $(0,T)\times\Omega_{\eps}$, which satisfies \eqref{u0}--\eqref{ID3} with \eqref{assumptions0} and \eqref{assumptions1}.
	Then 
	\begin{align}\label{konv1}
		\rho_\eps\to\bar{\rho} &\text{ strongly in }L^\infty(0,T;L^\gamma(D)),\\\label{konv3}
		\uu_\eps\rightharpoonup\uu&\text{ weakly in } L^2(0,T;W_0^{1,2}(D,\RR^3)).
	\end{align}
	If in addition \ref{M1} holds, then $(\bar{\rho},\uu)$ is a dissipative solution to \eqref{systemstarkeloesung}.
\end{prop}
\begin{proof}
	The validity of \eqref{konv1} and \eqref{konv3} has already been proven in \cite[Theorem 1.3]{BellaFeireislOschmann2023b}. Note that their assumption \ref{M0} is actually exactly the same as ours such that their proof remains valid. Additionally, the energy inequality enforces (see \cite[Section~2]{BellaFeireislOschmann2023b})
    \begin{align}\label{dvojka}
        \|\sqrt{\rho_\eps} \uu_\eps\|_{L^\infty(0,T; L^2(D; \RR^3))} \lesssim 1,
    \end{align}
    as well as
    \begin{align}\label{petka}
        \sup_{\tau \in [0, T]} \int_D \frac{1}{\Ma^2(\eps)} (P(\rho_\eps) - P'(\bar{\rho})(\rho_\eps - \bar{\rho}) - P(\bar{\rho}) ) \dx \lesssim 1.
    \end{align}
Assume also \ref{M1} holds.    Thus, our task is to show that $(\bar{\rho}, \uu)$ is a dissipative solution to \eqref{systemstarkeloesung}. 
	
\item\paragraph{Step 1: Incompressibility.}
%\red Let $\pp\in C_c^1([0,T)\times\Omega)$ be of the form \cb{Do we need $\div \pp = 0$?} \textcolor{cyan}{Nein und generell ist das ein komplett anderer beweis als ich in meiner MA hatte}
%\begin{equation}\label{test}
%\pp(t,x)=\psi(t)\pphi(x), \quad\text{ with }\quad \psi\in C_c^\infty([0,T)),~\pphi\in C_c^\infty(\Omega;\RR^3).
%\end{equation}
%We then choose $\pphi_\eps \in C_c^\infty(\Omega_\eps;\RR^3)$ such that $\pphi_\eps\rightharpoonup\pphi$ in $W_0^{1,2}(D;\RR^3)$. \cb{Why is this possible? Not guaranteed by \ref{M0}.} By testing \eqref{continuity} with $\pp_\eps(t,x):=\psi(t)\pphi_\eps(x)\in C_c^1([0,T)\times\Omega_\eps)$, and \black
% by letting $\eps\to 0$, we have \cb{rephrase a bit with actual limits etc.}
%	\begin{align*}
	%	0&=\int_0^T\int_\Omega \bar{\rho} \partial_t\pp\dx\dt+\int_0^T\int_\Omega\bar{\rho} \uu\cdot \Grad\pp\dx\dt+\int_\Omega \bar{\rho} \pp(0,\cdot)\dx\\
	%	&=\int_0^T\int_\Omega ( \partial_t\bar{\rho})\pp\dx\dt-\int_\Omega \bar{\rho} \pp(0,\cdot)\dx+\int_0^T\int_\Omega\bar{\rho} \uu\cdot \Grad\pp\dx\dt+\int_\Omega \bar{\rho} \pp(0,\cdot)\dx\\
	%	&=\int_0^T\int_\Omega\bar{\rho} \uu\cdot \Grad\pp\dx\dt.
	%\end{align*}
%Since the class of functions of type \eqref{test} are dense in $C_c^1([0,T)\times\Omega)$, the above identity holds for all $\pp\in C_c^1([0,T)\times\Omega)$, and hence
%	\begin{align*}
%		\div\uu=0 \quad\text{ a.e. in } (0,T)\times\Omega.
%	\end{align*}

 Let $\varphi\in C_c^1([0,T)\times\overline{\Omega})$. 
Using \eqref{continuity}, we can show as in \cite[Proposition~3.3]{LuSchwarzacher2018} that 
\begin{align}\label{cont2}
 			\int_0^T\int_{D} \rho_\eps\partial_t\varphi+\rho_\eps\uu_\eps\cdot\Grad\varphi \dx\dt=-\int_{D}\rho_{0,\eps}\varphi(0,\cdot)\dx,
\end{align}
where we extended $\rho_\eps$ and $\uu_\eps$ by $0$ to the whole of $D$.

With \eqref{konv1}, \eqref{konv3}, and \eqref{dvojka} we get
\begin{align}\label{trojka}
    \rho_\eps\uu_\eps\xrightharpoonup{\ast} \bar{\rho}\uu ~\text{weakly}^\ast \text{ in }L^\infty (0,T;L^{\frac{2\gamma}{\gamma+1}}(D;\RR^3)).
\end{align}
Indeed, for $\pphi \in L^1(0,T; L^\frac{2\gamma}{\gamma-1}(D; \RR^3))$,
\begin{align*}
    \int_0^T \int_D (\rho_\eps \uu_\eps^\delta - \bar{\rho} \uu) \cdot \pphi \dx \dt &= \int_0^T \int_D ( \sqrt{\rho_\eps} - \sqrt{\bar{\rho}}) \sqrt{\rho_\eps} \uu_\eps^\delta \cdot \pphi \dx \dt + \int_0^T \int_D \sqrt{\bar{\rho}} (\sqrt{\rho_\eps} - \sqrt{\bar{\rho}} ) \uu_\eps^\delta \cdot \pphi \dx \dt \\
    &\quad + \int_0^T \int_D \bar{\rho} (\uu_\eps^\delta - \uu) \cdot \pphi \dx \dt,
\end{align*}
where the superscript $\delta$ denotes a convolution in time only to ensure $\uu_\eps^\delta \in L^\infty(- \delta, T+\delta; W_0^{1,2}(D; \RR^3))$ such that all integrals above are well defined. Due to the strong convergence $\rho_\eps \to \bar{\rho}$ in $L^\infty(0,T; L^\gamma(D))$, the weak convergence $\uu_\eps \rightharpoonup \uu$ in $L^2(0,T; W_0^{1,2}(D; \RR^3))$ that ensures as well $\uu_\eps^\delta \rightharpoonup \uu^\delta$ weakly in $W_0^{1,2}(D; \RR^3)$ uniformly in time, and the bound \eqref{dvojka}, we infer
\begin{align*}
    \lim_{\delta \to 0} \lim_{\eps \to 0} \int_0^T \int_D (\rho_\eps \uu_\eps^\delta - \bar{\rho} \uu) \cdot \pphi \dx \dt = 0.
\end{align*}

Using \eqref{trojka} and \eqref{konv1}, we have for $\eps\to 0$ in (\ref{cont2}) 
\begin{align*}
    0&=\int_0^T\int_D \bar{\rho} \partial_t\varphi \dx\dt+\int_0^T\int_D \bar{\rho} \uu\cdot \Grad\varphi
     \dx\dt + \int_D \bar{\rho} \varphi(0,\cdot) \dx.
\end{align*}
 Now, we conclude with partial integration and $\bar{\rho}$ being a positive constant that
\begin{align*}
    0&=-\int_0^T\int_D ( \partial_t\bar{\rho})\varphi \dx\dt-\int_D \bar{\rho} \varphi(0,\cdot) \dx+\int_0^T\int_D\bar{\rho} \uu\cdot \Grad\varphi \dx\dt+\int_D \bar{\rho} \varphi(0,\cdot) \dx \\
    &=\int_0^T\int_D\bar{\rho} \uu\cdot \Grad\varphi \dx\dt,
\end{align*}	
showing $\div \uu = 0$.\\
	
\item\paragraph{Step 2: Momentum Equation.} 

We consider a specific ansatz for the test functions, namely
\begin{align*}
    \psi(t)\pphi(x), \quad \psi\in C_c^1([0,T)), \quad \pphi\in C_c^1(\Omega;\RR^3), \quad \div(\pphi)=0.
\end{align*}
Due to \ref{M1} there exists a sequence of solenoidal functions $(\pphi_\eps)_{\eps>0}$ approximating $\pphi$ such that \eqref{M1-i}--\eqref{M1-vi} hold. We  use that $\pphi_\eps\psi\in C_c^1([0,T)\times \Omega_\eps;\RR^3)$ is a good test function to system \eqref{NavierSystem} to get
\begin{align*}
    0 &= \int_0^\tau\int_{\Omega_\eps}\rho_\eps \uu_\eps\cdot \partial_t (\pphi_\eps \psi) \dx\dt + 
    \int_0^\tau\int_{\Omega_\eps}(\rho_\eps \uu_\eps\otimes \uu_\eps):\Grad (\pphi_\eps \psi) \dx\dt\nonumber\\
    &\qquad -\int_0^\tau\int_{\Omega_\eps}\SS(\Grad  \uu_\eps):\Grad (\pphi_\eps \psi) \dx \dt +\int_0^\tau\int_{\Omega_\eps}\rho_\eps \ff \cdot (\pphi_\eps \psi) \dx\dt\nonumber\\
    &\qquad +\int_{\Omega_\eps}(\rho \uu)_{0,\eps}\cdot (\pphi_\eps \psi)(0, \cdot) \dx -\int_{\Omega_\eps} (\rho_\eps \uu_\eps)(\tau,\cdot)\cdot (\pphi_\eps \psi)(\tau, \cdot) \dx\,.
\end{align*}
Now using that $\pphi_\eps=(\pphi_\eps-\pphi)+\pphi$ the above identity can be rewritten as
\begin{align}\label{basisGL}
   &\int_0^\tau\int_{\Omega_\eps}\rho_\eps  \uu_\eps\cdot\partial_t(\pphi\psi)
    +(\rho_\eps \uu_\eps\otimes  \uu_\eps):\Grad(\pphi\psi)
    -\SS(\Grad  \uu_\eps):\Grad(\pphi\psi)+\rho_\eps \ff\cdot(\pphi\psi)  \dx\dt\nonumber\\
    &\qquad +\int_{\Omega_\eps}(\rho \uu)_{0,\eps}\cdot \pphi \psi(0) \dx -\int_{\Omega_\eps}(\rho_\eps \uu_\eps)(\tau,\cdot)\cdot \pphi \psi(\tau) \dx\nonumber\\
    &=\int_0^\tau\int_{\Omega_\eps}\rho_\eps \uu_\eps\cdot (\pphi-\pphi_\eps) \partial_t \psi \dx\dt + 
    \int_0^\tau\int_{\Omega_\eps}(\rho_\eps \uu_\eps\otimes \uu_\eps):\Grad(\pphi-\pphi_\eps) \psi \dx\dt\nonumber\\
    &\qquad -\int_0^\tau\int_{\Omega_\eps}\SS(\Grad  \uu_\eps):\Grad (\pphi-\pphi_\eps)\psi \dx \dt +\int_0^\tau\int_{\Omega_\eps}\rho_\eps \ff \cdot (\pphi-\pphi_\eps)\psi  \dx\dt\nonumber\\
    &\qquad +\int_{\Omega_\eps}(\rho \uu)_{0,\eps}\cdot (\pphi-\pphi_\eps)\psi(0) \dx -\int_{\Omega_\eps} (\rho_\eps \uu_\eps)(\tau,\cdot)\cdot (\pphi-\pphi_\eps) \psi(\tau) \dx
    =:\sum_{i=1}^6I_i.\nonumber
\end{align} 
 
 %\red Let $\pphi\in C_c^1([0,T)\times\Omega)$ satisfy \eqref{test} and such that $\div\pphi=0$. \black
	%Due to \ref{M0} there exists a sequence of cut-off functions $(\pphi_\eps)_{\eps>0}$ approximating $\pphi$ such that \eqref{M1-i}--\eqref{M1-vi} hold.
	
	%By splitting $\pphi=\pphi_\eps+(\pphi-\pphi_\eps)$ and using that $\pp_\eps:=\psi\pphi_\eps$ is a test function for \eqref{momentum} with $\div\pp_\eps=0$ we get
	%\begin{equation}\label{basisGL}
		%\begin{split}
	%\int_0^{\red \tau}\int_{\Omega_\eps}&\Big(\rho_\eps  \uu_\eps\cdot\partial_t\pp\black
%+\rho_\eps \uu_\eps\otimes  \uu_\eps:\Grad{\pp}
%-\SS(\Grad  \uu_\eps):\Grad{\pp}+\rho_\eps \ff \cdot{\pp} \Big) \dx\dt\\
%&+\int_{\Omega_\eps}(\rho u)_{0,\eps}\cdot {\pp}(0,\cdot)\dx-	{\red 	\int_{\Omega_\eps}(\rho_\eps\uu_\eps)(\tau,\cdot)\cdot\pp(\tau,\cdot)\dx}\\
%&=\int_0^{\red\tau}\int_{\Omega_\eps}\rho_\eps  \uu_\eps\cdot\partial_t({\pp-\pp_\eps})
%\dx\dt	+\int_0^{\red\tau}\int_{\Omega_\eps}\rho_\eps \uu_\eps\otimes  \uu_\eps:\Grad({\pp-\pp_\eps})\dx\dt\\&
%-\int_0^{\red\tau}\int_{\Omega_\eps}\SS(\Grad  \uu_\eps):\Grad({\pp-\pp_\eps})\dx\dt
%+\int_0^{\red\tau}\int_{\Omega_\eps}\rho_\eps \ff \cdot({\pp-\pp_\eps}) \dx\dt \\&+\int_{\Omega_\eps}(\rho u)_{0,\eps}\cdot ({\pp(0,\cdot)-\pp_\eps(0,\cdot)})\dx
%-{\red 	\int_{\Omega_\eps}(\rho_\eps\uu_\eps)(\tau,\cdot)\cdot (\pp(\tau, \cdot) - \pp_\eps(\tau,\cdot)) \dx} =: \sum_{i=1}^6 I_i .
%		\end{split}
%	\end{equation}
 
 We now show that $I_i \to 0$ as $\eps \to 0$ for all $i$ except $I_3$ which will eventually yield the Brinkman term. We start by estimating $I_1$.
This term can be bounded using the Sobolev embedding and H\"older's inequality as
		\begin{align*}
		|I_1|
	&	\leq \int_0^T|\partial_t\psi|\int_{\Omega_\eps}|\rho_\eps \uu_\eps\cdot (\pphi-\pphi_\eps)|\dx\dt\\
    &\lesssim \| \rho_\eps\|_{L^\infty(0,T;L^\gamma(D))}\| \uu_\eps\|_{L^2(0,T;L^6(D; \RR^3))}\|\pphi-\pphi_\eps\|_{L^\frac{6\gamma}{5\gamma -6}(D; \RR^3)}.
		%&\leq \int_0^T|\partial_t\psi|\| \rho_\eps-\bar{\rho}\|_{L^\gamma(\Omega)}\| \uu_\eps\|_{L^q(\Omega)}\| \pphi-\pphi_\eps\|_{L^{\frac{q\gamma}{q(\gamma -1)-\gamma}}(\Omega)}~dt\\
		%&\quad\quad +\int_0^T|\partial_t\psi|\bar{\rho}\|  \uu_\eps\|_{L^2(\Omega)}\| \pphi-\pphi_\eps\|_{L^2(\Omega)}~dt\\
		%&\lesssim 
		%\|\partial_t\psi\|_{L^2(0,T)}
		%\| \rho_\eps-\bar{\rho}\|_{L^\infty(0,T;L^\gamma(\Omega))}\| \uu_\eps\|_{L^2(0,T;L^6(\Omega))}\|\pphi-\pphi_\eps\|_{L^\frac{6\gamma}{5\gamma -6}(\Omega)}+\bar{\rho}
		%\|\partial_t\psi\|_{L^2(0,T)}
		%\|  \uu_\eps\|_{L^2(0,T;L^2(\Omega))}\| \pphi-\pphi_\eps\|_{L^2(\Omega)}.
		%\\
		%&\lesssim {\red ???}\Ma(\eps)^{\frac{2}{\gamma}}
	%	\|\pphi-\pphi_\eps\|_{L^{r}(\Omega)}+\| \pphi-\pphi_\eps\|_{L^2(\Omega)},
	\end{align*}
By $\gamma > 3/2$, we have $6\gamma/(5\gamma - 6) < 6$ and thus $W_0^{1,2}$ is compactly embedded in $L^\frac{6\gamma}{5\gamma-6}$. Hence, by \eqref{M1-ii}, \eqref{konv1}, and \eqref{konv3}, it follows that $|I_1|\to 0$.

		To estimate the convective term $I_2$, we write
\begin{equation*}
	\cM_\eps^{\rm ess}:=\{ (t,x)\in (0,T)\times \Omega_\eps\mid \frac{1}{2}\bar{\rho} <\rho_\eps(t,x)<\bar{\rho} +1 \},\quad
\cM_\eps^{\rm res}:=((0,T)\times \Omega_\eps)\setminus \cM_\eps^{\rm ess},
\end{equation*}
	\begin{equation}\label{resandess}\rho_\eps=\rho_\eps^{\rm ess}+\rho_\eps^{\rm res},\quad
		\rho_\eps^{\rm ess}:=\rho_\eps\chi_{\cM_\eps^{\rm ess}},\quad
		\rho_\eps^{\rm res}:=\rho_\eps\chi_{\cM_\eps^{\rm res}}.
	\end{equation}
Similarly, we split $I_2=I_2^{\rm ess}+I_2^{\rm res}$, where $I_2^{\rm ess}$ and $I_2^{\rm res}$ denote the corresponding integrals on $\cM_\eps^{\rm ess}$ and $\cM_\eps^{\rm res}$, respectively. For the essential part we use $|\rho_\eps^{\rm ess}|\lesssim 1$ and Sobolev embedding to conclude 
	\begin{align*}
	|I_2^{\rm ess}|
&\leq \int_0^T|\psi|\int_{\Omega_\eps}| \rho_\eps^{\rm ess}|| \uu_\eps|^2|\Grad ( \pphi-\pphi_\eps)|\dx\dt
%&\leq \int_0^T|\partial_t\psi|\| \rho_\eps^{ess}\|_{L^\infty(\Omega)}\| \uu_\eps\|^2_{L^6(\Omega)}\|\Grad (\pphi-\pphi_\eps)^{ess}\|_{L^{\frac{3}{2}}(\Omega)}~dt\\
\\
&\lesssim \| \uu_\eps\|^2_{L^2(0,T;L^6(D; \RR^3))}\|\Grad ( \pphi-\pphi_\eps) \|_{L^{\frac{3}{2}}(D; \RR^3)}\\
&\lesssim \|\Grad ( \pphi-\pphi_\eps) \|_{L^{\frac{3}{2}}(D; \RR^3)}.
	\end{align*}
    The last term vanishes as $\eps \to 0$ due to \eqref{M1-iii}. Moreover, we find from \eqref{petka} that
    \begin{align}\label{coerc1}
        \|\rho_\eps^{\rm ess} - \bar{\rho}\|_{L^\infty(0,T; L^2(D))} \lesssim \Ma(\eps),
    \end{align}
    where we used that on $\cM_\eps^{\rm ess}$, the function $P(\rho_\eps) - P'(\bar{\rho})(\rho_\eps - \bar{\rho}) - P(\bar{\rho})$ is coercive in the sense that
    \begin{align*}
        P(\rho_\eps) - P'(\bar{\rho})(\rho_\eps - \bar{\rho}) - P(\bar{\rho}) \gtrsim |\rho_\eps - \bar{\rho}|^2,
    \end{align*}
    see \cite[Lemma~5.1]{FeireislNovotny2009singlim}. Due to \eqref{M1-iii} the essential part of the convective term vanishes for $\eps\to 0$. For the residual part we use the following estimate from the proof of \cite[Theorem 1.3]{BellaFeireislOschmann2023b}:
	\begin{align}\label{estRhores}
		\| \rho_\eps^{\rm res}\|_{L^\infty(0,T;L^\gamma(D))}\lesssim \Ma(\eps)^{\frac{2}{\gamma}}.
	\end{align}
From this  we obtain
	\begin{align*}
		|I_2^{\rm res}|
		&\leq \int_0^T |\psi|\int_{\Omega_\eps}| \rho_\eps^{\rm res}|| \uu_\eps|^2|\Grad ( \pphi-\pphi_\eps) |\dx\dt\\
		%&\leq \int_0^T|\partial_t\psi|\| [\rho_\eps-\bar{\rho}]_{res}\|_{L^\gamma(\Omega)}\| \uu_\eps\|^2_{L^6(\Omega)}\|\Grad (\pphi-\pphi_\eps)^{res}\|_{L^{r}(\Omega)}~dt\\
		%&+\int_0^T|\partial_t\psi|\bar{\rho}\| \uu_\eps\|^2_{L^6(\Omega)}\|\Grad (\pphi-\pphi_\eps)^{res}\|_{L^{\frac{3}{2}}(\Omega)}~dt\\
		&\lesssim \| \rho_\eps^{\rm res}\|_{L^\infty(0,T;L^\gamma(D))}\| \uu_\eps\|^2_{L^2(0,T;L^6(D; \RR^3))}\|\Grad ( \pphi-\pphi_\eps) \|_{L^{r}(D; \RR^3)}\\
		&\lesssim \Ma(\eps)^{\frac{2}{\gamma}} \|\Grad ( \pphi-\pphi_\eps) \|_{L^{r}(D; \RR^3)},
	\end{align*}
	where $r=\frac{3\gamma}{2\gamma-3}>\frac{3}{2}$. By \eqref{M1-iv}, the right-hand side of the above inequality vanishes for $\eps\to 0$, which in turn shows that $|I_2| \to 0$ as $\eps \to 0$.\\

    Next, we want to prove that $I_3$ converges to the Brinkman term. Therefore, we first show that $\uu_\eps$ and $\uu$ satisfy the conditions of \eqref{v1}--\eqref{v4}. \eqref{v1} is clear from  \eqref{konv3}. Let us define
\begin{align*}
    g_\eps&:=-\frac{\rho_\eps}{\bar{\rho}},\\
    \hh_\eps&:=\frac{(\bar{\rho}-\rho_\eps)\uu_\eps}{\bar{\rho}}.
\end{align*}
Then, by the first equation of \eqref{NavierSystem}, we have
\begin{align*}
    \div(\uu_\eps)=\frac{1}{\bar{\rho}}\div((\bar{\rho}-\rho_\eps)\uu_\eps)+\frac{1}{\bar{\rho}}\div(\rho_\eps\uu_\eps)
    =\div(\hh_\eps)+\partial_t g_\eps
\end{align*}
in the sense of distributions. Further, we have from \eqref{coerc1} and \eqref{estRhores} that
\begin{align*}
    \| \hh_\eps\|_{L^2(0,T;L^{\frac{6\gamma}{\gamma +6}}(D;\RR^3))}\lesssim \| \rho_\eps-\bar{\rho}\|_{L^\infty (0,T;L^\gamma (D))}\| \uu_\eps\|_{L^2(0,T;L^6(D; \RR^3))}\lesssim \Ma(\eps)^{\min\{\frac{2}{\gamma},1\}},
\end{align*}
and $g_\eps\to -1$ in $L^\infty (0,T;L^\gamma (D))$ due to \eqref{konv1}. In turn, we get \eqref{v3} and \eqref{v4} for $M_\eps:=\cM_\eps^{{\rm ess}}$ such that we can use \eqref{M1-v} and \eqref{M1-vi} for $\vv_\eps=\uu_\eps$ and $\vv=\uu$. By the divergence-free assumption on $\pphi$ and \eqref{M1-i}, we have $\div(\pphi-\pphi_\eps)=0$. This together with \eqref{konv3} and \eqref{M1-v} leads to
\begin{align*}
    I_{3}
    &=-\int_0^\tau\int_{\Omega_\eps}\SS(\Grad  \uu_\eps):\Grad(\psi(\pphi-\pphi_\eps))  \dx\dt\\
    %=-\int_0^T\psi(t) \int_\Omega \SS(\Grad  \uu_\eps):\Grad (\pphi-\pphi_\eps)) \dx\dt
    %=-\int_0^T\psi(t)\left( \int_\Omega (\mu (\Grad  \uu_\eps+\Grad^T \uu_\eps-\frac{2}{3}\div( \uu_\eps)\bbbone)+\eta \div ( \uu_\eps\bbbone)):\Grad (\pphi-\pphi_\eps) \d x\right) \d t
    &=-\int_0^\tau\psi\left( \mu\int_{\Omega_\eps} \Grad \uu_\eps:\Grad (\pphi-\pphi_\eps) \dx+\left(\frac{\mu}{3}+\eta\right) \int_{\Omega_\eps} \div( \uu_\eps)\div(\pphi-\pphi_\eps)  \dx\right) \dt\\
    %=-\int_0^T \psi(t) \mu \left(\int_\Omega \Grad \uu_\eps:\Grad (\pphi-\pphi_\eps) \d x\right) \d t
    &=-\int_0^\tau \psi \mu \left(\int_{\Omega_\eps} \Grad \uu_\eps:\Grad \pphi \dx\right) \dt+\int_0^\tau \psi \mu \left(\int_{\Omega_\eps} \Grad \uu_\eps:\Grad \pphi_\eps \dx\right) \dt\\
    &\xrightarrow{\eps\to 0} -\int_0^\tau \psi \mu \left(\int_{\Omega} \Grad \uu:\Grad \pphi \dx\right) \dt+\int_0^\tau \psi \mu \left(\int_\Omega \Grad \uu:\Grad \pphi \dx+\int_{\Omega} (\bM \pphi)\cdot \uu\right) \dt\\
    &=\mu \int_{0}^T \int_{\Omega} \psi(\bM\pphi)\cdot \uu \dx\dt.
\end{align*}
Indeed, since
 $\div(\pphi-\pphi_\eps)=0$ and by smoothing $\pphi-\pphi_\eps$ by convolution with some suitable functions $(\eta_\delta)_{\delta>0}$, we get
 \begin{align*}
    \int_0^{\tau}\int_{\Omega_\eps}&\Grad^T \uu_\eps:\Grad(\psi(\pphi-\pphi_\eps)) \dx\dt=
\lim_{\delta\to 0}\int_0^{\tau}\psi\int_{\Omega_\eps}\Grad^T \uu_\eps:\Grad((\pphi-\pphi_\eps)\ast \eta_\delta) \dx\dt\\
&= -\lim_{\delta\to 0}\int_{0}^{\tau}\psi\int_{\Omega_\eps}  \uu_\eps\cdot\Grad(\div((\pphi-\pphi_\eps)\ast\eta_\delta))\dx\dt=0,
\end{align*}
and 
\begin{align*}
\int_0^{\tau}\int_{\Omega_\eps} (\div \uu_\eps \id):\Grad(\pphi-\pphi_\eps)\dx\dt = \int_0^{\tau}\psi\int_{\Omega_\eps} \div \uu_\eps \ \div(\pphi-\pphi_\eps)\dx\dt=0.
\end{align*}

	%\todo[inline]{Vorausetzung an Mach number}
	To estimate  $I_4$ we use that 
$\ff\in L^\infty(0,T; L^\infty(D;\RR^3))$, which together with {\eqref{konv1} and \eqref{M1-ii}} yields 
	\begin{align*}
				|I_{4}|
			&\leq\int_0^T{|\psi|}\int_{\Omega_\eps} |\rho_\eps \ff|   |\pphi-\pphi_\eps| \dx\dt\\
			%&\leq\| f\|_{L^\infty((0,T)\times\Omega,\RR^3)}\int_0^T|\partial_t\psi|\| \rho_\eps-\bar{\rho}\|_{L^\gamma(\Omega)}\| \pphi-\pphi_\eps\|_{L^{\frac{\gamma}{\gamma-1}}(\Omega)}\\
			%&\quad\quad+\| f\|_{L^\infty((0,T)\times\Omega,\RR^3)}\int_0^T\bar{\rho}|\partial_t\psi|\|\pphi-\pphi_\eps\|_{L^1(\Omega)}~dt\\
			&\lesssim
			%\| f\|_{L^\infty((0,T)\times\Omega,\RR^3)}\|\partial_t\psi\|_{L^1(0,T)}
			\| \rho_\eps\|_{L^\infty(0,T;L^\gamma(D))}\|\pphi-\pphi_\eps\|_{L^{\frac{\gamma}{\gamma-1}}(D; \RR^3)}
			%\\&\quad\quad+
			%\| f\|_{L^\infty((0,T)\times\Omega,\RR^3)}\| \partial_t\psi\|_{L^1(0,T)}
			%\bar{\rho}\|\pphi-\pphi_\eps\|_{L^1(\Omega)}
			%\\
		%	&\leq C\Ma(\eps)^{\min\{ \frac{2}{\gamma}, 1 \}}\|\pphi-\pphi_\eps\|_{L^{\frac{\gamma}{\gamma-1}}(\Omega)}+C\|\pphi-\pphi_\eps\|_{L^{1}(\Omega)}
			\to 0
	\end{align*}
    by compact Sobolev embedding as $\gamma/(\gamma-1) < 6$ for any $\gamma > 6/5$. Finally, due to $2\gamma/(\gamma - 1) < 6$ for any $\gamma > 3/2$, \eqref{ID2}, and \eqref{M1-ii}, we get
	\begin{align*}
		|I_5|
		&\leq |\psi(0)|\int_{\Omega_\eps} |(\rho\uu)_{0,\eps
		} \cdot (\pphi-\pphi_\eps)| \dx
        \lesssim \| (\rho\uu)_{0,\eps
		}\|_{L^\infty(0,T;L^{\frac{2\gamma}{\gamma +1}}(D; \RR^3))}\| \pphi-\pphi_\eps\|_{L^{\frac{2\gamma}{\gamma-1}}(D; \RR^3)}\to 0.
	\end{align*}
By \eqref{konv1}, \eqref{konv3}, and \eqref{M1-ii}, we have $I_6\in L^1(0,T)$ and for every function $\xi\in L^\infty(0,T)$
\begin{align*}
    \left|\int_0^T \xi (\tau) I_6(\tau)\dtau \right| &= \left|\int_0^T\int_{\Omega_\eps}\xi(\tau)(\rho_\eps \uu_\eps)(\tau,\cdot)\cdot (\pphi-\pphi_\eps) \psi(\tau) \dx\dtau \right|\\
    &\lesssim \|\xi\|_{L^\infty(0,T)}\|\rho_\eps\|_{L^\infty(0,T;L^\gamma (D))}\| \uu_\eps\|_{L^2(0,T;L^6(D; \RR^3))}\| \pphi-\pphi_\eps\|_{L^\frac{6\gamma}{5\gamma-6}(D; \RR^3)} \to 0,
\end{align*}
by the same arguments as for $I_1$. Consequently, we have $I_6\rightharpoonup 0$ in $L^1(0,T)$ and therefore $I_6(\tau)\to 0$ for almost any $\tau \in [0,T]$. In total, we proved that 
\begin{equation}\label{lim-I1-I5}
\lim_{\eps\to 0} \sum_{i=1}^6 I_i = \lim_{\eps\to 0} I_3=\mu \int_{0}^\tau\int_{\Omega} \psi(\bM\pphi)\cdot \uu\dx\dt.
\end{equation}

We now look at the left-hand side of \eqref{basisGL}. Using \eqref{ID2}, \eqref{konv1}, \eqref{konv3}, and the fact that we might prolong $\pphi \in C_c^1(\Omega; \RR^3)$ by zero to the whole of $D$, we get
	\begin{equation}\label{lim-lhs1}
		\begin{split}
\lim_{\eps\to 0}	\int_0^{\tau}\int_{\Omega_\eps} \rho_\eps \ff \cdot (\pphi\psi)\dx\dt &= \lim_{\eps\to 0}	\int_0^{\tau}\int_{D} \rho_\eps \ff \cdot (\pphi\psi)\dx\dt =\int_0^{\tau} \int_D \bar{\rho} \ff \cdot (\pphi\psi)\dx\dt \\
&= \int_0^{\tau} \int_\Omega \bar{\rho} \ff \cdot (\pphi\psi)\dx\dt,\\
\lim_{\eps\to 0}	\int_0^{\tau}\int_{\Omega_\eps} \SS(\Grad \uu_\eps):\Grad (\pphi\psi)\dx\dt&= \int_0^{\tau}\int_\Omega \SS(\Grad \uu):\Grad (\pphi\psi)\dx\dt = \int_0^{\tau}\int_\Omega \mu\Grad \uu:\Grad (\pphi\psi)\dx\dt ,\\
\lim_{\eps\to 0}
\int_{\Omega_\eps} (\rho\uu)_{0,\eps}\cdot  \pphi\psi(0)\dx&=\int_\Omega \bar{\rho}\uu_0\cdot  \pphi\psi(0)\dx,\\
\lim_{\eps\to 0}	\int_0^{\tau} \partial_t\psi\int_{\Omega_\eps} \rho_\eps\uu_\eps \cdot \pphi\dx\dt &= \int_0^{\tau} \partial_t\psi\int_\Omega \bar{\rho}\uu \cdot \pphi\dx\dt.
		\end{split}
	\end{equation}
Moreover, for every function $\xi\in L^\infty(0,T)$, we have, using \eqref{konv1} and \eqref{konv3},
\begin{align*}
    \lim_{\eps\to 0}\int_0^T\int_{\Omega_\eps} \xi(\tau)(\rho_\eps\uu_\eps)(\tau,\cdot)\pphi\psi(\tau)\dx\dtau = \int_0^T\int_{\Omega} \xi(\tau)(\bar\rho\uu)(\tau,\cdot)\pphi\psi(\tau)\dx\dtau.
\end{align*} Consequently, we see 
\begin{align}\label{lim-lhs3}
    \lim_{\eps\to 0}\int_{\Omega_\eps} (\rho_\eps\uu_\eps)(\tau,\cdot)\pphi\psi(\tau)\dx = \int_{\Omega} (\bar\rho\uu)(\tau,\cdot)\pphi\psi(\tau)\dx
\end{align}
for almost any $\tau\in [0,T]$.
    
To treat the convective term we use again the decomposition $\rho_{\eps}=\rho_{\eps}^{\rm ess}+\rho_{\eps}^{\rm res}$. Recall from \eqref{estRhores} that
	\begin{align*}
		\| \rho_\eps^{\rm res}\|_{L^\infty(0,T;L^\gamma(D))}\lesssim \Ma(\eps)^{\frac{2}{\gamma}}.
	\end{align*} 
Consequently, we may estimate
	\begin{equation}\label{conv1}
		\begin{split}
\Bigg|\int_0^{\tau}\int_{\Omega_\eps} \rho_\eps^{\rm res} \uu_\eps\otimes \uu_\eps :\Grad(\pphi\psi)\dx\dt\Bigg|
&\lesssim \int_0^T|\psi|\| \rho_\eps^{\rm res}\|_{L^\gamma (D)}\| \uu_\eps\|_{L^6(D;\RR^3)}^2\dt\\
&\lesssim \Ma(\eps)^{\frac{2}{\gamma}}\| \uu_\eps\|_{L^2(0,T;W^{1,2}(D;\RR^3))}\to 0,
		\end{split}
	\end{equation}
where we used \eqref{assumptions1} and \eqref{konv3}.\\

	Finally, we know that $\rho_\eps^{\rm ess}\uu_\eps\otimes\uu_\eps$ is uniformly bounded in $L^{1}(0,T;L^3(D;\RR^{3\times 3}))$ and $L^{\infty}(0,T;L^1(D;\RR^{3\times 3}))$. Hence, we may extract a subsequence (not relabeled) such that $ \rho_\eps^{\rm ess}\uu_\eps\otimes\uu_\eps$ converges weakly in $L^{\frac{4}{3}}(0,T;L^2(D;\RR^{3\times 3}))$, the weak limit of which we denote by $\overline{\rho \uu\otimes\uu}$.
Thus,
\begin{align}\label{conv2}
    \int_0^T\int_{\Omega_\eps} (\rho_\eps^{\rm ess} \uu_\eps\otimes \uu_\eps) :\Grad(\pphi\psi) \dx\dt
    \to \int_0^T\int_\Omega  \overline{\rho \uu\otimes\uu} :\Grad(\pphi\psi) \dx\dt.
\end{align}
	A similar argument leads to
	\begin{equation*}
		\sqrt{\rho_\eps}\uu_\eps\xrightharpoonup{*} \sqrt{\bar{\rho}}\uu\quad\text{ weakly-$*$ in }L^\infty(0,T;L^2(D;\RR^3)).
	\end{equation*}
	Since $\vv\mapsto \vv\otimes \vv$ is convex, we have
	\begin{align}\label{conv3}
		\fR:=\overline{\rho \uu\otimes\uu}-\bar{\rho}\uu\otimes \uu\geq 0 \quad\text{ in the sense of symmetric matrices}.
	\end{align}
Gathering \eqref{conv1}, \eqref{conv2}, and \eqref{conv3}, we find
	\begin{equation}\label{lim-lhs2}
	\lim_{\eps\to 0}	\int_0^{\tau}\int_{\Omega_\eps} \rho_\eps \uu_\eps\otimes \uu_\eps  :\Grad(\pphi\psi)\dx\dt =\int_0^{\tau}\int_\Omega \bar{\rho} \uu\otimes \uu  :\Grad(\pphi\psi)\dx\dt+\int_0^{\tau}\int_\Omega \fR :\Grad(\pphi\psi)\dx\dt.
	\end{equation}
Combining \eqref{basisGL}, \eqref{lim-I1-I5}, \eqref{lim-lhs1}, \eqref{lim-lhs3}, and \eqref{lim-lhs2}, we infer
	\begin{equation}\label{final-momentum}
	\begin{split}
		\int_0^{\tau}\int_\Omega \bar{\rho}\uu\cdot \partial_t\pp+\bar{\rho} \uu\otimes\uu:\Grad\pp-&\mu\Grad\uu:\Grad\pp+\bar{\rho} \ff \cdot \pp+\fR:\Grad\pp \dx\dt\\
        &+\int_\Omega \bar{\rho}\uu\cdot\pp(0,\cdot)\dx-\int_\Omega \bar{\rho}\uu\cdot\pp(\tau,\cdot)\dx
		=\mu \int_0^{\tau}\int_\Omega \bM\uu \cdot \pp\dx\dt,
	\end{split}
\end{equation} 
for all $\pp\in C^1_c([0,T)\times\Omega)$, $\pp(t,x)=\psi(t)\pphi(x)$ with $\div\pp=0$, and almost any $\tau\in [0,T]$. Eventually, by a density argument, \eqref{final-momentum} holds true for any $\pp\in C^1_c([0,T)\times\Omega)$ satisfying $\div\pp=0$.

\item\paragraph{Step 3: Energy inequality.}
    Since $(\rho_\eps, \uu_\eps)$ are finite energy weak solutions, we have
\begin{align}\label{energy-eps}
            &\int_{\Omega_\eps}\left[ \frac{1}{2} \id_{\{\rho_\eps>0\}} \frac{|\rho_\eps\uu_\eps|^2}{\rho_\eps} + \frac{1}{\Ma^2(\eps)}P(\rho_\eps) \right](\tau,\cdot) \dx+\int_0^\tau\int_{\Omega_\eps}\SS(\Grad\uu_\eps):\Grad\uu_\eps \dx\dt\nonumber\\
            &\leq \int_{\Omega_\eps} \frac{1}{2} \id_{\{\rho_{0, \eps} > 0\}} \frac{|(\rho\uu)_{0,\eps}|^2}{\rho_{0,\eps}} + \frac{1}{\Ma^2(\eps)}P(\rho_{0,\eps}) \dx+\int_0^\tau \int_{\Omega_\eps}\rho_\eps \ff\cdot \uu_\eps \dx\dt
\end{align}
for almost any $\tau\in [0,T]$. First, we want to control the limit of the second term of the energy inequality
\begin{align*}
    \int_0^\tau\int_{\Omega_\eps}\SS(\Grad\uu_\eps):\Grad\uu_\eps \dx\dt
    &=\int_0^\tau\int_{\Omega_\eps} \mu \Grad \uu_\eps:\Grad \uu_\eps \dx\dt+\left(\frac{\mu}{3}+\eta\right)\int_0^\tau\int_{\Omega_\eps} |\div(\uu_\eps)|^2 \dx\dt\nonumber\\
    &\geq \mu \int_0^\tau \int_{\Omega_\eps} |\Grad \uu_\eps|^2 \dx\dt,
\end{align*}
where we used that with partial integration and the symmetry of second derivatives, we have
\begin{align*}
    \int_0^T\int_{\Omega_\eps}&\Grad^T\uu_\eps:\Grad\uu_\eps  \dx\dt
    =\lim_{\delta\to 0}\int_0^T\int_{\Omega_\eps}\Grad^T\uu_\eps:\Grad(\uu_\eps\ast\eta_\delta)  \dx\dt
    =\lim_{\delta\to 0}-\int_0^T\int_{\Omega_\eps}\uu_\eps\cdot \Grad\div(\uu_\eps\ast\eta_\delta)  \dx\dt\\
    &=\lim_{\delta\to 0}\int_0^T\int_{\Omega_\eps}\div(\uu_\eps)\cdot \div(\uu_\eps\ast\eta_\delta)  \dx\dt=\int_0^T\int_{\Omega_\eps}|\div(\uu_\eps)|^2  \dx\dt,
\end{align*}
where $(\eta_\delta)_{\delta >0}$ are suitable mollifiers.
This leads to
\begin{align*}
            &\int_{\Omega_\eps}\left[ \frac{1}{2} \id_{\{\rho_\eps>0\}} \frac{|\rho_\eps\uu_\eps|^2}{\rho_\eps} +\frac{1}{\Ma^2(\eps)}P(\rho_\eps) \right](\tau,\cdot) \dx+\mu \int_0^\tau \int_{\Omega_\eps} |\Grad \uu_\eps|^2 \dx\dt\nonumber\\
            &\leq \int_{\Omega_\eps} \frac{1}{2} \id_{\{\rho_{0, \eps} > 0\}} \frac{|(\rho\uu)_{0,\eps}|^2}{\rho_{0,\eps}} + \frac{1}{\Ma^2(\eps)}P(\rho_{0,\eps}) \dx+\int_0^\tau \int_{\Omega_\eps}\rho_\eps \ff\cdot \uu_\eps \dx\dt
\end{align*}
for almost any $\tau\in [0,T] $. Let now $\xi\in C_c(0,T)$ with $\xi\geq 0$, $\int_0^T\xi(\tau) \dtau=1$. Multiplication by $\xi (\tau)$ and integration leads to
\begin{align}
            &\int_0^T\int_{\Omega_\eps}\xi(\tau)\left[ \frac{1}{2} \id_{\{\rho_\eps>0\}} \frac{|\rho_\eps\uu_\eps|^2}{\rho_\eps} + \frac{1}{\Ma^2(\eps)}P(\rho_\eps) \right](\tau,\cdot) \dx\dtau+\mu \int_0^T\xi(\tau)\int_0^\tau \int_{\Omega_\eps} |\Grad \uu_\eps|^2 \dx\dt\dtau\nonumber\\
            &\leq \int_0^T\int_{\Omega_\eps}\xi(\tau)\left[\frac{1}{2} \id_{\{\rho_{0, \eps} > 0\}} \frac{|(\rho\uu)_{0,\eps}|^2}{\rho_{0,\eps}} + \frac{1}{\Ma^2(\eps)}P(\rho_{0,\eps})\right] \dx\dtau+\int_0^T\xi(\tau)\int_0^\tau \int_{\Omega_\eps}\rho_\eps \ff\cdot \uu_\eps \dx\dt\dtau. \label{ctvrtka}
\end{align}
Since $\id_{\rho_{0, \eps} > 0} |(\rho \uu)_{0, \eps}|^2 / \rho_{0, \eps}$ is bounded in $L^{1}(0,T;L^3(D))$ and $L^{\infty}(0,T;L^1(D))$, we get
\begin{align*}
     \id_{\rho_\eps>0}\frac{|\rho_\eps\uu_\eps|^2}{\rho_\eps}\rightharpoonup \overline{\rho |\uu|^2}\quad\text{weakly in } L^\frac{4}{3}(0,T;L^2(D)).
\end{align*}
Additionally, by \eqref{integrability}, we have
\begin{align*}
    \int_{\Omega_\eps} \rho_\eps(\tau, \cdot) \dx = \int_{\Omega_\eps} \rho_{0, \eps} \dx \ \text{ for almost any } \ \tau \in [0,T],
\end{align*}
hence
\begin{align*}
    \int_{\Omega_\eps} P(\rho_\eps)(\tau, \cdot) \dx - \int_{\Omega_\eps} P(\rho_{0, \eps}) \dx &= \int_{\Omega_\eps} P(\rho_\eps)(\tau, \cdot) - P'(\bar{\rho})(\rho_\eps - \bar{\rho})(\tau, \cdot) - P(\bar{\rho}) \dx \\
    &\quad - \int_{\Omega_\eps} P(\rho_{0,\eps}) - P'(\bar{\rho})(\rho_{0, \eps} - \bar{\rho}) - P(\bar{\rho}) \dx.
\end{align*}
Moreover, by definition of $P$ (see \eqref{pressPot}), we have $P''(\rho) = p'(\rho)/\rho \geq 0$ and hence almost everywhere
\begin{align*}
    P(\rho_\eps) - P'(\bar{\rho})(\rho_\eps - \bar{\rho}) - P(\bar{\rho}) \geq 0,
\end{align*}
see also Lemma~5.1 in \cite{FeireislNovotny2009singlim}. With this at hand, \eqref{ID1}, \eqref{ID3}, \eqref{M1-vi}, \eqref{konv1}, and \eqref{konv3}, we can pass to the limit $\eps \to 0$ in \eqref{ctvrtka} to obtain
\begin{align*}
    &\int_0^T\int_\Omega \frac{1}{2}\xi(\tau)\overline{\rho|\uu|^2} (\tau,\cdot)\dx\dtau+\mu\int_0^T\int_0^{\tau}\int_\Omega \xi(\tau)|\Grad\uu|^2\dx\dt\dtau+\mu \int_0^T\int_0^{\tau}\int_\Omega \xi(\tau)(\bM\uu)\cdot\uu \dx\dt\dtau\\
    &\leq \int_{\Omega}\frac{1}{2}\bar{\rho} |\uu_0|^2\dx+\int_0^T\int_0^\tau \int_{\Omega}\xi(\tau)\bar{\rho}\ff\cdot \uu\dx\dt\dtau.
\end{align*}
Since $\rho_\eps |\uu_\eps|^2={\rm trace}[\rho_\eps \uu_\eps\otimes\uu_\eps]$ and ${\rm trace}[\, \cdot \, ]$ is a linear operator, we have for $\fR$ from $\eqref{conv3}$ that
\begin{align*}
    {\rm trace} [\fR]%={\rm trace}[\overline{\rho\uu\otimes\uu}]-{\rm trace}[\bar \uu\otimes\uu]
    =\overline{\rho|\uu|^2}-\bar{\rho}|\uu|^2.
\end{align*}
As the function $\xi$ is arbitrary, we get
\begin{align*}
    &\int_\Omega \frac{1}{2}\bar{\rho}|\uu|^2 (\tau,\cdot) \dx+\int_\Omega \frac{1}{2}{\rm trace} [\fR](\tau,\cdot) \dx+\mu\int_0^{\tau}\int_\Omega |\Grad\uu|^2 \dx\dt+\mu \int_0^{\tau}\int_\Omega (\bM\uu)\cdot\uu \dx\dt\\
    &\leq \int_{\Omega}\frac{1}{2}\bar{\rho} |\uu_0|^2 \dx+\int_0^\tau \int_{\Omega}\bar{\rho}\ff\cdot \uu \dx\dt,
\end{align*}
showing that $(\bar{\rho}, \uu)$ is a dissipative solution of the incompressible Navier-Stokes-Brinkman system \eqref{systemstarkeloesung}.

\end{proof}
To finally prove Theorem~\ref{mainTheorem}, we will make use of the weak-strong uniqueness principle (see, e.g., \cite{AbbatielloFeireisl2020}), which is the content of the following proposition.

\begin{prop}[Weak-Strong uniqueness]\label{prop2}
	Let $(\bar{\rho},\uu)$ be a dissipative solution to \eqref{systemstarkeloesung} emanating from the initial datum $\uu_0$ in \eqref{u0}. If there exists a strong solution to \eqref{systemstarkeloesung} emanating from the same initial datum $\uu_0$, then it coincides with $(\bar{\rho},\uu)$.
\end{prop}
\begin{proof}
We divide the proof into two steps. First, we recall the derivation of the relative energy inequality, and subsequently apply it to our setting.
	
\item\paragraph{Step 1: Relative Energy Inequality.}
	
We test \eqref{momentum-diss} with $\pp=\bU\in C^1_c([0,T]\times\Omega;\RR^3)$ such that $\div \bU = 0$, and add on both sides the term $\int_{0}^{\tau}\int_\Omega{\bar{\rho}}\partial_t \bU\cdot \bU\dx\dt$. Using that 
\begin{equation*}
\int_{0}^{\tau}\int_\Omega{\bar{\rho}}\partial_t \bU\cdot \bU\dx\dt= \left[ \int_\Omega \frac12{\bar{\rho}} |\bU|^2\dx \right]_{t=0}^{t=\tau},
	\end{equation*}
	and summing with the energy inequality \eqref{energy-diss}, we obtain 
		\begin{align*}
				\biggl[\int_\Omega\frac12 {\bar{\rho}}|\uu-\bU|^2\dx\biggr]_{t=0}^{t=\tau}
		&	+ \int_\Omega\frac12 {\rm trace}[\fR](\tau,\cdot)\dx + \mu\int_0^{\tau}\int_\Omega |\Grad\uu|^2\dx\dt+\mu \int_0^\tau\int_\Omega \bM\uu\cdot\uu\dx\dt\\
			&	\le
			  \int_{0}^{\tau}\int_\Omega\bar{\rho}\partial_t \bU\cdot \bU\dx\dt +\int_0^\tau \int_{\Omega}\bar{\rho}\ff\cdot \uu\dx\dt + \mu\int_0^\tau\int_\Omega\bM\uu\cdot \bU\dx\dt\\
			  &
			  -
			  \int_0^{\tau}\int_\Omega \bar{\rho}\uu\cdot \partial_t \bU+(\bar{\rho} \uu\otimes\uu):\Grad \bU-\mu\Grad\uu:\nabla_x \bU + \bar{\rho} \ff\cdot \bU+\fR:\Grad \bU \dx\dt.
	\end{align*}
	By regrouping some terms, we infer
	\begin{equation}\label{rel-energy-ineq1}
	\begin{split}
	&	\biggl[\int_\Omega\frac12 {\bar{\rho}}|\uu-\bU|^2\dx\biggr]_{t=0}^{t=\tau} + \int_\Omega\frac12 {\rm trace}[\fR](\tau,\cdot)\dx + \mu\int_0^{\tau}\int_\Omega \Grad\uu:\Grad( \uu-\bU)\dx\dt+\mu \int_0^\tau\int_\Omega \bM\uu\cdot(\uu-\bU)\dx\dt\\
&	\le
  \int_{0}^{\tau}\int_\Omega
  \bar{\rho}\partial_t \bU\cdot (\bU-\uu)-(\bar{\rho} \uu\otimes\uu):\Grad \bU \dx\dt 
  +\int_0^\tau \int_{\Omega}\bar{\rho}\ff\cdot( \uu-\bU)\dx\dt -
  \int_0^\tau \int_{\Omega} \fR:\Grad \bU\dx\dt.
	\end{split}
	\end{equation}
	 By the incompressibility condition \eqref{incompressibility-diss} it follows that
	\begin{align*}
 \int_0^\tau \int_{\Omega} \bar{\rho}\uu\otimes\uu:\Grad \bU \dx\dt&
 = \int_0^\tau \int_{\Omega} \bar{\rho}\uu\otimes(\uu-\bU):\Grad \bU \dx\dt
 \\
 &= \int_0^\tau \int_{\Omega} \bar{\rho}(\uu - \bU)\otimes(\uu-\bU):\Grad \bU \dx\dt+
\int_0^\tau \int_{\Omega} \bar{\rho} \bU\otimes(\uu-\bU):\Grad \bU \dx\dt.
	\end{align*}
	Thus, we rewrite \eqref{rel-energy-ineq1} to get the relative energy inequality as
		\begin{equation}\label{rel-energy-ineq2}
		\begin{split}
				\biggl[\int_\Omega\frac12 {\bar{\rho}}|\uu-\bU|^2\dx\biggr]_{t=0}^{t=\tau}
			&+ \int_\Omega\frac12 {\rm trace}[\fR](\tau,\cdot)\dx + \mu \int_0^{\tau}\int_\Omega |\Grad (\uu-\bU)|^2 + \bM (\uu - \bU) \cdot(\uu-\bU) \dx\dt\\
			&	\le
			\int_{0}^{\tau}\int_\Omega
			  \bar{\rho} (\partial_t \bU+ \bU\cdot\Grad \bU)\cdot (\bU-\uu) \dx\dt 
			+\int_0^\tau \int_{\Omega}\bar{\rho}\ff\cdot( \uu-\bU)\dx\dt \\&
		-	\int_0^\tau \int_{\Omega} \fR:\Grad \bU\dx\dt
		- \int_{0}^\tau\int_\Omega \bar{\rho}(\uu - \bU)\otimes(\uu-\bU):\Grad \bU \dx\dt \\
        &- \mu \int_{0}^\tau\int_\Omega \Grad \bU : \Grad(\uu - \bU) \dx\dt - \mu \int_0^\tau \int_\Omega \bM \bU \cdot (\uu - \bU) \dx \dt.
		\end{split}
	\end{equation}
By density we can extend \eqref{rel-energy-ineq2} to any $\bU\in L^q(0,T; W^{1,q}(\Omega; \RR^3)) \cap W^{1,q}(0,T; L^q(\Omega; \RR^3))$ satisfying $\div \bU = 0$ and $\bU|_{\partial \Omega} = 0$, provided $q>1$ is large enough such that all occurring integrals are well-defined.\\
	
\item\paragraph{Step 2: Weak-strong uniqueness.}
	Let $\hat\uu$ be a strong solution to \eqref{systemstarkeloesung} emanating from the initial datum $\uu_0$. Then, since $\div(\uu-\hat\uu)=0$ a.e. in $(0,T)\times \Omega$, it holds
	\begin{equation*}
		\int_0^\tau\int_\Omega
			(\bar{\rho}\partial_t \hat\uu+ \bar{\rho} \hat\uu\cdot\Grad \hat\uu)\cdot (\hat\uu-\uu) \dx\dt = 
			\int_0^\tau\int_\Omega (-\mu\bM\hat\uu+ \mu \Delta_x \hat\uu + \bar{\rho}\ff) \cdot(\hat\uu-\uu) \dx\dt.
	\end{equation*}
Combining this with \eqref{rel-energy-ineq2} for $\bU=\hat\uu$, we obtain
	\begin{align*}
			\biggl[\int_\Omega\frac12 {\bar{\rho}}|\uu-\hat\uu|^2\dx\biggr]_{t=0}^{t=\tau}
	&	+ \int_\Omega\frac12 {\rm trace}[\fR](\tau,\cdot)\dx + \mu\int_0^{\tau}\int_\Omega |\Grad( \uu-\hat\uu)|^2 + \bM(\uu - \hat \uu) \cdot(\uu-\hat\uu) \dx\dt\\
		&	\le
	 	\int_0^\tau\int_\Omega (-\mu\bM\hat\uu+ \mu\Delta_x \hat \uu + \bar{\rho}\ff) \cdot(\hat\uu-\uu) \dx\dt
		+\int_0^\tau \int_{\Omega}\bar{\rho}\ff\cdot( \uu-\hat\uu)\dx\dt \\&
		-	\int_0^\tau \int_{\Omega} \fR:\Grad \hat\uu\dx\dt
		- \int_{0}^\tau\int_\Omega \bar{\rho} (\uu- \hat\uu)\otimes(\uu-\hat\uu):\Grad \hat\uu \dx\dt \\
        & - \mu \int_{0}^\tau\int_\Omega \Grad \hat \uu : \Grad(\uu - \hat \uu) \dx\dt - \mu \int_0^\tau \int_\Omega \bM \hat \uu \cdot (\uu - \hat \uu) \dx \dt,
	\end{align*}
which in turn reduces to
	\begin{align*}
			\int_\Omega\frac12 {\bar{\rho}}|\uu-\hat\uu|^2(\tau,\cdot)\dx&	+ \int_\Omega\frac12 {\rm trace}[\fR](\tau,\cdot)\dx	+ \mu\int_0^{\tau}\int_\Omega | \Grad\uu- \Grad\hat\uu|^2 + \bM(\uu-\hat\uu)\cdot(\uu-\hat\uu) \dx\dt\\
		&	\le
		-	\int_0^\tau \int_{\Omega} \fR:\Grad \hat\uu\dx\dt
		- \int_{0}^\tau\int_\Omega\bar{\rho}(\uu- \hat\uu)\otimes(\uu-\hat\uu):\Grad \hat\uu \dx\dt
		.
	\end{align*}
Being the second and the third term on the left-hand side both non-negative, we infer
	\begin{align*}
			\int_\Omega\frac12 {\bar{\rho}}|\uu-\hat\uu|^2(\tau,\cdot)\dx + \int_\Omega\frac12 {\rm trace}[\fR](\tau,\cdot)\dx
			\le
		-	\int_0^\tau \int_{\Omega} \fR:\Grad \hat\uu\dx\dt
		- \int_{0}^\tau\int_\Omega\bar{\rho}(\uu- \hat\uu)\otimes(\uu-\hat\uu):\Grad \hat\uu \dx\dt
		,
	\end{align*}
and thus
	\begin{align*}
		\int_\Omega\frac12 {\bar{\rho}}|\uu-\hat\uu|^2(\tau,\cdot)\dx + \int_\Omega\frac12 {\rm trace}[\fR](\tau,\cdot)\dx
		&\le \|\Grad \hat\uu\|_{L^\infty((0,\tau)\times\Omega; \RR^3)}
			\int_0^\tau \int_{\Omega}| \fR|\dx\dt
	\\	&+ \|\Grad \hat\uu\|_{L^\infty((0,\tau)\times\Omega; \RR^3)} \int_{0}^\tau\int_\Omega\bar{\rho}|\uu- \hat\uu|^2\dx\dt
		,
	\end{align*}
for a.e. $\tau\in [0,T]$. Eventually, since $\fR$ is positive semidefinite, it follows $|\fR|\lesssim{\rm trace}[\fR]$ and thus by Gr\"onwall's lemma we deduce $\uu=\hat\uu$ and $\fR=0$. This finishes the proof of Theorem~\ref{mainTheorem}.
\end{proof}

%%%%%%%%%%%%%%%%%%%%%%%%%%%%%%%%%%%%%%%%%%%%%%%%%%%%%%%%%%%%%%
%%%%%%%%%%%%%%%%%%%%%%%%%%%%%%%%%%%%%%%%%%%%%%%%%%%%%%%%%%%%%%

\section{Application to randomly perforated domains}\label{sec:applic}
This section is devoted to provide an application of Theorem \ref{mainTheorem}. More precisely,  we consider a family of randomly perforated domains $(\Omega_\eps)_{\eps>0}$, whose holes are balls distributed through a Poisson point process, and whose radii are identically and independently distributed (i.i.d.)~random variables that scale as $\eps^{3}$. Note that this corresponds to the randomized counterpart of our initial example of domains given in \eqref{firstDom}. We will show that, if ${\rm Ma}(\eps)$ satisfies suitable assumptions (see \eqref{MaFinal} below), then $((\Omega_\eps)_{\eps>0}, \Omega)$ satisfy \ref{M0} and \ref{M1}. Thus, in particular, Theorem~\ref{mainTheorem} applies. Moreover, in this way we are able to generalize the outcomes of \cite{BellaOschmann2022} to the time-dependent setting.\\

We let $\Omega\subset\RR^3$ be a smoothly bounded open set which is star-shaped with respect to the origin. We then define the family of  randomly perforated domains $(\Omega_\eps)_{\eps>0}$ as
\begin{align}\label{Omega_epsilon}
	\Omega_\eps:=\Omega\setminus H^\eps,\quad H^\eps:=\bigcup_{I_\eps}T_i^\eps
		=
	\bigcup_{I_\eps}B_{\eps^{3}r_i}(\eps z_i),\quad
    I_\eps:= \{ z_i\in \Phi\cap \eps^{-1}\Omega : \overline{B_{\eps^{3}r_i}(\eps z_i)} \subset \Omega\},
\end{align}
where $\Phi$ is a Poisson point process on $\RR^3$ with homogeneous intensity rate $\lambda >0$. Such a process is characterized by the following two properties:
\begin{itemize}
    \item For each measurable set $S \subset \RR^3$ with finite measure, the probability of finding exactly $n$ points from $\Phi$ inside $S$ is $(\lambda |S|)^n \exp(- \lambda |S|) / n!$;
    \item for two disjoint measurable sets $S_1, S_2 \subset \RR^3$, the random variables $S_1 \cap \Phi$ and $S_2 \cap \Phi$ are independent.
\end{itemize}
The radii $\cR:={\{r_i\}_{z_i\in\Phi}\subset[R_0,\infty) }$ for some $R_0>0$ are i.i.d.~random variables. Further, we assume that there exists a constant $\beta>0$ such that the radii $r_i$ satisfy the moment bound
\begin{align}\label{meanrho}
	\langle r^{1+\beta}\rangle <\infty,
\end{align}
where $\langle f \rangle$ is the expected value of the random variable $f$.
Let $(\tilde\Omega,\mathcal F,\mathbb P)$ be a probability space associated to the marked point process $(\Phi,\cR)$, i.e., the joint process of centers and radii distributed as above. For more details on marked Poisson point processes, we refer the reader to \cite{Last2017}.
\begin{thm}\label{theorem-random-holes}
	Let $(\Omega_{\eps})_{\eps>0}$ be as in \eqref{Omega_epsilon}. Then there exists $\delta=\delta(\beta)>0$ and an exponent $p(\delta, \gamma) >0$ such that, if ${\rm Ma}(\eps)$ vanishes fast enough such that
 %   {\red Still kept for now to remember where the form of $p(\delta, \gamma)$ comes from. To be deleted after you give green light:
%	\begin{align}
   %     \limsup_{\eps\to 0}\eps^{4 (\frac13 - \frac2\gamma) - (\frac23 - \frac1\gamma)\delta}\Ma(\eps)^{\frac{2}{\gamma}} &= 0, \label{Ma2} \\[1em]
    %    \limsup_{\eps\to 0}\eps^{3\frac{\gamma -6}{2\gamma -3}}\Ma(\eps)^{\frac{2}{\gamma}} &= 0, \quad\text{if }\gamma<6,\\[1em]        
	%	\limsup_{\eps\to 0} \eps^{- 1 - \frac{6}{\gamma}} \Ma (\eps)^{\min\{1,\frac{2}{\gamma}\}} &= 0, \quad\text{if }\gamma<3, \label{Ma1} \\[1em]
        %\limsup_{\eps\to 0} \eps^{-\frac{3(2-\gamma)}{\gamma}} {\rm Ma}(\eps)^{\frac{2}{\gamma}} &= 0, \quad\text{if }\gamma<2, \label{Ma3} \\[1em]
     %   \limsup_{\eps\to 0} \eps^{-3} \Ma (\eps)^{\min\{ \frac{2}{\gamma}, 1 \}} &= 0,\quad \text{if }\gamma\geq 3, \label{Ma4}
%	\end{align}
   % }
    \begin{align}\label{MaFinal}
        \limsup_{\eps \to 0} \eps^{-p(\delta, \gamma)} \Ma(\eps)^{\min \{1, \frac{2}{\gamma}\}} = 0,
    \end{align}
then $((\Omega_\eps)_{\eps>0}, \Omega)$ satisfy \ref{M0} and \ref{M1}. The exponent $p(\delta, \gamma)$ can be taken to be
 \begin{align*}
        p(\delta, \gamma) = \max \left\{ 3, \ 1 + \frac{6}{\gamma}, \ 3 \frac{6-\gamma}{2\gamma-3} \right\}.
    \end{align*}
\end{thm}

The rest of this section is devoted to the proof of Theorem \ref{theorem-random-holes}. Clearly the family  $((\Omega_{\eps})_{\eps>0},\Omega)$ with $(\Omega_\eps)_{\eps>0}$ given in \eqref{Omega_epsilon} satisfies \ref{M0}. Thus, we have to prove that also \ref{M1} is satisfied, that is \eqref{M1-i}--\eqref{M1-iv} and \eqref{M1-v}-\eqref{M1-vi} hold true. We start with a general result Lemma \ref{(M2)} (see Section~\ref{sec:general}) which says that if $(\Omega_\eps)_{\eps>0}$ is a family of perforated domains such that \ref{M0} and \ref{H1}--\ref{H5} below hold, then  also \eqref{M1-vi} holds. Afterwards we prove that \ref{H1}--\ref{H5}  are satisfied when  $(\Omega_\eps)_{\eps>0}$ are defined as in  \eqref{Omega_epsilon} and ${\rm Ma}(\eps)$ obeys \eqref{MaFinal} (see Section~\ref{sec:random}). Eventually, for a given $\pphi\in C_c^\infty(\Omega;\RR^3)$ with $\div\pphi=0$ we construct a family of test functions with the properties \eqref{M1-i}--\eqref{M1-iv} and \eqref{M1-v} (see Section~\ref{sec:random-bis}).

\begin{rmk}
    The authors believe that the assumption $\cR \subset [R_0, \infty)$ for some $R_0 > 0$ can be relaxed to
    \begin{align*}
        \langle r^{-m} \rangle \leq C < \infty,
    \end{align*}
    for some $C>0$ and any $m>0$. However, we wish not to unnecessarily complicate the calculations, hence we focus on the case where $r_i \geq R_0 > 0$ for any $i$.
\end{rmk}

%%%%%%%%%%%%%%%%%%%%%%%%%%%%%%%%%%%%%%%%%%%%%%%%%%%%%%%%%%%%%%
%%%%%%%%%%%%%%%%%%%%%%%%%%%%%%%%%%%%%%%%%%%%%%%%%%%%%%%%%%%%%%

\subsection{General perforated domains and special test functions}\label{sec:general}
Let $\Omega\subset\RR^3$ be a bounded domain with smooth boundary. Let $N(\eps)\in\mathbb N$, $N(\eps)\to\infty$ as $\eps\to 0$, and let  $(T_i^\eps)_{1\leq i\leq N(\eps)} \subset \Omega$ be a family of closed smooth sets with non-empty interior (the holes). We consider the family of perforated domains  $(\Omega_{\eps})_{\eps>0}$ given by
\begin{align}\label{Omega-holes}
	\Omega_\eps:=\Omega\setminus
	\bigcup_{i=1}^{N(\eps)} T_i^\eps.
\end{align}
We assume that the holes $T_i^\eps$ are such that \ref{M0} holds. Moreover, we assume that there exist pairs of functions $(\oomega_k^\eps,{\mmu}_k)_{1\leq k\leq 3}$ with the following properties:
\begin{enumerate}[label=(H\arabic*)]
	\item\label{H1}  $\oomega_k^\eps\in W^{1,2}(\Omega;\RR^3)$;
	\item\label{H2} $\div (\oomega_k^\eps)=0$ in $\Omega$ and $\oomega_k^\eps=0$ in the holes $T_i^\eps$;
	\item\label{H3}$\oomega_k^\eps \rightharpoonup\ee_k$ weakly in $W^{1,2}(\Omega;\RR^3)$; % $\Grad\oomega_k^\eps\to 0$ strongly in $L^{\frac{6}{5}}(\Omega;\RR^3)$;
	\item\label{H4}$\mmu_k\in W^{-1,\infty}(\Omega;\RR^3)$;
	\item\label{H5} For all $\vv \in L^2(0,T;W^{1,2}(\Omega;\RR^3))$, $\vv_\eps \in L^2(0,T;W^{1,2}(\Omega_\eps;\RR^3))$ that satisfy \eqref{v1}--\eqref{v4}, there holds
	\begin{align*}
	\lim_{\eps\to 0} \int_0^\tau\int_{\Omega_\eps} \Grad \oomega_k^\eps:\Grad(\psi\vv_\eps)\dx\dt=\int_0^\tau\int_\Omega\mmu_k \psi\vv\dx\dt ,
	\end{align*}
	for every $\psi\in C_c^\infty((0,T)\times\Omega)$ and a.e. $\tau\in[0,T]$.
\end{enumerate}
Given the hypotheses \ref{H1}--\ref{H5}, we can show:
\begin{lem}\label{(M2)} Let $\Omega \subset \RR^3$ be a smoothly bounded domain, and let $\Omega_\eps$ be as in \eqref{Omega-holes}. Assume that \ref{M0} and \ref{H1}--\ref{H5} are satisfied for $((\Omega_\eps)_{\eps>0}, \Omega)$ and some $(\oomega_k^\eps,\mmu_k)_{1\leq k\leq 3}$.
Then for all $\vv \in L^2(0,T;W_0^{1,2}(\Omega;\RR^3))$, $\vv_\eps \in L^2(0,T;W_0^{1,2}(\Omega_\eps;\RR^3))$ that satisfy \eqref{v1}--\eqref{v4}, we have 
	\begin{align*}
		 \liminf_{\eps\to 0}\int_0^\tau\int_{\Omega_\eps}|\Grad \vv_\eps|^2\dx\dt
		 \ge 	\int_0^\tau\int_\Omega |\Grad\vv|^2\dx\dt+\int_0^\tau\int_\Omega (\bM \vv)\cdot \vv \dx\dt
		 \quad \forall \tau\in [0,T],
	\end{align*}
with $\bM_{ij}=\mu_i^j=\mmu_i \cdot \ee_j$.
\end{lem}

\begin{proof}
The proof we give here is the same for $d=3$ and $d \geq 3$, so let us prove the Lemma in the general setting with obvious changes on the assumptions if $d > 3$. Let $\pp=(\vp_1,...,\vp_d)\in C_c^\infty((0,T)\times\Omega;\RR^d)$. Then
\begin{align}\label{h1}
    0
    &\leq \int_0^\tau\int_\Omega |\Grad (\vv_\eps-\sum_{k=1}^d\vp_k\oomega_k^\eps)|^2\dx\dt\nonumber\\
    &=\int_0^\tau\int_\Omega |\Grad \vv_\eps|^2\dx\dt
    +\sum_{1\leq i,k\leq d}\left[\int_0^\tau\int_\Omega \vp_k\vp_i\Grad\oomega_k^\eps:\Grad \oomega_i^\eps \dx\dt+\int_0^\tau\int_\Omega(\Grad\vp_k\otimes\oomega_k^\eps):(\Grad\vp_i\otimes\oomega_i^\eps)\dx\dt\right]\nonumber\\
    & \quad\quad +2\sum_{1\leq i,k\leq d}\int_0^\tau\int_\Omega \vp_k\Grad \oomega_k^\eps :(\Grad\vp_i\otimes\oomega_i^\eps)\dx\dt
    -2\sum_{1\leq k\leq d}\int_0^\tau\int_\Omega \Grad
    \vv_\eps :(\Grad\vp_k\otimes\oomega_k^\eps)\dx\dt\nonumber\\
    & \quad\quad- 2\sum_{1\leq k\leq d}\int_0^\tau\int_\Omega \Grad\oomega_k^\eps :(\vp_k\Grad \vv_\eps)\dx\dt \nonumber\\
    &=\int_0^\tau\int_\Omega |\Grad \vv_\eps|^2\dx\dt
    +\sum_{1\leq i,k\leq d}\int_0^\tau\int_\Omega \Grad\oomega_k^\eps:\Grad (\vp_k\vp_i\oomega_i^\eps)\dx\dt\nonumber\\
    &\quad\quad-\sum_{1\leq i,k\leq d}\int_0^\tau\int_\Omega  \Grad\oomega_k^\eps: (\oomega_i^\eps \otimes\Grad(\vp_k\vp_i))\dx\dt
    +\sum_{1\leq i,k\leq d}\int_0^\tau\int_\Omega(\Grad\vp_k\otimes\oomega_k^\eps):(\Grad\vp_i\otimes\oomega_i^\eps)\dx\dt\nonumber\\
    & \quad\quad +2\sum_{1\leq i,k\leq d}\int_0^\tau\int_\Omega \vp_k\Grad \oomega_k^\eps :(\Grad\vp_i\otimes\oomega_i^\eps)\dx\dt
    -2\sum_{1\leq k\leq d}\int_0^\tau\int_\Omega \Grad
    \vv_\eps :(\Grad\vp_k\otimes\oomega_k^\eps)\dx\dt\nonumber\\
    & \quad\quad- 2\sum_{1\leq k\leq d}\int_0^\tau\int_\Omega \Grad\oomega_k^\eps :\Grad(\vp_k\vv_\eps)\dx\dt  
    +2\sum_{1\leq k\leq d}\int_0^\tau\int_\Omega \Grad\oomega_k^\eps :(\Grad \vp_k\otimes \vv_\eps)\dx\dt.
\end{align}
Now, we use \ref{H1}--\ref{H3} to infer that we can choose $\vv_\eps=\oomega_k^\eps$ in \ref{H5} to obtain
\begin{align*}
     \int_0^\tau\int_\Omega \Grad\oomega_k^\eps:\Grad (\vp_k\vp_i\oomega_i^\eps)\dx\dt~\to~ \int_0^\tau\int_\Omega \mmu_k\cdot(\vp_k\vp_i \ee_i)\dx\dt.
\end{align*}
Since $\oomega_k^\eps\rightharpoonup \ee_k$ weakly in $W^{1,2}(\Omega;\RR^d)$, we get $\oomega_k^\eps\to \ee_k$ strongly in $L^2(\Omega;\RR^d)$ such that
\begin{align*}
    \int_0^\tau\int_\Omega  \Grad\oomega_k^\eps: (\oomega_i^\eps \otimes\Grad(\vp_k\vp_i))\dx\dt~&\to~0,\\
    \int_0^\tau\int_\Omega(\Grad\vp_k\otimes\oomega_k^\eps):(\Grad\vp_i\otimes\oomega_i^\eps)\dx\dt~&\to~ \int_0^\tau\int_\Omega(\Grad\vp_k\otimes\ee_k):(\Grad\vp_i\otimes\ee_i)\dx\dt,\\
    \int_0^\tau\int_\Omega\vp_k\Grad\oomega_k^\eps:(\Grad\vp_i\otimes\oomega_i^\eps)\dx\dt~&\to~0.
\end{align*}
With $\vv_\eps\rightharpoonup \vv$ weakly in $L^2(0,T;W_0^{1,2}(\Omega;\RR^d))$ and $\oomega_k^\eps\to \ee_k$ strongly in $L^2(\Omega;\RR^d)$, we infer
\begin{align*}
    \int_0^\tau\int_\Omega \Grad
    \vv_\eps :(\Grad\vp_k\otimes\oomega_k^\eps)\dx\dt ~\to~ \int_0^\tau\int_\Omega \Grad
    \vv :(\Grad\vp_k\otimes\ee_k)\dx\dt.
\end{align*}
Due to \ref{H5}, we see that
\begin{align*}
    \int_0^\tau\int_\Omega \Grad\oomega_k^\eps :\Grad(\vp_k\vv_\eps)\dx\dt~\to~\int_0^\tau\int_\Omega \mmu_k\cdot(\vp_k\vv)\dx\dt,
\end{align*}
since $\vv_\eps\rightharpoonup\vv$ weakly in $L^2(0,T;L^6(\Omega;\RR^d))$. Moreover, boundedness of $\vv_\eps$ in $L^2(0,T; W_0^{1,2}(\Omega; \RR^d))$ and the fact that $\oomega_k^\eps \to \ee_k$ strongly in $L^2(\Omega; \RR^d)$ leads by partial integration to
\begin{align*}
    \int_0^\tau\int_\Omega \Grad\oomega_k^\eps :(\Grad \vp_k\otimes \vv_\eps)\dx\dt = - \int_0^\tau \int_\Omega (\oomega_k^\eps - \ee_k) \cdot \div(\Grad \vp_k \otimes \vv_\eps) \dx \dt \to 0.
\end{align*}

{Next, it is clear that there exists a subsequence, still denoted by $\vv_\eps$, such that}
\begin{align*} {\lim_{\eps\to0}\int_0^\tau\int_\Omega|\Grad \vv_\eps|^2\dx\dt=\liminf_{\eps\to 0}\int_0^\tau\int_\Omega|\Grad \vv_\eps|^2\dx\dt.}
\end{align*}
Passing to the limits in (\ref{h1}) leads to
\begin{align*}
    0
    &\leq\liminf_{\eps\to 0}\int_0^\tau\int_\Omega |\Grad \vv_\eps|^2\dx\dt
    +\sum_{1\leq k\leq d}\int_0^\tau\int_\Omega \mmu_k\cdot(\vp_k\pp) \dx\dt \\
    &\quad  +\int_0^\tau\int_\Omega|\Grad \pp|^2\dx\dt
    -2 \int_0^\tau\int_\Omega \Grad \vv:\Grad\pp\dx\dt
    - 2\sum_{1\leq k\leq d} \int_0^\tau\int_\Omega \mmu_k\cdot(\vp_k\vv)\dx\dt.
\end{align*}
Now we choose a sequence $\pp^{(n)}\in C_c^\infty((0,T)\times\Omega;\RR^d)$ with $\pp^{(n)}\to \vv$ strongly in $L^2(0,T;W_0^{1,2}(\Omega;\RR^d))$ and pass to the limit to obtain
\begin{align*}
    0
    &\leq\liminf_{\eps\to 0}\int_0^\tau\int_\Omega |\Grad \vv_\eps|^2\dx\dt
    +\int_0^\tau\int_\Omega (\bM\vv)\cdot\vv \dx\dt 
    +\int_0^\tau\int_\Omega|\Grad \vv|^2 \dx\dt\\
    &\quad -2\int_0^\tau\int_\Omega|\Grad \vv|^2 \dx\dt
    - 2\int_0^\tau\int_\Omega (\bM\vv)\cdot\vv\dx\dt.
\end{align*}
Rearranging leads to 
\begin{align*}
    \int_0^\tau\int_\Omega|\Grad \vv|^2\dx\dt+ \int_0^\tau\int_\Omega (\bM\vv)\cdot\vv\dx\dt
    &\leq\liminf_{\eps\to 0}\int_0^\tau\int_\Omega |\Grad \vv_\eps|^2\dx\dt.
\end{align*}
\end{proof}

\subsection{Validity of \ref{H1}--\ref{H5} in the random  setting} \label{sec:random}
Coming from the rather general class of domains \eqref{Omega-holes} back to randomly perforated domains, from now on let $(\Omega_\eps)_{\eps>0}$ be defined by \eqref{Omega_epsilon}. We then construct a family $(\oomega_k^\eps,{\mmu}_k)_{1\leq k\leq 3}$ verifying \ref{H1}--\ref{H5} under the assumption \eqref{MaFinal}. We start by collecting some useful  notation.\\

 For a Poisson point process $\Phi$ on $\RR^3$ and any bounded set $E\subset\RR^3$ that is star-shaped with respect to the origin, we define the random variables
\begin{equation*}
	\Phi(E):=\Phi\cap E,\quad \Phi^\eps(E):=\Phi\cap(\eps^{-1}E),\quad \mathcal{N}(E):=\#\Phi(E),\quad \mathcal{N}^\eps(E):=\#\Phi^\eps(E).
\end{equation*}
For each $\eta>0$ we define the thinned process as
\begin{equation*}
	\Phi_\eta:=\{x\in\Phi\colon\min_{y\in \Phi, y\ne x}|x-y|\ge\eta\},
\end{equation*}
and $\Phi_\eta(E)$, $\Phi_\eta^\eps(E)$, $\mathcal{N}_\eta(E)$, $\mathcal{N}^\eps_\eta(E)$ accordingly. Furthermore, we recall two lemmas (\cite[Lemmas~3.1 and 3.2]{GiuntiHoefer2019}) which guarantee that the holes $H^\eps$ can be decomposed as $H^\eps=H^\eps_{\rm g}\cup H^\eps_{\rm b}$, where
\begin{itemize}
	\item $H_{\rm g}^\eps$ contains the good holes, which are small and well separated, and
	\item $H_{\rm b}^\eps$ contains the bad holes, which are big and possibly overlapping.
\end{itemize}
\begin{lem}[{\cite[Lemma 3.1]{GiuntiHoefer2019}}]
	\label{giunti1}
	There exists $\delta=\delta(\beta)>0$ such that for almost every $\omega\in \tilde\Omega$ and all $\eps\leq \eps_0=\eps_0(\omega)$, there exists a partition $H^\eps=H^\eps_{\rm g}\cup H^\eps_{\rm b}$ and a set $\Omega^\eps_{\rm b}\subset\RR^3$ such that $H^\eps_{\rm b}\subset \Omega^\eps_{\rm b}$ and 
	\begin{align*}
		\dist(H^\eps_{\rm g};\Omega^\eps_{\rm b})>\eps^{1+\delta},\quad |\Omega^\eps_{\rm b}|\to 0.
	\end{align*}
	Furthermore, $H^\eps_{\rm g}$ is a union of disjoint balls centered in $n^\eps\subset \Phi^\eps(\Omega)$, namely
	\begin{align}\label{k1}
		H^\eps_{\rm g}=\bigcup_{z_i\in n^\eps}B_{\eps^{3}r_i}(\eps z_i),\quad \eps^3 \# n^\eps\to\lambda |\Omega|,\quad \min_{z_i\neq z_j\in n^\eps} \eps|z_i-z_j|\geq 2\eps^{1+\frac{\delta}{2}},\quad \eps^{3}r_i\leq \eps^{1+2\delta}.
	\end{align}
	Finally, for any $\eta>0$ we have
	\begin{align*}
		\lim_{\eps\to 0}\eps^3\#(\{ z_i\in \Phi^\eps_{2\eta}(\Omega)\mid \dist(\eps z_i,\Omega^\eps_{\rm b})\leq \eta \eps \})=0.
	\end{align*}
\end{lem}

We write $\cI^\eps:=\Phi^\eps(\Omega)\setminus n^\eps$, i.e., the set of centers of the balls in $H^\eps_{\rm b}$.

\begin{lem}[{\cite[Lemma 3.2]{GiuntiHoefer2019}}]\label{giunti2}
	Let $\theta>1$ be fixed. Then, for almost every $\omega\in\tilde\Omega$ and $\eps\leq \eps_0(\omega,\beta,\theta)$, we may choose $H^\eps_{\rm g}$, $H^\eps_{\rm b}$ of Lemma \ref{giunti1} in such a way that the following holds:
	\begin{itemize}
		\item There exist $\Lambda (\beta)>1$, a sub-collection $J^\eps\subset \cI^\eps$, and constants $\{ \lambda_l^\eps \}_{z_l\in J^\eps}\subset [1,\Lambda]$ such that
		\begin{align*}
			H_{\rm b}^\eps\subset \bar{H}_{\rm b}^\eps:=\bigcup_{z_j\in J^\eps}B_{\lambda_j^\eps\eps^{3}r_j}(\eps z_j),\quad \lambda_j^\eps \eps^{3}r_j\leq \Lambda \eps^{6\delta}.
		\end{align*}
		\item There exists $k_{\max}=k_{\max}(\beta)>0$ such that we may partition
		\begin{align*}
			\cI^\eps=\bigcup_{k=-3}^{k_{\max}}\cI_k^\eps,\qquad J^\eps=\bigcup_{k=-3}^{k_{\max}}J_k^\eps,
		\end{align*}
		with $J_k^\eps\subset \cI_k^\eps$ for all $k=1,...,k_{\max}$ and
		\begin{align*}
			\bigcup_{z_i\in\cI_k^\eps}B_{\eps^{3}r_i}(\eps z_i)\subset \bigcup_{z_i\in J_k^\eps}B_{\lambda_i^\eps\eps^{3}r_i}(\eps z_i);
		\end{align*}
		\item For all $k=-3,...,k_{\max}$ and every $z_i,z_j\in J_k^\eps$, $z_i\neq z_j$,
		\begin{align*}
			B_{\theta^2\lambda_i^\eps\eps^{3}r_i}(\eps z_i)\cap B_{\theta^2\lambda_j^\eps\eps^{3}r_j}(\eps z_j)= \emptyset;
		\end{align*}
		\item For each $k=-3,...,k_{\max}$ and $z_i\in\cI_k^\eps$, and for all $z_j\in\bigcup_{l=-3}^{k-1}J_l^\eps$ we have
		\begin{align*}
			B_{\eps^{3}r_i}(\eps z_i)\cap B_{\theta\lambda_j^\eps\eps^{3}r_j}(\eps z_j)=\emptyset.
		\end{align*}
	\end{itemize}
	Finally, the set $\Omega_{\rm b}^\eps$ of Lemma \ref{giunti1} may be chosen as
	\begin{align*}
		\Omega_{\rm b}^\eps=\bigcup_{z_i\in J^\eps}B_{\theta\lambda_i^\eps\eps^{3}r_i}(\eps z_i).\\[0.5cm]
	\end{align*}
\end{lem}

Let us moreover recall the Strong Law of Large Numbers (in our particular setting), which can be found, e.g., in \cite[Theorem~8.14]{Last2017}.
\begin{lem}[Strong Law of Large Numbers]\label{SLLN}
    Let $E\subset \RR^3$ be a measurable bounded set, and $(\Phi, \{r_i\}_{z_i \in \Phi})$ be a marked Poisson point process with
    \begin{align*}
        \langle r^m \rangle < \infty
    \end{align*}
    for some $m>0$. Then, almost surely,
    \begin{align*}
        \lim_{\eps \to 0} \eps^3 \cN(\eps^{-1} E) = \lambda |E|, && \lim_{\eps \to 0} \eps^3 \sum_{z_i \in \eps^{-1} E} r_i^m = \lambda \langle r^m \rangle |E|.
    \end{align*}
\end{lem}

In the following, we will extend each function which is defined on $\Omega_\eps$ by zero in $H^\eps$. 
For $r>0$ we set $	B_r:=B_r(0)\subset\RR^3$.
For each $z_i\in n^\eps$ we define
\begin{align*}
	a_{\eps,i}:=\eps^{3}r_i,\quad
	d_{\eps,i}:= \min\Big\{  \dist (\eps z_i,\Omega_{\rm b}^\eps),~\frac{1}{2}\min_{\substack{z_j\in n^\eps\\z_j\neq z_i}} (\eps |z_i-z_j|),~\eps  \Big\}.
\end{align*} 
Since $z_i\in n^\eps$, Lemma \ref{giunti1} guarantees the existence of some $\delta >0$ such that 
\begin{align}\label{grundabschätzung}
	a_{\eps,i}\leq \eps^{1+2\delta},\quad \eps^{1+\delta}\leq d_{\eps,i}\leq \eps.
\end{align}
For legibility, we drop the dependence on $\eps$ and set for $z_i \in n^\eps$ and $\eps>0$ small enough such that $a_{\eps, i} < \frac12 d_{\eps, i}$
\begin{align}
&	T_i:=B_{a_{\eps,i}}(\eps z_i),\quad C_i:=B_{\frac{1}{2}d_{\eps,i}}(\eps z_i)\setminus B_{a_{\eps,i}}(\eps z_i),\quad D_i:=B_{d_{\eps,i}}(\eps z_i)\setminus B_{\frac{1}{2}d_{\eps,i}}(\eps z_i), \notag \\[1em]
&A_i:=B_{d_{\eps,i}}(\eps z_i)\setminus B_{\frac{1}{4}d_{\eps,i}}(\eps z_i),\quad     E_i:=B_{a_{\eps,i}^{-1}d_{\eps,i}}(\eps z_i)\setminus B_{\frac{1}{2}a_{\eps,i}^{-1}d_{\eps,i}}(\eps z_i), \notag \\[1em]
&	B_{1,i}:=B_{\frac12 d_{\eps,i}}(\eps z_i),\quad B_{2,i}:=B_{d_{\eps,i}}(\eps z_i), \notag \\[1em]
&	A_i^0:=B_{d_{\eps,i}}(0)\setminus B_{\frac{1}{4}d_{\eps,i}}(0),\quad   E_i^0:=B_{a_{\eps,i}^{-1}d_{\eps,i}}(0)\setminus B_{\frac{1}{2}a_{\eps,i}^{-1}d_{\eps,i}}(0). \label{sets}
\end{align}
\begin{rmk}
	By \cite[Lemma C.1 and Lemma C.2]{GiuntiHoefer2019}, it follows that
	\begin{equation*}
		\lim_{\eps\to 0}\eps^3\sum_{z_i\in\Phi^\eps(\Omega)}r_i = \lambda\langle r \rangle|\Omega|\quad \text{ and } \quad \lim_{\eps\to 0}\eps^3\sum_{z_i\in
			\mathcal I_\eps}r_i=0 \quad \text{a.s.} 
	\end{equation*}
	Thus, we have a.s.
	\begin{equation}\label{limit-radii}
		\lim_{\eps\to 0}\sum_{z_i\in n^\eps} a_{\eps,i}= 	\lim_{\eps\to 0}\eps^3\sum_{z_i\in\Phi^\eps(\Omega)}r_i-  \lim_{\eps\to 0}\eps^3\sum_{z_i\in
			\mathcal	I_\eps}r_i= \lambda\langle r\rangle|\Omega|.
	\end{equation}
\end{rmk}

Let $\theta>1$ be fixed. Let $ J^\eps=\bigcup_{i=-3}^{k_{\max}}J^\eps_i$, $\{\lambda^\eps_l\}_{z_l\in  J^\eps}$ be given by Lemma \ref{giunti2}. For each $z_i\in J^\eps$ we define 
\begin{align}\label{def:sets}
	R_i:=\lambda_i^\eps r_i,\quad B_{R,i}:=B_{\eps^{3}R_i}(\eps z_i),\quad 
	B_{\theta,i}:=B_{\eps^{3}\theta R_i}(\eps z_i),\quad W_i:=B_{\theta,i}\setminus B_{R,i}.
\end{align}
For each $k\in\{1,2,3\}$ let also $(\oomega_k,q_k)$ be the unique weak solution to the Stokes problem
\begin{align*}
	\left\{ \begin{array}{ll}
		\Grad q_k -\Delta \oomega_k=0 & \text{in } \RR^3\setminus B_1, \\
		\div(\oomega_k)=0 & \text{in } \RR^3\setminus B_1,\\
		\oomega_k=0 & \text{on } \partial B_1,\\
		\oomega_k=\ee_k & \text{at infinity}.
	\end{array} \right.
\end{align*}
In the next lemma we  use $(\oomega_{k},q_k)$ to construct pairs of functions $(\oomega_{k}^{\eps, g},q_k^{\eps, g})_{\eps>0}$ in the spirit of  Allaire's functions in \cite{Allaire1990a}, which vanish on the good holes $H^\eps_{\rm g}$ and such that we have good control of their $L^p$ and $W^{1,p}$ norm, respectively. In particular, the functions $\oomega_{k}^{\eps, g}$ are the ones for which we will find $\mmu_k$ such that $(\oomega_{k}^{\eps, g},\mmu_k)_{1\le k\le 3}$ satisfy \ref{H1}--\ref{H5}.
	\begin{figure}
		\centering
	\begin{tikzpicture}
	\draw (0,0) circle (.5cm);
	\node at (0,0) {$T_i$};
	\draw (0,0) circle (1.5cm);
	\node at (-1,0) {$C_i$};
	\draw (0,0) circle (2.5cm);
	\node at (-2,0) {$D_i$};
	\end{tikzpicture}
	\caption{Cells for the construction of $(\oomega_k^\eps, q_k^\eps)$.}
\end{figure}
\begin{lem}\label{AllaireFunk}
	 For each $\eps>0$ and $k\in \{1,2,3\}$ we consider the pair of functions  $(\oomega_k^{\eps,g},q_k^{\eps,g})$ given by
	
	\begin{equation*}
	(\oomega_k^{\eps,g},q_k^{\eps,g}):=\begin{cases}
		(\ee_k,0)& \text{ in }  \Omega\setminus \bigcup_{z_i\in n^\eps}B_{2,i},\\[1em]
		(\tilde\oomega_k^{\eps,g}, \tilde q_k^{\eps,g})& \text{ in }  \bigcup_{z_i\in n^\eps}D_i,\\[1em]
		\left(\oomega_{k}\left(\frac{\cdot -\eps z_i}{a_{\eps,i}}\right), \frac{1}{a_{\eps,i}}q_k\left(\frac{\cdot -\eps z_i}{a_{\eps,i}}\right)\right)& \text{ in }  C_i \text{ for }z_i\in n^\eps,\\[1em]
		(0,0)&\text{ in }   \bigcup_{z_i\in n^\eps}T_i,
	\end{cases}
	\end{equation*}
	where $	(\tilde\oomega_k^{\eps,g}, \tilde q_k^{\eps,g})$ is the solution to the Stokes problem
	
	\begin{align*}
		\left\{ \begin{array}{ll}
			\Grad \tilde q_k^{\eps,g}-\Delta \tilde \oomega_k^{\eps,g}=0 &  \text{ in }  D_i,\\[1em]
			\div(\tilde \oomega_k^{\eps,g})=0& \text{ in }  D_i , \\[1em]
			(\tilde\oomega_k^{\eps,g}, \tilde q_k^{\eps,g})=(\ee_k,0) & \text{ on } \partial B_{2,i},\\[1em]
				(\tilde\oomega_k^{\eps,g}, \tilde q_k^{\eps,g})=	\left(\oomega_k\left(\frac{\cdot -\eps z_i}{a_{\eps,i}}\right), \frac{1}{a_{\eps,i}}q_k\left(\frac{\cdot -\eps z_i}{a_{\eps,i}}\right)\right) & \text{ on } \partial B_{1,i}.
		\end{array} \right.
	\end{align*}
	Then for all $p>\frac{3}{2}$ it holds
	\begin{align}
		\| \Grad q_k^{\eps,g}\|_{L^p(\cup_{z_i\in n^\eps} C_i;\RR^3)}&\lesssim \eps^{6(\frac{1}{p}-1)},\label{allaire2}\\
        \| \Grad\oomega_k^{\eps,g}\|_{L^p(\cup_{z_i\in n^\eps} C_i;\RR^{3 \times 3})}+\| q_k^{\eps,g}\|_{L^p(\cup_{z_i\in n^\eps} C_i)} &\lesssim  \left\{ \begin{array}{ll}
			\eps^{(1+2\delta)(2-p)}  & \text{if } p\leq 2 \\
			\eps^{3(2-p)}  & \text{if } p>2
		\end{array} ,\right.\label{allaire4}\\
\| \Grad\oomega_k^{\eps,g}\|_{L^p(\cup_{z_i\in n^\eps} A_i;\RR^{3 \times 3})}+\| q_k^{\eps,g}\|_{L^p(\cup_{z_i\in n^\eps} A_i)}&\lesssim \left\{ \begin{array}{ll}
	\eps^{\delta (1-\frac{1}{p})}\eps^{\frac{2}{p}-1} & \text{if } p\leq 2, \\[1em]
	\eps^{\delta (1-\frac{1}{p})}\eps^{-(1+\delta)(1-\frac{2}{p})} & \text{if } p>2.
\end{array} \right. ,\label{allaire3}\\
		\| \Grad\oomega_k^{\eps,g}\|_{L^p(\Omega;\RR^{3 \times 3})}+\| q_k^{\eps,g}\|_{L^p(\Omega)}
		&\lesssim \left\{ \begin{array}{ll}
			(\eps^{\delta (1-\frac{1}{p})}+\eps^{2\delta(\frac{2}{p}-1)})\eps^{\frac{2}{p}-1} & \text{if } p<2 ,\\[1em]
			\eps^{\delta (1-\frac{1}{p})}\eps^{-(1+\delta)(1-\frac{2}{p})}+\eps^{3(\frac{2}{p}-1)} & \text{if } p \geq 2.
		\end{array} \right.\label{allaire1}
	\end{align}
\end{lem}

\begin{proof}
	We use standard regularity theory for Stokes equations giving $|\Grad^{l+1}\oomega_k(x)|+|\Grad^l q_k(x)|\lesssim |x|^{-(l+2)}$, and $p>\frac{3}{2}$ to conclude
	\begin{align}
		\| \Grad q_k\|_{L^p(\RR^3\setminus B_1)}^p&\lesssim 1,\label{w1}\\
		\| \Grad\oomega_k\|^p_{L^p(\RR^3\setminus B_1; \RR^{3\times 3})} + \| q_k\|^p_{L^p(\RR^3\setminus B_1)}&\lesssim 1,\label{w2}\\
		\| \Grad\oomega_k\|_{L^p(E_i^0; \RR^{3 \times 3})}^p + \| q_k\|_{L^p(E_i^0)}^p&\lesssim \left( \frac{a_{\eps,i}}{d_{\eps, i}} \right)^{2p-3},\label{w3}\\
		\| \Grad\oomega_k\|_{L^p(B_{\frac{1}{2}d_{\eps,i}a_{\eps,i}^{-1}}\setminus B_{\frac{1}{4}d_{\eps,i}a_{\eps,i}^{-1}}; \RR^{3 \times 3})}^p +\| q_k\|_{L^p( B_{\frac{1}{2}d_{\eps,i}a_{\eps,i}^{-1}}\setminus B_{\frac{1}{4}d_{\eps,i}a_{\eps,i}^{-1}})}^p&\lesssim \left( \frac{a_{\eps,i}}{d_{\eps, i}} \right)^{2p-3} .\label{w4}       
	\end{align}
	Rescaling together with \eqref{w1}, $r_i^{-1}<R_0^{-1}\lesssim 1$, and \eqref{limit-radii} leads to
	\begin{align*}
		\| \Grad q_k^{\eps,g}\|^p_{L^p(\bigcup_{z_i \in n^\eps} C_i; \RR^3)}
		&=\sum_{z_i\in n^\eps} a_{\eps,i}^{3-2p}\| \Grad q_k \|^p_{L^p(B_{\frac{1}{2}d_{\eps,i}a_{\eps,i}^{-1}}\setminus B_1; \RR^3)}
		\lesssim \sum_{z_i\in n^\eps} a_{\eps,i}^{3-2p}\\
		&=\sum_{z_i\in n^\eps} a_{\eps,i}\eps^{6(1-p)} r_i^{2(1-p)}
		\lesssim \eps^{6(1-p)},
	\end{align*}
	and \eqref{allaire2} follows.
    
    We now prove \eqref{allaire4}, \eqref{allaire3}, and \eqref{allaire1}.
We start by noticing that, since $(\Grad \oomega_k^{\eps,g}, q_k^{\eps,g})=(0,0)$ in $[ \Omega\setminus (\cup_{z_i\in n^\eps}B_{2,i})] \cup (\cup_{z_i\in n^\eps}T_i)$, we have
\begin{align*}
	\| \Grad\oomega_k^{\eps,g}\|^p_{L^p(\Omega;\RR^{3 \times 3})}+\| q_k^{\eps,g}\|^p_{L^p(\Omega)}&=
\| \Grad\oomega_k^{\eps,g}\|^p_{L^p(\cup_{z_i\in n^\eps} C_i;\RR^{3 \times 3})}+\| q_k^{\eps,g}\|^p_{L^p(\cup_{z_i\in n^\eps} C_i)}	\\
&\qquad + \| \Grad\oomega_k^{\eps,g}\|^p_{L^p(\cup_{z_i\in n^\eps} D_i;\RR^{3 \times 3})}+\| q_k^{\eps,g}\|^p_{L^p(\cup_{z_i\in n^\eps} D_i)}.
\end{align*}

From \eqref{limit-radii}, \eqref{w2}, the fact that $a_{\eps,i}^{2-p}=(\eps^{3})^{(2-p)}(r_i)^{(2-p)}\le \eps^{(1+2\delta)(2-p)}$, and $r_i\ge R_0 > 0$, we find
\begin{equation}\label{allaire1-1}
\begin{split}
	\| \Grad \oomega_k^{\eps,g}\|^p_{L^p(\cup_{z_i\in n^\eps} C_i; \RR^{3 \times 3})} + \|  q_k^{\eps,g}\|^p_{L^p(\cup_{z_i\in n^\eps} C_i)}
 &= \sum_{z_i\in n^\eps} a_{\eps,i}^{3-p}(\| \Grad \oomega_k \|^p_{L^p(B_{\frac{1}{2}d_{\eps,i}a_{\eps,i}^{-1}}\setminus B_1; \RR^{3 \times 3})} +\|  q_k \|^p_{L^p(B_{\frac{1}{2}d_{\eps,i}a_{\eps,i}^{-1}}\setminus B_1)})\\
&\lesssim \sum_{z_i\in n^\eps} a_{\eps,i}^{3-p}
= \sum_{z_i\in n^\eps} a_{\eps,i} a_{\eps,i}^{2-p}
\lesssim \left\{ \begin{array}{ll}
	\eps^{(1+2\delta)(2-p)}  & \text{if } p \leq 2 \\
	\eps^{3(2-p)}  & \text{if } p>2
\end{array} .\right. 
\end{split}
\end{equation}
This proves \eqref{allaire4}.

Let now $(\vv_k^\eps,\pi_k^\eps)$ be the solution of the Stokes system
	\begin{align*}
		\left\{ \begin{array}{ll}
			\Grad \pi_k^\eps -\Delta \vv_k^\eps=0 & \text{ in } B_2\setminus B_1, \\
			\div(\vv_k^\eps)=0 & \text{ in } B_2\setminus B_1,\\
		{\vv_k^\eps=\oomega_k(\frac{1}{2}d_{\eps,i}a_{\eps,i}^{-1}\cdot)-\ee_k} & \text{ on } \partial B_1,\\
			\vv_k^\eps=0 & \text{ on }\partial B_2.
		\end{array} \right.
	\end{align*}

 Thus, we have that the tuple $(\vv_k^\eps + \ee_k, \, 2d_{\eps,i}^{-1}\pi_k^\eps)(2d_{\eps,i}^{-1}(\cdot -\eps z_i))$ solves the Stokes equation in $D_i$ with boundary data
	\begin{align*}
		\vv_k^\eps(2d_{\eps,i}^{-1}(\cdot-\eps z_i))+\ee_k &= \oomega_k\left(\frac{\cdot -\eps z_i}{a_{\eps,i}}\right) \text{ on }\partial B_{1,i},\\
		\vv_k^\eps(2d_{\eps,i}^{-1}(\cdot-\eps z_i))+\ee_k &= \ee_k \text{ on }\partial B_{2,i}.
	\end{align*}    
	Due to the uniqueness of the solution of the Stokes equation and by definition of $(\oomega_k^{\eps,g},q_k^{\eps,g})$ we get
	\begin{align*}
		\oomega_k^{\eps,g}-\ee_k&=\vv_k^\eps(2d_{\eps,i}^{-1}(\cdot -z_i))\quad\hspace{0.6cm}\text{ in }D_i,\\
		q_k^{\eps,g}&=\frac{2}{d_{\eps,i}}\pi_k^\eps(2d_{\eps,i}^{-1}(\cdot -z_i))\quad \text{ in }D_i.
	\end{align*}
	By rescaling and appealing to \cite[Theorem II.4.3 and Theorem IV.6.1]{Galdi2011}, it follows that
	\begin{equation}\label{eq1}
		\begin{split}
		&\| \Grad\oomega_k^{\eps,g} \|_{L^p(\cup_{z_i\in n^\eps} D_i; \RR^{3 \times 3})}^p + \| q_k^{\eps,g}\|_{L^p(\cup_{z_i\in n^\eps} D_i)}^p
		\lesssim \sum_{z_i\in n^\eps} d_{\eps,i}^{3-p}\| \Grad\vv_k^\eps\|_{L^p(B_2\setminus B_1; \RR^{3 \times 3})}^p+d_{\eps,i}^{3-p}\| \pi_k^\eps\|_{L^p(B_2\setminus B_1)}^p\\
		&\lesssim \sum_{z_i\in n^\eps} d_{\eps,i}^{3-p} \| \oomega_k(2^{-1}d_{\eps,i}a_{\eps,i}^{-1}\cdot)-\ee_k \|_{W^{1-\frac{1}{p},p}(\partial B_1; \RR^3)}^p \\
		&\lesssim \sum_{z_i\in n^\eps} d_{\eps,i}^{3-p} \| (\eta (\oomega_k-\ee_k ))(2^{-1}d_{\eps,i}a_{\eps,i}^{-1}\cdot)\|_{W^{1,p}(B_2\setminus B_1; \RR^3)}^p\\
		&=\sum_{z_i\in n^\eps}a_{\eps,i}^3 d_{\eps,i}^{-p}\left(\| \eta (\oomega_k-\ee_k)\|_{L^p(E_i^0; \RR^3)}^p +\left(\frac{1}{2}\frac{d_{\eps,i}}{a_{\eps,i}}\right)^p\| \Grad\eta(\oomega_k-\ee_k)+\eta\Grad\oomega_k \|_{L^p(E_i^0; \RR^{3 \times 3})}^p\right),
	\end{split}
	\end{equation}
	where $0\le\eta\le1$ is a cut-off function  in $B_{d_{\eps,i}a_{\eps,i}^{-1}} \setminus B_{\frac{1}{2}d_{\eps,i}a_{\eps,i}^{-1}}$  such that $\eta=0$ on $\partial B_{d_{\eps,i}a_{\eps,i}^{-1}}$, $\eta =1$ on $B_{\frac{1}{2}d_{\eps,i}a_{\eps,i}^{-1}}$, and $|\Grad\eta|\lesssim\frac{a_{\eps,i}}{d_{\eps,i}}$. Combining this with \eqref{w3} gives
\begin{equation}\label{eq2}
	\begin{split}
	&\| \eta (\oomega_k-\ee_k)\|_{L^p(E_i^0; \RR^3)}^p +\left(\frac{1}{2}\frac{d_{\eps,i}}{a_{\eps,i}}\right)^p\| \Grad\eta(\oomega_k-\ee_k)+\eta\Grad\oomega_k \|_{L^p(E_i^0; \RR^{3 \times 3})}^p\\
	&\lesssim \| \oomega_k-\ee_k\|_{L^p(E_i^0; \RR^3)}^p +\left(\frac{d_{\eps,i}}{a_{\eps,i}}\right)^p\| \Grad\oomega_k \|_{L^p(E_i^0; \RR^{3 \times 3})}^p+\left(\frac{1}{2}\frac{d_{\eps,i}}{a_{\eps,i}}\right)^p\| \Grad\eta\|_{L^\infty(E_i^0; \RR^3)}^p\| \oomega_k-\ee_k\|_{L^p(E_i^0)}^p\\
	&\lesssim \| \oomega_k-\ee_k\|_{L^p(E_i^0; \RR^3)}^p +\left(\frac{d_{\eps,i}}{a_{\eps,i}}\right)^p\left(\frac{d_{\eps,i}}{a_{\eps,i}}\right)^{3-2p}.
\end{split}    
\end{equation}
	Now we use $|\oomega_k(x)-\ee_k|\leq |x|^{-1}$ to infer
	\begin{align*}
		\|  \oomega_k-\ee_k\|_{L^p(E_i^0; \RR^3)}^p
		\lesssim \int_{E_i^0} \frac{1}{|x|^{p}}\dx
		\lesssim \left( \frac{d_{\eps,i}}{a_{\eps,i}} \right)^{3-p}.
	\end{align*}
	This together with \eqref{eq1} and \eqref{eq2} gives us
\begin{equation}\label{allaire1-2}
	\begin{split}
		\| \Grad\oomega_k^{\eps,g}\|_{L^p(\cup_{z_i\in n^\eps} D_i; \RR^{3 \times 3})}^p+\| q_k^{\eps,g}\|_{L^p(\cup_{z_i\in n^\eps} D_i)}^p
	&\lesssim \sum_{z_i\in n^\eps} d_{\eps,i}^{-p} a_{\eps,i}^3 \left( \frac{d_{\eps,i}}{a_{\eps,i}} \right)^{3-p}
	= \sum_{z_i\in n^\eps} a_{\eps,i} \left( \frac{a_{\eps,i}}{d_{\eps,i}} \right)^{p-1}d_{\eps,i}^{2-p} \\
	&\lesssim \left\{ \begin{array}{ll}
		\eps^{\delta p-1}\eps^{2-p} & \text{if } p \leq 2 \\
		\eps^{\delta p-1}\eps^{-(1+\delta)(p-2)} & \text{if } p>2
	\end{array} ,\right.
	\end{split}
\end{equation}

where the last inequality follows from \eqref{limit-radii} and the fact that by \eqref{grundabschätzung}, we have
\begin{equation*}
\left( \frac{a_{\eps,i}}{d_{\eps,i}} \right)^{p-1}d_{\eps,i}^{2-p} \lesssim  \frac{\eps^{(1+2\delta)(p-1)}}{\eps^{(\delta+1)(p-1)}}\eps^{2-p} 
\end{equation*}
for $p>2$. Combining \eqref{allaire1-1} with \eqref{allaire1-2} we find \eqref{allaire1}. To prove \eqref{allaire3}, we observe that by rescaling and \eqref{w4} we get
	\begin{align*}
		&\| \Grad\oomega_k^{\eps,g}\|_{L^p(\bigcup_{z_i \in n^\eps} B_{\frac{1}{2}d_{\eps,i}}(\eps z_i)\setminus B_{\frac{1}{4}d_{\eps,i}}(\eps z_i); \RR^{3 \times 3})}^p+\| q_k^{\eps,g}\|_{L^p(\bigcup_{z_i \in n^\eps} B_{\frac{1}{2}d_{\eps,i}}(\eps z_i)\setminus B_{\frac{1}{4}d_{\eps,i}}(\eps z_i))}^p\\
		&\lesssim \sum_{z_i\in n^\eps} a_{\eps,i}^{3-p}\Big(\| \Grad\oomega_k\|_{L^p(B_{\frac{1}{2}d_{\eps,i}a_{\eps,i}^{-1}}\setminus B_{\frac{1}{4}d_{\eps,i}a_{\eps,i}^{-1}}; \RR^{3 \times 3})}^p +\| q_k\|_{L^p(B_{\frac{1}{2}d_{\eps,i}a_{\eps,i}^{-1}}\setminus B_{\frac{1}{4}d_{\eps,i}a_{\eps,i}^{-1}})}^p\Big)\\
		&\lesssim 
		\sum_{z_i\in n^\eps} a_{\eps,i}^{3-p}\left( \frac{d_{\eps,i}}{a_{\eps,i}} \right)^{3-2p}
		\lesssim \sum_{z_i\in n^\eps} a_{\eps,i}\left( \frac{a_{\eps,i}}{d_{\eps,i}} \right)^{p-1}d_{\eps,i}^{2-p}\\
		&\lesssim \left\{ \begin{array}{ll}
			\eps^{\delta (p-1)}\eps^{2-p} & \text{if } p\leq 2 \\
			\eps^{\delta (p-1)}\eps^{-(1+\delta)(p-2)} & \text{if } p>2
		\end{array} .\right.
	\end{align*}
	This together with \eqref{allaire1-2} and $$\bigcup_{z_i\in n^\eps} A_i=\Big(\bigcup_{z_i\in n^\eps} D_i\Big)\cup \Big(\bigcup_{z_i\in n^\eps} B_{\frac{1}{2}d_{\eps,i}}(\eps z_i)\setminus B_{\frac{1}{4}d_{\eps,i}}(\eps z_i)\Big)$$ implies
	\begin{align*}
		\| \Grad\oomega_k^{\eps,g}\|_{L^p(\cup_{z_i\in n^\eps} A_i; \RR^{3 \times 3})}^p+\| q_k^{\eps,g}\|_{L^p(\cup_{z_i\in n^\eps} A_i)}^p
		\lesssim \left\{ \begin{array}{ll}
			\eps^{\delta (p-1)}\eps^{2-p} & \text{if } p\leq 2 \\
			\eps^{\delta (p-1)}\eps^{-(1+\delta)(p-2)} & \text{if } p>2
		\end{array} .\right.
	\end{align*}
\end{proof}
	
	\begin{defin}\label{def:Nfunctions}
 For $N\in\mathbb N$	we define
		\begin{align*}
			n_N^\eps:=\left\{ z_i\in n^\eps\mid d_{\eps,i}\geq \frac{\eps}{N} \right\},\qquad
			r_{i,N}:=\min\{ r_{i},N \},\qquad
			\cR_N:=\{ r_{i, N} \}_{z_i\in n_N^\eps},
		\end{align*}
		and we let $(\oomega_{k,N}^{\eps,g},q_{k,N}^{\eps,g})$ be the analogue of $(\oomega_k^{\eps,g},q_k^{\eps,g})$ when $n^\eps$ is replaced by $n_N^\eps$.
	\end{defin}
	Note that by definition we have for any $z_i \in n_N^\eps$ that
	\begin{align}\label{x1}
		a_{\eps,i}\leq N\eps^{3} \quad\text{and}\quad
		d_{\eps,i}^{-1}\leq \frac{N}{\eps}.
	\end{align}
\begin{rmk}\label{AllaireFunk2}
	From $n_N^\eps\subset n^\eps$ and the definition of $(\oomega_{k,N}^{\eps,g},q_{k,N}^{\eps,g})$ and $(\oomega_k^{\eps,g},q_k^{\eps,g})$, it readily follows that for all $p>\frac{3}{2}$ we have uniformly in $N$
\begin{align}
	\| \Grad q_{k,N}^{\eps,g}\|_{L^p(\cup_{z_i \in n_N^\eps} C_i; \RR^3)} \lesssim \eps^{6(\frac{1}{p}-1)},\quad \| q_{k,N}^{\eps,g}\|_{L^p(\cup_{z_i \in n_N^\eps} C_i)} &\lesssim  \left\{ \begin{array}{ll}
		\eps^{(1+2\delta)(2-p)}  & \text{if } p\leq 2 \\
		\eps^{3(2-p)}  & \text{if } p>2
	\end{array} .\right.\label{allaire5} \\
	\| \Grad\oomega_{k,N}^{\eps,g}\|_{L^p(\cup_{z_i \in n_N^\eps} A_i; \RR^{3 \times 3})} + \| q_{k,N}^{\eps,g}\|_{L^p(\cup_{z_i \in n_N^\eps} A_i)} &\lesssim \left\{ \begin{array}{ll}
		\eps^{\delta (1-\frac{1}{p})}\eps^{\frac{2}{p}-1} & \text{if } p\leq 2 ,\\
		\eps^{\delta (1-\frac{1}{p})}\eps^{-(1+\delta)(1-\frac{2}{p})} & \text{if } p>2.
	\end{array} \right.\label{allaire6} \\
	\| \Grad\oomega_{k,N}^{\eps,g}\|_{L^p(\Omega; \RR^{3 \times 3})} + \| q_{k,N}^{\eps,g}\|_{L^p(\Omega)}
	&\lesssim \left\{ \begin{array}{ll}
			(\eps^{\delta (1-\frac{1}{p})}+\eps^{2\delta(\frac{2}{p}-1)})\eps^{\frac{2}{p}-1} & \text{if } p<2 ,\\[1em]
			\eps^{\delta (1-\frac{1}{p})}\eps^{-(1+\delta)(1-\frac{2}{p})}+\eps^{3(\frac{2}{p}-1)} & \text{if } p \geq 2.
		\end{array} \right.\label{allaire7}
\end{align}
\end{rmk}
    Before we show that the functions $(\oomega_k^{\eps, g})_{1 \leq k \leq 3}$ from Lemma~\ref{AllaireFunk} satisfy \ref{H5}, we need some preliminary lemmas, which we collect in the following:
	\begin{lem}\label{reihen}
		Let $c\in(0,1]$ be a fixed real number, and $\delta_i^{c d_{\eps,i}}$ be the measure concentrated on the sphere $\partial B_{c d_{\eps,i}}(\eps z_i)$, that is, $\delta_i^{cd_{\eps,i}}=  \mathcal H^2\res\partial B_{cd_{\eps,i}}(\eps z_i)$. Let also $\sigma_3$ be the area of the unit sphere in $\RR^3$. Then, almost surely,
		\begin{align}
			\sum_{z_i\in n^\eps_N}(cd_{\eps,i})^{-2}a_{\eps,i} \delta_i^{cd_{\eps,i}} &\to \sigma_3\langle \mathcal N_{\frac{2}{N}}( \Omega)\rangle   \langle r_N \rangle \quad\text{strongly in } W^{-1,2}(\Omega), \label{series1} \\
            \sum_{z_i\in n^\eps_N}(cd_{\eps,i})^{-2}a_{\eps,i} \delta_i^{cd_{\eps,i}}(\ee_k\cdot \ee^i)\ee^i &\to \frac{\sigma_3}{3}\langle \mathcal N_{\frac{2}{N}}(\Omega)\rangle   \langle r_N \rangle\ee_k\quad\text{strongly in } W^{-1,2}(\Omega), \label{series2}
		\end{align}
	%
	% \begin{equation}\label{series2}
	% 	\sum_{z_i\in n^\eps_N}(cd_{\eps,i})^{-(d-1)}a_{\eps,i}\delta_i^{cd_{\eps,i}}(\ee_k\cdot \ee^i)\ee^i \to \frac{\sigma_d}{d}\langle \mathcal N_{\frac{2}{N}}(\Omega)\rangle   \langle r_N\rangle\ee_k\quad\text{strongly in } W^{-1,2}(\Omega),
	% \end{equation}
	where $\ee^i\colon x\mapsto\frac{x-\eps z_i}{|x-\eps z_i|}$. 
	\end{lem}
	\begin{proof}
    To prove Lemma \ref{reihen} we argue as in \cite[Lemma 2.3]{CioranescuMurat1982} and \cite[Lemma II.3.5]{Allaire1989PhD}.
    
		\item\paragraph{Step 1: Proof of \eqref{series1}.}
	Let $p_i^\eps\colon\overline{B_{cd_{\eps,i}}(\eps z_i)}\to\RR$ be the solution of
		\begin{align*}
			\left\{ \begin{array}{ll}
				-\Delta p_i^\eps =-3(cd_{\eps,i})^{-3}a_{\eps,i} & \text{ in  }B_{cd_{\eps,i}}(\eps z_i), \ z_i\in n^\eps ,\\[1em]
				\dfrac{\partial p_i^\eps}{\partial n} = (cd_{\eps,i})^{-2}a_{\eps,i}& \text{ on }\partial B_{cd_{\eps,i}}(\eps z_i).
			\end{array} \right.
		\end{align*}
		Then  $p^\eps$  satisfies
		\begin{align*}
			\int_{B_{cd_{\eps,i}}(\eps z_i)}-\Delta p_i^\eps \dx =-a_{\eps,i} \sigma_3 =- \int_{\partial B_{cd_{\eps,i}}(\eps z_i)}\frac{\partial p_i^\eps}{\partial n} {\rm d}\mathcal H^2.
		\end{align*}
		Hence, there exists a unique solution with $p^\eps=0$ on $\partial B_{cd_{\eps,i}}(\eps z_i)$, given by
		\begin{equation*}
			p_i^\eps(x)=3(cd_{\eps,i})^{-3} a_{\eps,i} \left(\frac{s^2}{2}- \frac{c^2 d^2_{\eps,i}}{2}\right)\quad \text{ for }x\in \overline{B_{cd_{\eps,i}}(\eps z_i)},
		\end{equation*}
		where $s=|x-\eps z_i|$.
		We have 
		\begin{align}\label{star1}
			\frac{\partial p_i^\eps}{\partial s}= 3(cd_{\eps,i})^{-3} a_{\eps,i} s \quad \text{in }B_{cd_{\eps,i}}(\eps z_i).
		\end{align}
	We define the function $p_N^\eps \colon \Omega\to\RR$ via 
	\begin{equation*}
		p_N^\eps(x):=\begin{cases}
p_i^\eps (x) & \text{ if } x\in B_{cd_{\eps,i}}(\eps z_i), z_i \in n^\eps_N ,\\
0& \text{ if } x\in \Omega\setminus \bigcup_{z_i \in n^\eps_N}B_{cd_{\eps,i}}(\eps z_i).
		\end{cases}
	\end{equation*}
Hence, from \eqref{x1}, \eqref{star1}, and the definition of $n_N^\eps$ and $r_{i, N}$ it follows that
		\begin{align*}
			|\Grad p_N^\eps|\leq 3C  (cd_{\eps,i})^{-2}a_{\eps,i} \leq C_N \eps.
		\end{align*}
As a consequence we have for any fixed $N \in \NN$
		\begin{align}\nonumber
			p_N^\eps\to 0 \quad&\text{strongly in }W^{1,\infty}(\Omega),\\\label{n10}
   \Delta p_N^\eps\to 0 \quad&\text{strongly in }W^{-1,\infty}(\Omega).
		\end{align}
		By definition we have in the sense of $W^{-1,\infty}(\Omega)$
		\begin{align*}
			-\Delta p_N^\eps = - \sum_{z_i\in n_N^\eps}3c^{-3}d_{\eps,i}^{-3} a_{\eps,i} \chi_{B_{cd_{\eps,i}}(\eps z_i)} -\sum_{z_i\in n_N^\eps}c^{-2}d_{\eps,i}^{-2} a_{\eps,i} \delta^{cd_{\eps,i}}.
		\end{align*}    
Let now
		\begin{align*}
			\eta_N^\eps&=\sum_{z_i\in n_N^\eps}c^{-3}d_{\eps,i}^{-3}\eps^3 r_{i, N} \chi_{B_{cd_{\eps,i}}(\eps z_i)},\\
			\tilde\eta_N^\eps&=\sum_{z_i\in \Phi_{\frac2N}^\eps}c^{-3}d_{\eps,i}^{-3}\eps^3 r_{i, N} \chi_{B_{cd_{\eps,i}}(\eps z_i)}.
		\end{align*}
		By \cite[Lemma 5.3]{GiuntiHoeferVelazquez2018}, it follows that for every $\xi\in C^1_0(\Omega)$ almost surely
		\begin{align*}
		\lim_{\eps\to 0}\int_\Omega \tilde\eta^\eps_N \xi\dx
&	= \lim_{\eps\to 0} \sum_{z_i\in \Phi_{\frac{2}{N}}^\eps}\left(\frac{\eps}{cd_{\eps,i}} \right)^3 r_{i, N} \int_{B_{cd_{\eps,i}}(\eps z_i)}\xi\dx\\
&	=\frac{\sigma_3}{3}\langle \mathcal N_{\frac{2}{N}}(\Omega)\rangle   \langle r_N \rangle \int_{\Omega}\xi\dx.
		\end{align*}
		Thus, we proved that 
		\begin{align}\label{n1}
			\tilde\eta_N^\eps ~\xrightharpoonup{\ast}~ \frac{\sigma_3}{3}\langle \mathcal N_{\frac{2}{N}}(\Omega)\rangle   \langle r_N \rangle \quad \text{in } L^\infty(\Omega)\text{ almost surely.}
		\end{align}
		Furthermore,
		\begin{align}\label{n2}
			\tilde\eta_N^\eps-\eta_N^\eps ~\xrightharpoonup{\ast}~ 0 \quad \text{in } L^\infty(\Omega)\text{ almost surely.}
		\end{align}
		Indeed, we have $n_N^\eps\subset \Phi_{\frac{2}{N}}^\eps$ and $\dist(\eps z_i,\Omega_b^\eps)<\frac{\eps}{N}$ for each $z_i\in \Phi_{\frac{2}{N}}^\eps\setminus n_N^\eps$. Let $\xi\in C_0^1(\Omega)$, then we get
		\begin{align*}
			\left| \int_\Omega (\tilde\eta_N^\eps-\eta_N^\eps) \xi \dx\right|
			&\lesssim \sum_{z_i\in \Phi_{\frac{2}{N}}^\eps\setminus n_N^\eps}c^{-3}d_{\eps,i}^{-3}\eps^3 r_{i, N} \int_{B_{cd_{\eps,i}}(\eps z_i)}\xi\dx
			\lesssim \sum_{z_i\in \Phi_{\frac{2}{N}}^\eps\setminus n_N^\eps}\| \xi\|_{L^\infty(\Omega)}\eps^3 N\\
			&\lesssim N\eps^3 \#\{z_i\in\Phi_{\frac{2}{N}}^\eps\mid \dist(\eps z_i,\Omega_b^\eps)\leq \frac{\eps}{N}\}.
		\end{align*}
		Lemma~3.1 in \cite{GiuntiHoefer2019} implies that the right-hand side vanishes as $\eps\to 0$. By 
		\eqref{n1} and \eqref{n2}, we deduce
		\begin{align}\label{n3}
			\eta_N^\eps ~\xrightharpoonup{\ast}~ \frac{\sigma_3}{3}\langle \mathcal N_{\frac{2}{N}}(\Omega)\rangle   \langle r_N \rangle \quad \text{in } L^\infty(\Omega)\text{ almost surely and hence strongly in } W^{-1,\infty}(\Omega).
		\end{align}
		By the strong convergences in \eqref{n10}, we get for fixed $N\in\NN$
		\begin{align*}
			\sum_{z_i\in n_N^\eps}c^{-2}d_{\eps,i}^{-2}a_{\eps,i}\delta^{cd_{\eps,i}} ~\to ~\sigma_3\langle\mathcal N_{\frac{2}{N}}(\Omega)\rangle   \langle r_N\rangle 
			\quad \text{in } W^{-1,\infty}(\Omega)\text{ strongly almost surely,}
		\end{align*}
        which shows \eqref{series1}.\\

\item\paragraph{Step 2: Proof of \eqref{series2}.} 
	Let $\p_{i,k}^\eps:\overline{B_{cd_{\eps,i}}(\eps z_i)}\to\RR^3$ be the solution of
		\begin{align*}
			\left\{ \begin{array}{ll}
				-\Delta \p_{i,k}^\eps =-(cd_{\eps,i})^{-3}a_{\eps,i}\ee_k & \text{ in  }B_{cd_{\eps,i}}(\eps z_i),\\[1em]
				\dfrac{\partial \p_{i,k}^\eps}{\partial n} = (cd_{\eps,i})^{-2}a_{\eps,i}(\ee_k\cdot\ee^i)\ee^i& \text{ on }\partial B_{cd_{\eps,i}}(\eps z_i).
			\end{array} \right.
		\end{align*}
		Then $\p_{i,k}^\eps$ satisfies
		\begin{align*}
			\int_{B_{cd_{\eps,i}}(\eps z_i)}-\Delta \p_{i,k}^\eps \dx =-a_{\eps,i}\frac{\sigma_3}{3}\ee_k =- \int_{\partial B_{cd_{\eps,i}}(\eps z_i)}\frac{\partial \p_{i,k}^\eps}{\partial n}\ds.
		\end{align*}
		Thus, there exists a unique solution, up to an additive constant, given by
		\begin{align*}        
			\p_{i,k}^\eps(x)=\frac{1}{2}(cd_{\eps,i})^{-3}a_{\eps,i}(x-\eps z_i)_kr\ee^i(x)=\frac{1}{2}(cd_{\eps,i})^{-3}a_{\eps,i} s^2 (\ee_k\cdot\ee^i) \ee^i
			\quad \text{ for } x\in \overline{B_{cd_{\eps,i}}(\eps z_i)},
		\end{align*}
where $s=|x-\eps z_i|$. We get
		\begin{align*}
			\frac{\partial \p_{i,k}^\eps}{\partial s}=(cd_{\eps,i})^{-3}a_{\eps,i} s (\ee_k\cdot \ee^i)\ee^i=(cd_{\eps,i})^{-2}a_{\eps,i}(\ee_k\cdot \ee^i)\ee^i\quad \text{ on }\partial B_{cd_{\eps,i}}(\eps z_i).
		\end{align*}

        Next, we are searching a function $\ss_{i,k}^\eps:B_{\eps d_{\eps,i}}(\eps z_i)\to \RR^3$ satisfying
		\begin{align}\label{n8}
			\begin{array}{cc}
				\p_{i,k}^\eps=\ss_{i,k}^\eps  & \text{on }\partial B_{cd_{\eps,i}}(\eps z_i), \\[1em]
				\dfrac{\partial \ss_{i,k}^\eps}{\partial n} =0 & \text{on }\partial B_{cd_{\eps,i}}(\eps z_i).
			\end{array}
		\end{align}
        
		The unique solution of this system is given by
		\begin{align*}
			\ss_{i,k}^\eps=(cd_{\eps,i})^{-3}a_{\eps,i}\frac{s^2(3cd_{\eps,i}-2s)}{2cd_{\eps,i}}(\ee_k\cdot \ee^i)\ee^i.
		\end{align*}
		
		For its Laplacian we compute
		\begin{align*}
			\Delta \ss_{i,k}^\eps=((3cd_{\eps,i}-2s)\ee_k-6s(\ee_k\cdot\ee^i)\ee^i)c^{-4}d_{\eps,i}^{-4}a_{\eps,i}.
		\end{align*}
		Hence, from \eqref{x1} we conclude
		\begin{align}\label{n9}
			|\Delta \ss_{i,k}^\eps|\leq 9c^{-3}d_{\eps,i}^{-3}a_{\eps,i}\lesssim C_N.
		\end{align}  
	Let $\rr_{N,k}^\eps:\Omega\to \RR^3$ be given by
		\begin{align*}
			\rr_{N,k}^\eps:=
			\left\{ \begin{array}{ll}
				\p_{i,k}^\eps-\ss_{i,k}^\eps & \text{in }  B_{cd_{\eps,i}}(\eps z_i), \ z_i \in n_N^\eps, \\[1em]
				0 & \text{in }\Omega\setminus \bigcup_{i\in n^\eps_N} B_{cd_{\eps,i}}(\eps z_i),
			\end{array}\right.
		\end{align*}
in other words, 
		\begin{align*}                              
			\rr_{N,k}^\eps(x) = s^2c^{-3}d_{\eps,i}^{-3}a_{\eps,i}\left( \frac{r}{cd_{\eps,i}}-1 \right)(\ee_k\cdot \ee^i)\ee^i
			\quad\text{in } B_{cd_{\eps,i}}(\eps z_i).
		\end{align*}
		Using \eqref{x1} we get
		\begin{align*}
			|\Grad \rr_{N,k}^\eps|\lesssim d_{\eps,i}^{-2}a_{\eps,i}\lesssim C_N \eps.
		\end{align*}
As a consequence we have for any fixed $N\in \NN$
		\begin{equation}\label{n4}
			\begin{split}
		\rr_{N,k}^\eps&\to 0 \quad\text{strongly in }W^{1,\infty}(\Omega;\RR^3), \\
	\Delta\rr_{N,k}^\eps&\to 0 \quad\text{strongly in }W^{-1,\infty}(\Omega;\RR^3).
			\end{split}
		\end{equation}
		Moreover, by definition we find in the sense of $W^{-1, \infty}(\Omega; \RR^3)$
		\begin{align}\label{n7}
			-\Delta\rr_{N,k}^\eps=
			-\eta_N^\eps\ee_k 
			-\sum_{z_i\in n_N^\eps}\chi_{B_{cd_{\eps,i}}(\eps z_i)}^\eps\Delta \ss_{i,k}^\eps+\sum_{z_i\in n_N^\eps}c^{-2}d_{\eps,i}^{-2}a_{\eps,i}\delta^{cd_{\eps,i}}(\ee_k\cdot \ee^i)\ee^i.
		\end{align} 
		With \eqref{n8} and \eqref{n9}, we have $\sum_{z_i\in n_N^\eps}\chi_{B_{cd_{\eps,i}}(\eps z_i)}^\eps\Delta \ss_{i,k}^\eps\in L^\infty(\Omega;\RR^3)$ uniformly in $\eps > 0$ and
		\begin{align*}
			\int_\Omega \sum_{z_i\in n_N^\eps}\chi_{B_{cd_{\eps,i}}(\eps z_i)}^\eps\Delta \ss_{i,k}^\eps\dx
			=\sum_{z_i\in n_N^\eps}\int_{B_{cd_{\eps,i}}(\eps z_i)} \Delta \ss_{i,k}^\eps\dx
			=\sum_{z_i\in n_N^\eps}\int_{\partial B_{cd_{\eps,i}}(\eps z_i)}\frac{\partial \ss_{i,k}^\eps}{\partial n}=0.
		\end{align*}
		Thus,
		\begin{align}\label{n5}
			\sum_{z_i\in n_N^\eps}\chi_{B_{cd_{\eps,i}}(\eps z_i)}^\eps\Delta \ss_{i,k}^\eps& ~\xrightharpoonup{\ast}~0\quad\text{in }L^\infty(\Omega;\RR^3) \text{ almost surely},\nonumber\\
			\sum_{z_i\in n_N^\eps}\chi_{B_{cd_{\eps,i}}(\eps z_i)}^\eps\Delta \ss_{i,k}^\eps&\to 0\quad\text{strongly in } W^{-1,\infty}(\Omega;\RR^3).
		\end{align}
		We conclude \eqref{series2} with \eqref{n3}, \eqref{n4}, \eqref{n7}, and \eqref{n5} since a.s.
		\begin{align*}
			\sum_{z_i\in n_N^\eps}c^{-2}d_{\eps,i}^{-2}a_{\eps,i}\delta^{cd_{\eps,i}}(\ee_k\cdot \ee^i)\ee^i~&\to~\frac{\sigma_3}{3}\langle \cN_{\frac{2}{N}}(\Omega)\rangle   \langle r_N\rangle\ee_k\quad\text{strongly in }W^{-1,\infty}(\Omega;\RR^3).
		\end{align*}
		\end{proof}
\begin{lem}\label{divergenzterm}
	Assume that ${\rm Ma}(\eps)$ obeys \eqref{MaFinal}.  
	Let $\vv_\eps \in L^2(0,T;W_0^{1,2}(\Omega_\eps;\RR^3))$ be such that \eqref{v1}--\eqref{v4} hold for some $\vv \in L^2(0,T;W_0^{1,2}(\Omega;\RR^3))$.
	Then, for every $N\in\NN$ and every $\psi \in C_c^\infty((0,T)\times \Omega)$, we have 
	\begin{align*}
	\lim_{\eps\to 0}	\int_0^\tau\int_\Omega q_{k,N}^{\eps,g} \psi  \div(\vv_\eps)\dx\dt=0\quad\text{ for a.e. } \tau\in [0,T].
	\end{align*}
\end{lem}

\begin{proof}   
	Let $\xi\in C_c^\infty(\RR^3)$ be a cut-off function such that $\xi\equiv1$ in $B_{\frac14}$ and $\xi\equiv0$ in $\RR^3\setminus {B_{\frac12}}$.
	Then, we define
	\begin{align}\label{cutofffunkrandom}
		\xi_\eps (x):= 	\left\{ \begin{array}{ll}
		\xi (\frac{x-\eps z_i}{d_{\eps,i}}) &\text{ in }B_{1,i}, \ z_i \in n_N^\eps, \\
		 0 & \text{ in }\Omega\setminus (\bigcup_{z_i\in n^\eps_N}B_{1,i}),
		\end{array} \right.
	\end{align}
	hence $|\Grad \xi_\eps| \lesssim d_{\eps,i}^{-1}$. By the assumed form of $\div \vv_\eps$ in \eqref{v1} we may write 
	\begin{align}\label{divterm}
		\int_0^\tau\int_\Omega q_{k,N}^{\eps,g} \div(\vv_\eps) \psi \dx\dt
		&= \int_0^\tau\int_\Omega q_{k,N}^{\eps,g} \xi_\eps \div(\vv_\eps) \psi\dx\dt +\int_0^\tau\int_\Omega q_{k,N}^{\eps,g} (1-\xi_\eps) \div( \vv_\eps )\psi \dx\dt\nonumber\\
		&=-\int_0^\tau\int_\Omega \Grad(q_{k,N}^{\eps,g} \xi_\eps \psi )\hh_\eps\dx\dt -\int_0^\tau\int_\Omega  q_{k,N}^{\eps,g} \xi_\eps  g_\eps\partial_t\psi  \dx\dt\\
        &\hphantom{=}+\int_\Omega q_{k,N}^{\eps,g}\xi_\eps g_\eps(\tau,\cdot)\psi(\tau,\cdot)\dx
        +\int_0^\tau\int_\Omega  q_{k,N}^{\eps,g} (1-\xi_\eps) \div( \vv_\eps )\psi\dx\dt. \nonumber
	\end{align}
Note that \eqref{MaFinal} in particular implies $\lim_{\eps\to 0}{\rm Ma}(\eps)=0$, hence from \eqref{v1} we have ${\| \div (\vv_\eps)\|_{L^2(0,T;L^2(\Omega))}\lesssim 1}$. This together with \eqref{allaire6} yields
	\begin{align*}
		\Bigg| \int_0^\tau\int_\Omega q_{k,N}^{\eps,g} (1-\xi_\eps) \div( \vv_\eps )\psi \dx\dt \Bigg| 
		&\lesssim \int_0^\tau\int_{\Omega\setminus \bigcup_{z_i\in n^\eps_N} B_{\frac{1}{4}d_{\eps,i}}(\eps z_i)}  |q_{k,N}^{\eps,g} ||\div( \vv_\eps )|\dx\dt\\
		&\lesssim \| q_{k,N}^{\eps,g}\|_{L^2(\bigcup_{z_i\in n^\eps_N} A_i)} \| \div (\vv_\eps)\|_{L^2(0,T;L^2(\Omega))}
		\lesssim \eps^{\frac{\delta}{2}},
	\end{align*}
	which in turn implies that the last term on the right-hand side of \eqref{divterm} vanishes as $\eps\to 0$.
    
    Since $\xi_\eps$ is bounded in $L^\infty(\Omega)$ and by \eqref{v3}, \eqref{v4}, we find for the second and third term in \eqref{divterm} and for $\gamma < 2$ 
	\begin{align*}%\label{allairePeriAbsch2}
	   %\left|\int_0^\tau\int_\Omega\right.&  \left.\vphantom{\int_0^\tau\int_\Omega} q_{k,N}^{\eps,g} \xi_\eps  \psi\partial_t g_\eps\dx\dt\right|
		%\leq
        \left|\int_0^\tau\int_\Omega\right.&  \left.\vphantom{\int_0^\tau\int_\Omega}  q_{k,N}^{\eps,g} \xi_\eps g_\eps\partial_t\psi\dx\dt\right|
		+\left|\int_\Omega q_{k,N}^{\eps,g}\xi_\eps g_\eps(\tau,\cdot)\psi(\tau,\cdot)\dx\right|\\
        &\lesssim \| \xi_\eps\|_{L^{\frac{7}{3}}(\Omega)} \| g_\eps \chi_{M_\eps} \|_{L^\infty((0,T)\times\Omega)} \| q_{k,N}^{\eps,g} \|_{L^{\frac{7}{4}}(\Omega)}
        + \| \xi_\eps \|_{L^\infty(\Omega)} \| g_\eps\chi_{M_\eps^c} \|_{L^\infty(0,T;L^\gamma(\Omega))} \| q_{k,N}^{\eps,g}\|_{L^\frac{\gamma}{\gamma-1}(\Omega)}\\
        &\lesssim \| q_{k,N}^{\eps,g}\|_{L^{\frac{7}{4}}(\Omega)}+\| q_{k,N}^{\eps,g}\|_{L^\frac{\gamma}{\gamma-1}(\Omega)}{\rm Ma}(\eps)^{\frac{2}{\gamma}} \\
        &\lesssim (\eps^{\frac{\delta}{\gamma}}\eps^{-\frac{(1+\delta)(2-\gamma)}{\gamma}}+\eps^{-\frac{3(2-\gamma)}{\gamma}}){\rm Ma}(\eps)^{\frac{2}{\gamma}} \lesssim \eps^{-\frac{3(2-\gamma)}{\gamma}} {\rm Ma}(\eps)^{\frac{2}{\gamma}},
	\end{align*}
	where $\frac{\gamma}{\gamma-1}>2>\frac{3}{2}$ for any $\frac32 < \gamma < 2$, and the last inequality comes from the fact that $(\delta - \delta (2-\gamma)) = \delta(\gamma - 1) > 0$. Consequently, seeing that also $-3 \gamma^{-1}(2-\gamma) \in (-1, 0)$ for any $\gamma \in (\frac32, 2)$, the right-hand side vanishes due to \eqref{MaFinal}.\\
    
 For $\gamma \geq 2$ we have
    \begin{align*}%\label{allairePeriAbsch2}
	   %\left|\int_0^\tau\int_\Omega\right.&  \left.\vphantom{\int_0^\tau\int_\Omega} q_{k,N}^{\eps,g} \xi_\eps  \psi\partial_t g_\eps\dx\dt\right|
       \left|\int_0^\tau\int_\Omega\right.&  \left.\vphantom{\int_0^\tau\int_\Omega}  q_{k,N}^{\eps,g} \xi_\eps g_\eps\partial_t\psi\dx\dt\right|
		+\left|\int_\Omega q_{k,N}^{\eps,g}\xi_\eps g_\eps(\tau,\cdot)\psi(\tau,\cdot)\dx\right|\\
		 &\lesssim \| \xi_\eps\|_{L^{\frac{7}{3}}(\Omega)}\| g_\eps\chi_{M_\eps}\|_{L^\infty((0,T)\times\Omega)}\| q_{k,N}^{\eps,g}\|_{L^{\frac{7}{4}}(\Omega)}
        +\| \xi_\eps\|_{L^{\frac{2\gamma}{\gamma-2}}(\Omega)}\| g_\eps\chi_{M_\eps^c}\|_{L^\infty(0,T;L^\gamma(\Omega))}\| q_{k,N}^{\eps,g}\|_{L^2(\Omega)}\\
        &\lesssim \eps^{\frac{3\delta}{7}}+\eps^{\frac{1+2\delta}{7}}+{\rm Ma}(\eps)^{\frac{2}{\gamma}},
	\end{align*}
    and the right-hand side vanishes as $\eps\to 0$.
 
 Focusing on the first term in \eqref{divterm}, for $\gamma<3$, we can use \eqref{allaire5}, since $\frac{6\gamma}{5\gamma -6}\geq 2 > \frac{3}{2}$, and  $\|\hh_\eps\|_{L^2(0,T;L^{\frac{6\gamma}{\gamma+6}}(\Omega_\eps;\RR^3))}\lesssim {\rm Ma}(\eps)^{\min\{\frac{2}{\gamma},1\}}$ to get
	\begin{equation}\label{bem1}
		\begin{split}
	\Bigg| \int_0^\tau\int_\Omega \Grad(q_{k,N}^{\eps,g} \xi_\eps \psi )  \hh_\eps \dx\dt \Bigg| &\lesssim \| \Grad(q_{k,N}^{\eps,g} \xi_\eps \psi) \|_{L^2(0,T;L^{\frac{6\gamma}{5\gamma-6}}(\bigcup_{z_i\in n^\eps_N} C_i))}\|\hh_\eps\|_{L^2(0,T;L^{\frac{6\gamma}{6+\gamma}}(\Omega; \RR^3))}\\
&\lesssim \left(\| \Grad q_{k,N}^{\eps,g} \|_{L^{\frac{6\gamma}{5\gamma-6}}(\bigcup_{z_i\in n^\eps_N} C_i; \RR^3)}+\frac{1}{d_{\eps,i}}\| q_{k,N}^{\eps,g} \|_{L^{\frac{6\gamma}{5\gamma-6}}(\bigcup_{z_i\in n^\eps_N} C_i)}\right)\Ma (\eps)^{\min\{ \frac{2}{\gamma}, 1 \}}\\
&\lesssim \left( \eps^{-\frac{\gamma+6}{\gamma}}+\eps^{-(1+\delta)}  \eps^{6} \right)\Ma (\eps)^{\min\{ \frac{2}{\gamma}, 1 \}} \lesssim \eps^{-\frac{\gamma+6}{\gamma}} \Ma (\eps)^{\min\{ \frac{2}{\gamma}, 1 \}},
		\end{split}
	\end{equation}
	where we used that on each $C_i$
	\begin{align*}
		|\Grad(q_{k,N}^{\eps,g} \xi_\eps \psi )|\lesssim |\xi_\eps||\psi ||\Grad q_{k,N}^{\eps,g}|+|\Grad \xi_\eps||\psi ||q_{k,N}^{\eps,g}|+|\xi_\eps||\Grad\psi ||q_{k,N}^{\eps,g}|\lesssim |\Grad q_{k,N}^{\eps,g}|+\frac{1}{d_{\eps,i}}|q_{k,N}^{\eps,g}|.
	\end{align*}
	By assumption \eqref{MaFinal} the right-hand side of \eqref{bem1} vanishes for $\eps\to 0$.
    
    Finally, for $\gamma\geq 3$, we can use \eqref{allaire5}, $\frac{6\gamma}{5\gamma -6}\leq 2$, and  $\|\hh_\eps\|_{L^2(0,T;L^{\frac{6\gamma}{\gamma+6}}(\Omega_\eps;\RR^3))}\lesssim {\rm Ma}(\eps)^{\min\{\frac{2}{\gamma},1\}}$ to get
	\begin{equation}\label{bem2}
		\begin{split}
	\Bigg| \int_0^\tau\int_\Omega \Grad(q_{k,N}^{\eps,g} \xi_\eps \psi )  \hh_\eps \dx\dt \Bigg|&\lesssim \| \Grad(q_{k,N}^{\eps,g} \xi_\eps \psi) \|_{L^2(0,T;L^{\frac{6\gamma}{5\gamma-6}}(\bigcup_{z_i\in n^\eps_N} C_i; \RR^3))}\|\hh_\eps\|_{L^2(0,T;L^{\frac{6\gamma}{6+\gamma}}(\Omega; \RR^3))}\\
&\lesssim \left(\| \Grad q_{k,N}^{\eps,g} \|_{L^{\frac{6\gamma}{5\gamma-6}}(\bigcup_{z_i\in n^\eps_N} C_i; \RR^3)} + \frac{1}{d_{\eps,i}}\| q_{k,N}^{\eps,g} \|_{L^{\frac{6\gamma}{5\gamma-6}}(\bigcup_{z_i\in n^\eps_N} C_i)}\right)\Ma (\eps)^{\min\{ \frac{2}{\gamma}, 1 \}}\\
&\lesssim \left(\| \Grad q_{k,N}^{\eps,g} \|_{L^2(\bigcup_{z_i\in n^\eps_N} C_i; \RR^3)}+\frac{1}{d_{\eps,i}}\| q_{k,N}^{\eps,g} \|_{L^2(\bigcup_{z_i\in n^\eps_N} C_i)}\right)\Ma (\eps)^{\min\{ \frac{2}{\gamma}, 1 \}}\\
&\lesssim \left(\eps^{-3}+\eps^{-(1+\delta)} \right)\Ma (\eps)^{\min\{ \frac{2}{\gamma}, 1 \}} \lesssim \eps^{-3} \Ma (\eps)^{\min\{ \frac{2}{\gamma}, 1 \}},
		\end{split}
	\end{equation}
    By assumption \eqref{MaFinal} the right-hand side of \eqref{bem2} vanishes for $\eps\to 0$ and therefore the right-hand side of \eqref{divterm} does as well.
\end{proof}
    Having the above lemmas at hand, we can now show:
	\begin{lem}\label{allarirsfunkt2} Assume that ${\rm Ma}(\eps)$ obeys \eqref{MaFinal}.
		Let $(\oomega_k^{\eps,g},q_k^{\eps,g})_{1\le k\le 3}$ be the pairs of functions from Lemma~\ref{AllaireFunk}. Then there exist $\mmu_k \in W^{-1, \infty}(\Omega; \RR^3)$, $k \in \{1,2,3\}$, such that
        %$C_{k}>0$ such that, setting $\mmu_k=C_{k}\lambda\langle r\rangle\eins $ with $\eins =\sum_{j=1}^3 \ee_j$,
        the pairs $(\oomega_k^{\eps,g},\mmu_k)_{1\le k\le 3}$ satisfy \ref{H5}. 
	\end{lem}
	\begin{proof}
		%\ref{H1} and \ref{H2} are clear by construction. \cb{See comment in the definition of these functions, indeed $\oomega_k^\eps = e_k$ on the bad holes!} \ref{H3} follows directly from Lemma \ref{AllaireFunk}. Since $\mmu_k$ is a constant, \ref{H4} holds. \red Moreover for later convenience we also observe that   $q_k^\eps\in L^2(\Omega)$ and $q_k^\eps \rightharpoonup 0$ in $L^2(\Omega)/\RR$ weakly. \black
		%It is left to show \ref{H5}.\\[0.2cm]
		\item\paragraph{Step 1:} First, we show that there exist $\mmu_{k, N}$ such that \ref{H5} holds for $(\oomega_{k,N}^{\eps,g}, \mmu_{k,N})$,
        % with
		% \begin{align*}
		% 	\mmu_{k,N}=C_{k}\langle \cN_{\frac{2}{N}}(\Omega)\rangle   \langle r_N\rangle\eins,
		% \end{align*}
		%From \eqref{allaire4}-\eqref{allaire6}, we have that $(\oomega_{k,N}^{\eps,g},\mmu_{k,N}^\eps)_{1\le k\le d}$ satisfy \ref{H1}--\ref{H4}.  
		where $\oomega_{k, N}^{\eps, g}$ is given in Definition~\ref{def:Nfunctions}. To this end, we start in rewriting
		\begin{equation}\label{4a}
			\begin{split}
	\int_0^\tau\int_{\Omega} \Grad \oomega_{k,N}^{\eps,g} : \Grad (\psi  \vv_\eps)\dx\dt
&=\int_0^\tau\int_{\Omega} \Grad \oomega_{k,N}^{\eps,g}:\Grad(\psi\vv_\eps)-q_{k,N}^{\eps,g}\div(\psi \vv_\eps) \dx\dt\\&+\int_0^\tau\int_{\Omega} q_{k,N}^{\eps,g} \vv_\eps \cdot\Grad \psi \dx\dt+\int_0^\tau\int_{\Omega} \psi q_{k,N}^{\eps,g} \div(\vv_\eps)\dx\dt.
			\end{split}
		\end{equation}
	The first term of the right-hand side becomes
		\begin{align*}
			\int_0^\tau\int_{\Omega} \Grad \oomega_{k,N}^{\eps,g}:\Grad(\psi\vv_\eps)&-q_{k,N}^{\eps,g}\div(\psi \vv_\eps) \dx\dt
			=-\int_0^\tau\int_{\bigcup_{z_i\in n^\eps_N} C_i} (\Delta \oomega_{k,N}^{\eps,g}-\Grad q_{k,N}^{\eps,g})\cdot(\psi\vv_\eps)\dx\dt\\
			&+\int_0^\tau\int_{\Omega}\left(\sum\limits_{z_i\in n^\eps_N} \left(\frac{\partial \oomega_{k,N}^{\eps,g}}{\partial s_i}-\ee^i q_{k,N}^{\eps,g}\right)\delta_i^{\frac{d_{\eps,i}}{2}}\right)\cdot(\psi \vv_\eps)\dx\dt\\
			&+\int_0^\tau\int_{\bigcup_{z_i\in n^\eps_N} D_i}( \Grad \oomega_{k,N}^{\eps,g}:\Grad(\psi\vv_\eps)-q_{k,N}^{\eps,g}\div(\psi \vv_\eps))\dx\dt,
		\end{align*}
        where $s_i=|x-\eps z_i|$ is the radial coordinate and $\ee^i\colon x\mapsto\frac{x-\eps z_i}{|x-\eps z_i|}$ the unit vector of $x\in B_{cd_{\eps,i}}(\eps z_i)$ in $B_{cd_{\eps,i}}(\eps z_i)$. Due to the definition of $\oomega_{k,N}^{\eps,g}$ and $q_{k,N}^{\eps,g}$, the first term on the right-hand side above vanishes identically. For the last term we get from \eqref{allaire6}
		\begin{align*}
			\Bigg| \int_0^\tau\int_{\bigcup_{z_i\in n^\eps_N} D_i} &\Grad \oomega_{k,N}^{\eps,g}:\Grad(\psi\vv_\eps)-q_{k,N}^{\eps,g}\div(\psi \vv_\eps)\dx\dt \Bigg|\\
			&\lesssim (\| \Grad\oomega_{k,N}^{\eps,g}\|_{L^2(\Omega; \RR^3)}+\| q_{k,N}^{\eps,g}\|_{L^2(\Omega)})\| \vv_\eps\|_{L^2(0,T;W^{1,2}(\Omega;\RR^3))}\lesssim \eps^{\frac{\delta}{2}}.
		\end{align*}
		Hence, this term vanishes for $\eps\to 0$. Next, we define
		\begin{align*}
			\mmu_{k,N}^\eps:=\sum_{z_i\in n^\eps_N}\left(\frac{\partial \oomega_{k,N}^{\eps,g}}{\partial s_i}-\ee^i q_{k,N}^{\eps,g}\right)\delta_i^{\frac{d_{\eps,i}}{2}},
		\end{align*} 
	and show that 
		\begin{align}\label{konvfurN}
			\int_0^\tau\int_\Omega \mmu_{k,N}^\eps\cdot(\psi \vv_\eps)\dx\dt \to \int_0^\tau\int_\Omega \mmu_{k,N}\cdot( \psi\vv)\dx\dt,
		\end{align}
        where
        \begin{align*}
            \mmu_{k, N} = 2 \sigma_3 \langle \cN_\frac{2}{N}(\Omega) \rangle \langle r_N \rangle \FF_k, && \FF_k=\int_{\partial B_1} \left( \frac{\partial \oomega_k}{\partial \mathbf{n}}-q_k \mathbf{n} \right){\rm d}\mathcal{H}^{2}.
        \end{align*}
		By invoking \cite[Lemma 2.3.5]{Allaire1990a} we get
		\begin{align*}
			\left(\frac{\partial \oomega_{k,N}^{\eps,g}}{\partial s_i}-\ee^i q_{k,N}^{\eps,g}\right)\Bigg\lvert_{ s_i=\frac{d_{\eps,i}}{2}}
			=\frac{1}{a_{\eps,i}}\left(\frac{\partial \oomega_k}{\partial s_i}-\ee^i q_k\right)\Bigg\lvert_{ s_i=\frac{d_{\eps,i}}{2a_{\eps,i}}}
			= \frac{2 a_{\eps,i}}{\sigma_3 d_{\eps,i}^{2}}[\FF_k + 3(\FF_k\cdot \ee^i)\ee^i]+\R_\eps,
		\end{align*}
		where
		\begin{align*}
			\| \R_\eps \|_{L^\infty(\Omega)} \lesssim \frac{a_{\eps,i}^{2}}{d_{\eps,i}^3}\lesssim \eps^\delta d_{\eps,i}^{-2}  a_{\eps,i}.
		\end{align*}
	Now by Lemma \ref{reihen} it follows that
		\begin{align*}
			\sum_{z_i\in n_N^\eps} \left( \frac{d_{\eps,i}}{2} \right)^{-2} a_{\eps,i}[\FF_k+3(\FF_k\cdot\ee^i)\cdot\ee^i] \delta_i^\frac{d_{\eps, i}}{2} \to 2\sigma_3 \langle \cN_{\frac{2}{N}}(\Omega) \rangle\langle r_N\rangle \FF_k\quad\text{strongly in } W^{-1,2}(\Omega;\RR^3)\text{ a.s.}
		\end{align*}
		For the remaining term we have
		\begin{align*}
			-4C\sum_{z_i\in n^\eps_N}\eps^\delta d_{\eps,i}^{-2}  a_{\eps,i}\delta_i^{d_{\eps,i}/2}\leq \sum_{z_i\in n^\eps_N} \R_\eps (x)\cdot \ee_k \delta_i^{d_{\eps,i}/2}\leq 4C \sum_{z_i\in n^\eps_N}\eps^\delta d_{\eps,i}^{-2}  a_{\eps,i}\delta_i^{d_{\eps,i}/2}.
		\end{align*}
		By adding $4C\sum_{z_i\in n^\eps_N}\eps^\delta d_{\eps,i}^{-2}  a_{\eps,i}\delta_i^{d_{\eps,i}/2}$ on both sides we get 
		\begin{align*}
			0\leq  \sum_{z_i\in n^\eps_N} \R_\eps (x)\cdot \ee_k \delta_i^{d_{\eps,i}/2} + 4C\sum_{z_i\in n^\eps_N}\eps^\delta d_{\eps,i}^{-2}  a_{\eps,i}\delta_i^{d_{\eps,i}/2}\leq 8C\sum_{z_i\in n^\eps_N} \eps^\delta d_{\eps,i}^{-2}  a_{\eps,i}\delta_i^{d_{\eps,i}/2}.
		\end{align*}
		Note that Lemma \ref{reihen} enforces 
		\begin{align*}
			4 \sum_{z_i\in n^\eps_N}\eps^\delta d_{\eps,i}^{-2}  a_{\eps,i}\delta_i^{d_{\eps,i}/2}\to 0 \quad \text{ strongly in } W^{-1,2}(\Omega) \text{ a.s.}
		\end{align*}
		Using \cite[Lemma 2.3.8]{Allaire1990a} we deduce that
		\begin{align*}
			\sum_{z_i\in n^\eps_N} \R_\eps (x)\cdot \ee_k \delta_i^{d_{\eps,i}/2} + 4C \sum_{z_i\in n^\eps_N}d_{\eps,i}^{-2}  a_{\eps,i}\delta_i^{d_{\eps,i}/2}\to 0 \quad \text{ strongly in } W^{-1,2}(\Omega) \text{ a.s.},
		\end{align*}
		and finally
		\begin{align*}
			\sum_{z_i\in n^\eps_N} \R_\eps (x)\cdot \ee_k \delta_i^{d_{\eps,i}/2} \to 0 \quad \text{ strongly in } W^{-1,2}(\Omega)\text{ a.s.}
		\end{align*}
		Therefore, we find
		\begin{align*}
			\mmu_{k,N}^\eps \to  \mmu_{k,N} %=C_{k}\langle \cN_{\frac{2}{N}}(\Omega) \rangle\langle r_N\rangle\eins
            \quad\text{ strongly in } W^{-1,2}(\Omega;\RR^3)\text{ a.s.,}
		\end{align*}
	and \eqref{konvfurN} is proven.
		%\todo[inline]{Beweis von Fehlerterm}
		Using moreover $q_{k,N}^{\eps,g}\to 0$ strongly in $L^{\frac{6}{5}}(\Omega)$ and $\vv_\eps\rightharpoonup \vv$ weakly in $L^2(0,T;W^{1,2}(\Omega;\RR^3))$, we additionally get that
		\begin{align*}
			\int_0^\tau\int_{\Omega} q_{k,N}^{\eps,g} \Grad \psi \cdot \vv_\eps\dx\dt \to 0.
		\end{align*}
        Combining this with Lemma~\ref{divergenzterm}, the last two terms in \eqref{4a} vanish, which concludes the proof of this first step.
	
		\item\paragraph{Step 2:} We prove that \ref{H5} holds true for $\oomega_{k}^{\eps, g}$ and 
        \begin{align}\label{measDrag}
            \mmu_k = 2 \sigma_3 \lambda \langle r \rangle \FF_k, && \FF_k=\int_{\partial B_1} \left( \frac{\partial \oomega_k}{\partial \mathbf{n}}-q_k \mathbf{n} \right){\rm d}\mathcal{H}^{2}.
        \end{align}
        This follows arguing as in the proof of \cite[Lemma 2.5]{GiuntiHoefer2019}. Indeed, for each $N\in\NN$ we have
		\begin{align*}
			&\limsup_{\eps\to 0} \left| \int_0^\tau\int_\Omega \Grad\oomega_k^{\eps,g} :\Grad(\psi \vv_\eps)\dx\dt - \int_0^\tau\int_\Omega \mmu_k\cdot(\psi \vv)\dx\dt\right|\\
			&\lesssim \left| \int_0^\tau\int_\Omega (\mmu_{k,N}-\mmu_k)\cdot(\psi \vv)\dx\dt\right|
			+\limsup_{\eps\to 0} \left| \int_0^\tau\int_\Omega \Grad(\oomega_k^{\eps,g}-\oomega_{k,N}^{\eps,g}):\Grad(\psi \vv_\eps)\dx\dt\right|\\
			&\quad+\limsup_{\eps\to 0}\left| \int_0^\tau\int_\Omega \Grad\oomega_{k,N}^{\eps,g} :\Grad(\psi \vv_\eps)\dx\dt - \int_0^\tau\int_\Omega \mmu_{k,N}\cdot(\psi \vv)\dx\dt\right|\\
			&\lesssim \left| \int_0^\tau\int_\Omega (\mmu_{k,N}-\mmu_k)\cdot(\psi \vv)\dx\dt\right|
			+\limsup_{\eps\to 0}\sum_{z_i\in n^\eps}\| \Grad (\oomega_k^{\eps,g}-\oomega_{k,N}^{\eps,g})\|_{L^2(B_{d_{\eps,i}}(\eps z_i))}.
		\end{align*}		
		In the last inequality we used that \ref{H5} holds for $\oomega_{k,N}^{\eps,g}$ and that due to the definition of $\oomega_k^{\eps,g}$ and $\oomega_{k,N}^{\eps,g}$ we have
		\begin{align*}
			\supp (\oomega_k^{\eps,g}-\oomega_{k,N}^{\eps,g})&\subset~ \bigcup_{\mathclap{\substack{z_i\in n_N^\eps\\r_i\geq N}}}B_{d_{\eps,i}}(\eps z_i) \cup \bigcup_{\mathclap{z_i\in n^\eps\setminus n_N^\eps}}B_{d_{\eps,i}}(\eps z_i),\\
			\oomega_k^{\eps,g}-\oomega_{k,N}^{\eps,g}=\oomega_k^{\eps,g}-\ee_k &\quad\text{in } \bigcup_{\mathclap{z_i\in n^\eps\setminus n_N^\eps}}B_{d_{\eps,i}}(\eps z_i).
		\end{align*}
		Therefore, it follows along with Lemma \ref{AllaireFunk} that
		\begin{align}\label{a6}
			&\sum_{z_i\in n^\eps} \| \Grad (\oomega_k^{\eps,g}-\oomega_{k,N}^{\eps,g})\|_{L^2(B_{d_{\eps,i}}(\eps z_i); \RR^{3 \times 3})}\nonumber\\
			&\lesssim\sum_{\mathclap{\substack{z_i\in n_N^\eps\\r_i\geq N}}}(\| \Grad \oomega_k^{\eps,g}\|_{L^2(B_{d_{\eps,i}}(\eps z_i); \RR^{3 \times 3})} + \| \Grad \oomega_{k,N}^{\eps,g}\|_{L^2(B_{d_{\eps,i}}(\eps z_i); \RR^{3 \times 3})}) + \sum_{z_i\in n^\eps\setminus n_N^\eps}\| \Grad \oomega_k^{\eps,g}\|_{L^2(B_{d_{\eps,i}}(\eps z_i); \RR^{3 \times 3})}\nonumber\\
			&\lesssim\sum_{\mathclap{\substack{z_i\in n_N^\eps\\r_i\geq N}}}\eps^3(r_i + r_{i, N})
			+\sum_{z_i\in n^\eps\setminus n_N^\eps}\eps^3 r_i
			\lesssim \sum_{\mathclap{z_i\in n^\eps}}\eps^3r_i\id_{r_i\geq N}
			+\sum_{z_i\in n^\eps\setminus n_N^\eps}\eps^3 r_i.
		\end{align}
		The Strong Law of Large Numbers (Lemma \ref{SLLN}) yields that
		\begin{align*} 
			\lim_{\eps\to 0}\eps^3\sum_{z_i\in n^\eps}r_i\id_{r_i\geq N}=\langle r\id_{r \geq N} \rangle.
		\end{align*}
		We use \cite[Lemma C.1]{Allaire1990a}, \eqref{k1} and $n_N^\eps\subset n^\eps$ to conclude
		\begin{align}\label{b6}
			\lim_{N\to \infty} \lim_{\eps\to 0}\eps^3\#(n^\eps\setminus n_N^\eps)=\lim_{N\to \infty} \lambda |\Omega|-\langle \cN_{\frac{2}{N}}(\Omega) \rangle =0,
		\end{align}
		showing that \eqref{a6} vanishes for $\eps\to 0$ and $N\to \infty$. We conclude Step 2 by arguing that almost surely
		\begin{align}\label{jedna}
			\lim_{N\to\infty}\left| \int_0^\tau\int_\Omega (\mmu_{k,N}-\mmu_k)\cdot(\psi \vv)\dx\dt\right|= 0.
		\end{align}
		Indeed, with \cite[Lemma C.1]{Allaire1990a} we get
		\begin{align*}
			\lim_{N\to \infty}\langle \cN_{\frac{2}{N}}(\Omega)\rangle = \lambda |\Omega|,
		\end{align*}
		and by \eqref{meanrho} we have
		\begin{align*}
			\lim_{N\to \infty}\langle r_N \rangle
			= \langle r\rangle.
		\end{align*}
		This proves that
		\begin{align*}
			\mmu_{k,N}\xrightharpoonup{N\to \infty}\mmu_k\quad\text{weakly in } W^{-1,2}(\Omega; \RR^3),
		\end{align*}
        which in turn shows \eqref{jedna}.
	\end{proof}

In order to define finally the functions $\oomega_k^\eps$, we still need to define the ``bad'' functions $\oomega_k^{\eps, b}$. Let $\Ppsi\in C^\infty(\overline{\Omega};\RR^3)$ with $\div \Ppsi =0$. We recall from Lemma~\ref{giunti2} that $$\Omega_{\rm b}^\eps= \bigcup_{z_i\in J^\eps}B_{\theta\lambda_i^\eps\eps^{3}r_i}(\eps z_i)$$
		and $J^\eps=\bigcup_{k=-3}^{k_{\rm max}}J^\eps_k$, where for each $k$ the set of centers  $J_k^\eps$ is such that the corresponding balls are pairwise disjoint.
		 Let then $R_j, B_{R,j}, B_{\theta, j}, W_j$ be defined as in \eqref{def:sets}. We will define the ``bad'' functions $\Ppsi^\eps_{\rm b}$ recursively:
         
For all $0\leq i\leq k_{\max}+3$ let $\Ppsi_\eps^{-1} := \Ppsi$ and $\Ppsi_\eps^i$ be given by
		\begin{align*}
				\Ppsi_\eps^i:=
			\left\{ \begin{array}{ll}
			\Ppsi_\eps^{i-1} & \text{ in } \Omega\setminus \bigcup_{z_j\in J_{ k_{\rm max}-i}^\eps}B_{\theta,j},\\
			\Ppsi_\eps^{i-1}+\Ppsi_{j,\eps}^i\left(\frac{\cdot -\eps z_j}{\eps^{3} R_j}\right)& \text{ in } W_j\text{ for all } z_j\in J^\eps_{k_{\max}-i},
			 \\
			0 & \text{ in }B_{R,j} \text{ for all } z_j\in J^\eps_{k_{\max}-i},
			\end{array}\right.
		\end{align*}
		where $(\Ppsi_{j,\eps}^i,\pi_{j,\eps}^i)$ is the weak solution of
		\begin{align*}
			\left\{\begin{array}{ll}
				\Delta \Ppsi_{j,\eps}^i-\Grad \pi_{j,\eps}^i=0 & \text{ in } B_{\theta}\setminus B_{1}, \\
				\div \Ppsi_{j,\eps}^i=0 & \text{ in } B_{\theta}\setminus B_{1},\\
				\Ppsi_{j,\eps}^i=-\Ppsi_\eps^{i-1}(\eps^{3}R_j\cdot +\eps z_j)& \text{ on } \partial B_{1},\\
				\Ppsi_{j,\eps}^i=0 & \text{ on } \partial B_{\theta}.
			\end{array}\right.
		\end{align*}
		Finally, set $\Ppsi^\eps_{\rm b}:=\Ppsi_\eps^{k_{\max}+3}$. With this definition, we can show:
        
    \begin{lem}\label{PsiLemma}
       $\Psi^\eps_{\rm b}$ satisfies
		\begin{align}\label{phi_b}
			\| \Grad(\Ppsi^\eps_{\rm b}-\Ppsi)\|_{L^p(\Omega; \RR^{3 \times 3})}^p
			\lesssim 
			\left\{ \begin{array}{ll}
				\eps^{6(2-p)\delta}
				& \text{ for } 1<p\leq 2 \\
				\eps^{3(2-p)}& \text{ for } p>2
			\end{array} \right.\quad\text{ almost surely,}
		\end{align}
		and
		\begin{align}\label{L2konvphi_b}
		\lim_{\eps\to 0}	\| \Grad(\Ppsi^\eps_{\rm b}-\Ppsi)\|_{L^2(\Omega; \RR^{3 \times 3})}^2 = 0 \quad \text{ almost surely.}
		\end{align}
    \end{lem}

    \begin{proof}
	 
		We observe that for all $0\leq i\leq k_{\max}+3$
		\begin{align}\label{s}
			\| \Ppsi_\eps^i\|_{C^0(\Omega; \RR^3)}\lesssim \| \Ppsi\|_{C^0(\Omega; \RR^3)}.
		\end{align}
		In the next step we show by induction that for all $0\leq i\leq k_{\max}+3$ there holds
		\begin{align}\label{IV}
			\| \Grad (\Ppsi-\Ppsi_{\eps}^i)\|_{L^p(\Omega; \RR^{3 \times 3})}^p
			&\lesssim \sum\limits_{z_j\in \bigcup_{k=0}^iJ^\eps_{k_{\max}-k}} (\| \Grad \Ppsi\|_{L^p(B_{\theta,j}; \RR^{3 \times 3})}^p+(\eps^{3}R_j)^{3-p}\| \Ppsi\|_{C^0(\Omega; \RR^3)}^p).
		\end{align}
		Let $i=0$, then we have by definition for all $z_j\in J^\eps_{k_{\max}}$  
		\begin{align*}
			\| \Grad(\Ppsi-\Ppsi_\eps^0)\|_{L^p(B_{R,j}; \RR^{3 \times 3})} = \| \Grad \Ppsi\|_{L^p(B_{R,j}; \RR^{3 \times 3})}, &&
			\| \Grad(\Ppsi-\Ppsi_\eps^0)\|_{L^p(\Omega \setminus \bigcup_{z_j\in J_{k_{\max}}^\eps}B_{\theta,j}; \RR^{3 \times 3})}=0.
		\end{align*}
		Let $\eta$ be a cut-off function on $W_i$ with $\eta=0$ on $\partial B_{\theta,i}$, $\eta=1$ on $\partial B_{R,j}$, and $|\Grad \eta|\lesssim (\eps^{3}R_j)^{-1}$. By \cite[Theorem II.4.3. and Theorem IV.6.1.]{Galdi2011} and \eqref{s} we have
		\begin{equation}\label{ID2verfahren}
			\begin{split}
		\| \Grad (\Ppsi-\Ppsi_\eps^0)\|_{L^p(W_j; \RR^{3 \times 3})}^p
	&=(\eps^{3}R_j)^{3-p}\| \Grad \Ppsi_{j,\eps}^0\|_{L^p(B_\theta\setminus B_1; \RR^{3 \times 3})}^p\\
	&\lesssim (\eps^{3} R_j)^{3-p}\| (\eta\Ppsi)(\eps^{3}R_j\cdot+\eps z_i)\|_{W^{1-\frac{1}{p},p}(\partial(B_\theta\setminus B_1); \RR^3)}^p\\
	&\lesssim (\eps^{3}R_j)^{3-p}\| (\eta\Ppsi)(\eps^{3}R_j\cdot+\eps z_i)\|_{W^{1,p}(B_\theta\setminus B_1; \RR^3)}^p\\
	&\lesssim (\eps^{3}R_j)^{-p}\| \eta\Ppsi\|_{L^p(W_j; \RR^3)}^p + \| \Grad(\eta\Ppsi)\|_{L^p(W_j; \RR^3)}^p \\
	&\lesssim (\eps^{3}R_j)^{-p}\| \Ppsi\|_{L^p(W_j; \RR^3)}^p + \|\Grad\Ppsi\|_{L^p(W_j; \RR^3)}^p\\
	&\lesssim (\eps^{3}R_j)^{3-p}\| \Ppsi\|_{C^0(\Omega; \RR^3)}^p + \|\Grad\Ppsi\|_{L^p(W_j; \RR^3)}^p,
			\end{split}
		\end{equation}
		where we used that $\Grad\Ppsi_{j,\eps}^0$ and $(\eta\Ppsi)(\eps^{3} R_j\cdot + \eps z_i)$ fulfill the same boundary conditions in $B_\theta\setminus B_1$. Thus,
		\begin{align*}
			\| \Grad (\Ppsi-\Ppsi_\eps^0)\|_{L^p(\Omega; \RR^{3 \times 3})}^p
			&\lesssim \sum_{z_j\in J^\eps_{k_{\max}}}  (\| \Grad \Ppsi\|_{L^p(B_{\theta,j}; \RR^{3 \times 3})}^p+(\eps^{3}R_j)^{3-p}\| \Ppsi\|_{C^0(\Omega; \RR^3)}^p).
		\end{align*}
		%Induktionsschritt
		Assume now \eqref{IV} being true for $i-1$ and we want to prove it for $i$. By definition we have for all $z_j\in J^\eps_{k_{\rm max}-i}$
		\begin{align*}
			\| \Grad(\Ppsi_\eps^{i-1}-\Ppsi_\eps^i)\|_{L^p(B_{R,j}; \RR^{3 \times 3})}=\| \Grad \Ppsi_\eps^{i-1}\|_{L^p(B_{R,j}; \RR^{3 \times 3})},\quad
			\| \Grad(\Ppsi_\eps^{i-1}-\Ppsi_\eps^{i})\|_{L^p(\Omega \setminus \bigcup_{z_j\in J_{k_{\max}-i}^\eps}B_{\theta,j}; \RR^{3 \times 3})}=0.
		\end{align*}
		Using \eqref{s} and the same arguments as in \eqref{ID2verfahren}, we have for the difference between the gradients of $\Ppsi_{\eps}^{i-1}$ and $\Ppsi_\eps^i$ 
		\begin{align*}
			\| \Grad (\Ppsi_{\eps}^{i-1}-\Ppsi_{\eps}^i)\|_{L^p(W_j; \RR^{3 \times 3})}^p
			&=(\eps^{3}R_j)^{3-p}\| \Grad \Ppsi_{j,\eps}^i\|_{L^p(B_\theta\setminus B_1; \RR^{3 \times 3})}^p\\
			&\lesssim(\eps^{3}R_j)^{3-p}\| \Ppsi_{\eps}^{i-1}\|_{C^0(\Omega; \RR^3)}^p+\| \Grad \Ppsi_{\eps}^{i-1}\|_{L^p(W_j; \RR^{3 \times 3})}^p\\
			&\lesssim(\eps^{3}R_j)^{3-p}\| \Ppsi\|_{C^0(\Omega; \RR^3)}^p+\| \Grad \Ppsi_{\eps}^{i-1}\|_{L^p(W_j; \RR^{3 \times 3})}^p.
		\end{align*}
		This leads to
		\begin{align*}
			\| \Grad (\Ppsi_{\eps}^{i-1}-\Ppsi_{\eps}^i)\|_{L^p(\Omega; \RR^{3 \times 3})}^p \lesssim \sum_{z_j\in J_{k_{\max}-i}^\eps}(\| \Grad \Ppsi_{\eps}^{i-1}\|_{L^p(B_{\theta,j}; \RR^{3 \times 3})}^p
			+(\eps^{3}R_j)^{3-p}\| \Ppsi\|_{C^0(\Omega; \RR^3)}^p)
		\end{align*}
		such that we get for the difference between the gradients of $\Ppsi$ and $\Ppsi_{\eps}^i$
		\begin{align*}
			\| \Grad (\Ppsi-\Ppsi_{\eps}^i) \|_{L^p(\Omega; \RR^{3 \times 3})}^p &\lesssim \| \Grad (\Ppsi-\Ppsi_{\eps}^{i-1})\|_{L^p(\Omega; \RR^{3 \times 3})}^p \\
            &\quad + \sum\limits_{z_j\in J_{k_{\max}-i}^\eps} (\| \Grad \Ppsi_{\eps}^{i-1}\|_{L^p(B_{\theta,j}; \RR^{3 \times 3})}^p
			+(\eps^{3}R_j)^{3-p}\| \Ppsi\|_{C^0(\Omega; \RR^3)}^p)\\
			&\lesssim \| \Grad (\Ppsi-\Ppsi_{\eps}^{i-1})\|_{L^p(\Omega; \RR^{3 \times 3})}^p \\
            &\quad + \sum\limits_{z_j\in J_{k_{\max}-i}^\eps} (\| \Grad( \Ppsi_{\eps}^{i-1}-\Ppsi)\|_{L^p(B_{\theta,j}; \RR^{3 \times 3})}^p + (\eps^{3}R_j)^{3-p}\| \Ppsi\|_{C^0(\Omega; \RR^3)}^p+\| \Grad \Ppsi\|_{L^p(B_{\theta,j}; \RR^{3 \times 3})}^p)\\
			&\lesssim \| \Grad (\Ppsi-\Ppsi_{\eps}^{i-1})\|_{L^p(\Omega; \RR^{3 \times 3})}^p + \sum\limits_{z_j\in J_{k_{\max}-i}^\eps} (\| \Grad \Ppsi\|_{L^p(B_{\theta,j}; \RR^{3 \times 3})}^p + (\eps^{3}R_j)^{3-p}\| \Ppsi\|_{C^0(\Omega; \RR^3)}^p).
		\end{align*}
	Since \eqref{IV} holds for $i-1$ we deduce
		\begin{align}\label{badpartabsch}
			\| \Grad (\Ppsi-\Ppsi_{\eps}^i)\|^p_{L^p(\Omega; \RR^{3 \times 3})}
			&\lesssim \sum\limits_{z_j\in \bigcup_{k=0}^{i}J^\eps_{k_{\max}-k}} (\| \Grad \Ppsi\|_{L^p(B_{\theta,j}; \RR^{3 \times 3})}^p + (\eps^{3}R_j)^{3-p}\| \Ppsi\|_{C^0(\Omega; \RR^3)}^p)
		\end{align}
	and \eqref{IV} is proven. For $i=k_{\max}+3$, being $r_j\geq R_0>0$, $\eps^{3}R_j\le \Lambda \eps^{6\delta}$, and $\lambda_j^\eps\in [1,\Lambda]$, we have
		\begin{align*}
			\| \Grad(\Ppsi_{\rm b}^\eps-\Ppsi)\|_{L^p(\Omega; \RR^{3 \times 3})}^p &= \| \Grad(\Ppsi_\eps^{k_{\rm max}+3}-\Ppsi)\|_{L^p(\Omega; \RR^{3 \times 3})}^p
			\\
		&	\lesssim \sum_{z_j\in J^\eps} ~~(\| \Grad \Ppsi\|_{L^p(B_{\theta,j}; \RR^{3 \times 3})}^p + (\eps^{3}R_j)^{3-p}\| \Ppsi\|_{C^0(\Omega; \RR^3)}^p)
			\\
			&\lesssim \eps^{12\delta}\eps^{3}\sum_{z_j\in J^\eps} r_j
			+\eps^{3(2-p)}\eps^{3}\sum_{z_j\in J^\eps} r_j^{3-p}\\
		&	\lesssim \eps^3\sum\limits_{z_j\in J^\eps} r_j\cdot
			\left\{ \begin{array}{ll}
				(\eps^{12\delta}+\eps^{6(2-p)\delta})
			& \text{ for } 1<p\leq 2, \\
				(\eps^{12\delta}+\eps^{3(2-p)})& \text{ for } p>2.
			\end{array} \right.
		\end{align*}
		Since $J^\eps\subset\cI^\eps$ and $n^\eps=\Phi^\eps(\Omega)\setminus \cI^\eps$ we get from Lemma \ref{giunti1} that $\eps^3\#J^\eps \to 0$ almost surely if $\eps\to 0$. This together with the Strong Law of Large Numbers Lemma~\ref{SLLN} implies that
		\begin{align*}
			\lim_{\eps\to 0}\eps^3\sum_{z_i\in J^\eps}r_i=0 \quad\text{ almost surely.}
		\end{align*}
		Consequently, we obtain 
		\begin{align*}
			\| \Grad(\Ppsi^\eps_{\rm b}-\Ppsi)\|_{L^p(\Omega; \RR^{3 \times 3})}^p
			\lesssim 
			\left\{ \begin{array}{ll}
				\eps^{6(2-p)\delta}
				& \text{ for } 1<p\leq 2 \\
				\eps^{3(2-p)}& \text{ for } p>2
			\end{array} \right.\quad\text{ almost surely,}
		\end{align*}
		and
		\begin{align*}
		\lim_{\eps\to 0}	\| \Grad(\Ppsi^\eps_{\rm b}-\Ppsi)\|_{L^2(\Omega; \RR^{3 \times 3})}^2= 0\quad\text{ almost surely.}
		\end{align*}
    \end{proof}

   % \begin{lem}
    %    We define $\oomega_{k,\eps}^i$ in $\Omega$ for all $0\leq i\leq k_{\max}+3$ as
 %  \begin{align*}
  %     \left\{ \begin{array}{ll}
   %        \oomega_{k,\eps}^i:=\oomega_{k,\eps}^{i-1} & \text{in } \Omega\setminus \bigcup_{J_{z_j\in k_{max}-i}^\eps}B_{\theta,j}, \\
    %       \oomega_{k,\eps}^i:=0 & \text{in }B_{R,j} \text{ for all } z_j\in J^\eps_{k_{\max}-i},\\
     %      \oomega_{k,\eps}^i:=\oomega_{k,\eps}^{i-1}+\oomega_{k,\eps}^{i,j}\left(\frac{\cdot -\eps z_j}{\eps^{3} R_j}\right)& \text{in } W_j\text{ for all } z_j\in J^\eps_{k_{\max}-i},
      % \end{array}\right.
   %\end{align*}
   %where $\oomega_{k,\eps}^{-1}:=\ee_k$ and $(\oomega_{k,\eps}^{i,j},\pi_{k,\eps}^{i,j})$ is the weak solution of
   %\begin{align*}
    %   \left\{\begin{array}{ll}
    %       \Delta \oomega_{k,\eps}^{i,j}-\Grad \pi_{k,\eps}^{i,j}=0 & \text{in } B_{\theta}\setminus B_{1}, \\
    %       \div (\oomega_{k,\eps}^{i,j})=0 & \text{in } B_{\theta}\setminus B_{1},\\
    %       \oomega_{k,\eps}^{i,j}=-\oomega_{k,\eps}^{i-1}(\eps^{3}R_j\cdot +\eps z_j)& \text{on } \partial B_{1},\\
    %       \oomega_{k,\eps}^{i,j}=0 & \text{on } \partial B_{\theta}.\\
    %   \end{array}\right.
   %\end{align*}
   %Now, we define $\oomega_k^{\eps,b}:=\oomega_{k,\eps}^{k_{\max}+3}$ and 
   %\begin{align}\label{wbad}
   %    \oomega_k^{\eps,b}\to \ee_k \text{ strongly in } W^{1,2}(\Omega).
   %\end{align}
   % \end{lem}

  %  \begin{proof}
   %     \textcolor{cyan}{Steht eigentlich schon im Beweis von Lemma \ref{phi}}
   % \end{proof}

    \begin{defin}
        We choose $\Ppsi(x)=\ee_k$ in Lemma \ref{PsiLemma} and define $\oomega_k^{\eps,b}:=\Ppsi^\eps_{\rm b}$ and $\oomega_k^\eps:=\oomega_k^{\eps,g}+\oomega_k^{\eps,b}-\ee_k$.
    \end{defin}
    
    \begin{prop}
        Assume that ${\rm Ma}(\eps)$ obeys \eqref{MaFinal}. The functions $(\oomega_k^\eps,\mmu_k)$ satisfy \ref{H1}--\ref{H5}, where $\mmu_k$ is given by \eqref{measDrag}.
    \end{prop}

    \begin{proof}
        \ref{H1} and \ref{H2} are valid through the definition of $\oomega_k^{\eps,g}$ and $\oomega_k^{\eps,b}$. Since 
        $\oomega_k^{\eps,g} \rightharpoonup \ee_k$ weakly in $W^{1,2}(\Omega;\RR^3)$ (Lemma \ref{AllaireFunk}) and $\oomega_k^{\eps,b}\rightarrow \ee_k$ strongly in $W^{1,2}(\Omega;\RR^3)$ (Lemma \ref{PsiLemma}) hold, \ref{H3} is valid. \ref{H4} is true due to the definition of $\mmu_k$. Finally, \ref{H5} holds due to Lemma \ref{allarirsfunkt2} and \eqref{L2konvphi_b}.
        
        %we get for$\vv \in L^2(0,T;W^{1,2}(\Omega;\RR^3)$, $\vv_\eps \in L^2(0,T;W^{1,2}(\Omega_\eps;\RR^3))$ that satisfy \eqref{v1}--\eqref{v4} and for every $\psi\in C_c^\infty((0,T)\times\Omega)$
        	%\begin{align*}
        	%\lim_{\eps\to 0}	\int_0^\tau\int_{\Omega_\eps} &\Grad \oomega_k^\eps:\Grad(\psi\vv_\eps)\dx\dt
            %=\lim_{\eps\to 0}	\int_0^\tau\int_{\Omega_\eps} \Grad \oomega_{k,g}^\eps:\Grad(\psi\vv_\eps)\dx\dt+\lim_{\eps\to 0}	\int_0^\tau\int_{\Omega_\eps} \Grad (\oomega_{k,b}^\eps-\ee_k):\Grad(\psi\vv_\eps)\dx\dt\\
            %&=\int_0^\tau\int_\Omega\mmu_k \psi\vv\dx\dt ,
        	%\end{align*}
    	 %for a.e. $\tau\in[0,T]$, since $\oomega_{k,g}^\eps$ satisfies \ref{H5} and \eqref{wbad}.
    \end{proof}

%%%%%%%%%%%%%%%%%%%%%%%%%%%%%%%%%%%%%%%%%%%%%%%%%%%%%%%%%%%%%%
%%%%%%%%%%%%%%%%%%%%%%%%%%%%%%%%%%%%%%%%%%%%%%%%%%%%%%%%%%%%%%

	\subsection{Construction of test functions in the random setting}\label{sec:random-bis}

	In this last section we prove that, if ${\rm Ma}(\eps)$ obeys \eqref{MaFinal}, then the following holds:
	 for a given $\pphi\in C_c^\infty(\Omega;\RR^3)$ with $\div\pphi=0$ there exists a sequence $(\pphi_\eps)_{\eps>0}$ that satisfies %\eqref{M1-i}--\eqref{M1-iv}
	 \ref{M1} and \eqref{M1-v} with $\bM=(\mu_k^i)_{1\le k,i\le 3}$ given by $\mu_k^i:=\mmu_k\cdot\ee_i$ and
	$\mmu_k = 2 \sigma_3 \lambda\langle r\rangle \FF_k $.
	\begin{lem}\label{phi} 
		Assume ${\rm Ma}(\eps)$ obeys \eqref{MaFinal}. Then for every $\pphi\in C_c^\infty(\Omega;\RR^3)$ with $\div\pphi=0$ there exists a sequence $(\pphi_\eps)_{\eps >0}$ that satisfies $\pphi_\eps=0$ in $\Omega\setminus \Omega_\eps$, \eqref{M1-i}--\eqref{M1-iv} and \eqref{M1-v}. Moreover, almost surely,
		\begin{equation}\label{estimate-nabla}
	\| \Grad(\pphi_\eps-\pphi)\|_{L^p(\Omega; \RR^{3 \times 3})}^p
\lesssim 
\left\{ \begin{array}{ll}
	%\eps^{6(2-p)\delta}+\eps^{(1+2\delta)} +\eps^{p - \delta(p-1)}  & \text{for } 1<p \leq 2, \\[1em]
    \eps^{6(2-p)\delta}+\eps^{(1+2\delta)(2-p)} + \eps^{2-p + (p-1)\delta}%+\eps^{p - \delta(p-1)} 
	&\text{for } 1<p \leq 2, \\[1em]
	\eps^{3(2-p)} + \eps^{2-p+\delta} %\eps^{(1+\delta)(2-p) + (p-1)\delta}
    & \text{for } p>2.
\end{array} \right.
		\end{equation}
	\end{lem}

    \begin{rmk}
        Note that \eqref{estimate-nabla} combined with \eqref{MaFinal} yields \eqref{M1-iv}. Indeed, \eqref{M1-iv} requires $p=\frac{3\gamma}{2\gamma-3}$. Plugging into \eqref{estimate-nabla}, we find $3(2-p) = 3\frac{\gamma-6}{2\gamma-3}$, which is the last term of the exponent given in \eqref{MaFinal}.
    \end{rmk}
	
	\begin{proof}
		The construction of  $(\pphi_\eps)_{\eps>0}$ that satisfy $\pphi_\eps=0$ in $\Omega\setminus\Omega_\eps$, \eqref{M1-i}, and
	\begin{equation*}
	\pphi_\eps\rightharpoonup\pphi\quad\text{ weakly in }W^{1,2}(\Omega;\RR^3),
	\end{equation*}
		is done  in \cite[Lemma 2.5.]{GiuntiHoefer2019}. We repeat the construction to show that also \eqref{M1-iii}, \eqref{M1-v}, and \eqref{estimate-nabla} hold. This in turn implies the validity of \eqref{M1-ii}.

	\item\paragraph{Step 1: Validity of \eqref{estimate-nabla}.} 
		Let $\theta>1$ be fixed. Let $H^\eps_{\rm g}, H^\eps_{\rm b}, \Omega^\eps_{\rm b}$ be given as in Lemma~\ref{giunti1} and Lemma~\ref{giunti2}. We split our domain into two parts $\Omega=\Omega^\eps_{\rm b}\cup (\Omega\setminus \Omega^\eps_{\rm b})$. Next, we define 
		\begin{align*}
			\pphi_\eps:=\left\{ \begin{array}{ll}
				\pphi^\eps_{\rm b} & \text{in }\Omega^\eps_{\rm b}, \\
				\pphi^\eps_{\rm g} & \text{in }\Omega\setminus \Omega^\eps_{\rm b},
			\end{array} \right.
		\end{align*}
	with $\pphi_{\rm b}^\eps$ and  $\pphi_{\rm g}^\eps$ constructed as follows.\\

    \paragraph{Construction of $\pphi^\eps_{\rm b}$ and control over $\| \Grad(\pphi^\eps_{\rm b}-\pphi)\|_{L^p(\Omega; \RR^{3 \times 3})}^p$.}
        We choose $\Ppsi=\pphi$ in Lemma \ref{PsiLemma} and define $\pphi^\eps_{\rm b}:=\Ppsi^\eps_{\rm b}$.\\
	
	\paragraph{Construction of $\pphi^\eps_{\rm g}$ and control over $\| \Grad(\pphi^\eps_g-\pphi)\|_{L^p(\Omega; \RR^{3 \times 3})}^p$.} 
		For each $z_i\in n^\eps$ let $(\pphi_{1,\eps}^i,\pi_\eps^i)$ and $(\pphi_{2,\eps}^i,q_\eps^i)$ be the unique weak solutions to
		\begin{align*}
			\left\{\begin{array}{ll}
				\Delta \pphi_{1,\eps}^i-\Grad \pi_\eps^i=0 & \text{ in } \RR^3\setminus B_{1}, \\
				\div \pphi_{1,\eps}^i=0 & \text{ in } \RR^3\setminus B_{1},\\
				\pphi_{1,\eps}^i=\pphi(a_{\eps,i}\cdot +\eps z_i)& \text{ on } \partial B_{1},\\
				\pphi_{1,\eps}^i\to 0 & \text{ for } |x|\to\infty,\\
			\end{array}\right.
		\quad \text{ and }\quad
			\left\{\begin{array}{ll}
			\Delta \pphi_{2,\eps}^i-\Grad q_\eps^i=0 & \text{ in } B_{2}\setminus B_{1}, \\
			\div \pphi_{2,\eps}^i=0 & \text{ in } B_{2}\setminus B_{1},\\
			\pphi_{2,\eps}^i=\pphi_{1,\eps}^i\left(\frac{d_{\eps,i}\cdot}{2a_{\eps,i}}\right)& \text{ on } \partial B_{1},\\
			\pphi_{2,\eps}^i=0 & \text{ on } \partial B_{2},\\
		\end{array}\right.
		\end{align*}
	respectively. Recalling the sets $T_i, C_i, D_i, B_{2,i}$ from \eqref{sets}, we now define
		\begin{align*}
			\pphi^\eps_{\rm g}:=	\left\{ \begin{array}{ll}
			0 & \text{ in } T_i, \\
			\pphi-\pphi_{1,\eps}^i\left(\frac{\cdot-\eps z_i}{a_{\eps,i}}\right) & \text{ in } C_i,\\
				\pphi-\pphi_{2,\eps}^i\left(\frac{2(\cdot-\eps z_i)}{d_{\eps,i}}\right)& \text{ in } D_i,\\
			\pphi & \text{ in } \Omega\setminus \bigcup_{z_i\in n^\eps} B_{2,i}.
			\end{array}\right.
		\end{align*}
		
By \cite[Theorem 7.1.]{Maremonti1999} and \cite[Theorem II.4.3.]{Galdi2011} we have
		\begin{align}\label{phi1abschaetzung}
			\| \Grad \pphi_{1,\eps}^i\|^p_{L^p(\RR\setminus B_1; \RR^{3 \times 3})}
			&\lesssim \| \pphi(a_{\eps,i}\cdot -z_i)\|^p_{W^{1-\frac{1}{p},p}(\partial B_1; \RR^3)}
			\lesssim \| \pphi(a_{\eps,i}\cdot -z_i)\|^p_{W^{1,p}(B_1; \RR^3)}\nonumber\\
			&\lesssim a_{\eps,i}^{-d}\| \pphi\|^p_{L^p(T_i; \RR^3)}+a_{\eps,i}^{-(3-p)}\| \Grad\pphi\|^p_{L^p(T_i; \RR^{3 \times 3})}.
		\end{align}
		Hence, it follows
		\begin{align}\label{phi_gC}
			\| \Grad(\pphi^\eps_{\rm g}-\pphi)\|^p_{L^p(C_i; \RR^{3 \times 3})}
			&\lesssim a_{\eps,i}^{3-p}\| \Grad \pphi_{1,\eps}^i\|^p_{L^p(\RR\setminus B_1; \RR^{3 \times 3})}
			\lesssim a_{\eps,i}^{-p}\| \pphi\|^p_{L^p(T_i; \RR^3)} + \| \Grad\pphi\|^p_{L^p(T_i; \RR^{3 \times 3})}\nonumber\\
			&\lesssim a_{\eps,i}^{3-p}\| \pphi\|^p_{C^1(T_i; \RR^3)}+a_{\eps,i}^{3}\| \pphi\|^p_{C^1(T_i; \RR^3)}
			\lesssim a_{\eps,i}a_{\eps,i}^{-(p-2)}.
		\end{align}
		Let us now consider a cut-off function $\eta\in C^\infty_c(\RR^3)$ such that $\eta=0$ in $\RR^3\setminus B_{d_{\eps,i}a_{\eps,i}^{-1}}$ and $\eta=1$ in $ B_{\frac{1}{2}d_{\eps,i}a_{\eps,i}^{-1}}$. Using \cite[Theorem IV.6.1.]{Galdi2011}, \cite[Theorem II.4.3.]{Galdi2011}, \eqref{phi1abschaetzung}, and $|\Grad\eta|\lesssim \frac{a_{\eps,i}}{d_{\eps,i}}$, we deduce
		\begin{align*}
			\| \Grad(\pphi^\eps_{\rm g}-\pphi)\|^p_{L^p(D_i; \RR^{3 \times 3})}
			&\lesssim d_{\eps,i}^{3-p}\| \Grad\pphi_{2,\eps}^i\|_{L^p(B_2\setminus B_1; \RR^{3 \times 3})}^p
			\lesssim d_{\eps,i}^{3-p}\left\| \Grad\pphi_{2,\eps}^i \right\|_{W^{1-\frac{1}{p},p}(\partial (B_2\setminus B_1); \RR^{3 \times 3})}^p\\ 
			&\lesssim d_{\eps,i}^{3-p}\left\| (\eta\pphi_{1,\eps}^i)\left(\frac{d_{\eps,i}\cdot -\eps z_i}{2a_{\eps,i}}\right)\right\|_{W^{1,p}(B_2\setminus B_1; \RR^3)}^p\\
			&\lesssim a_{\eps,i}^{3} d_{\eps,i}^{-p}\left\| \eta\pphi_{1,\eps}^i\right\|_{L^p(E_i^0; \RR^3)}^p + a_{\eps,i}^{3-p}\left\| \Grad(\eta\pphi_{1,\eps}^i)\right\|_{L^p(E_i^0; \RR^{3 \times 3})}^p\\
			&\lesssim  a_{\eps,i}^{3} d_{\eps,i}^{-p}\left\| \pphi_{1,\eps}^i\right\|_{L^p(E_i^0; \RR^3)}^p + a_{\eps,i}^{3-p}\left\| \Grad\pphi_{1,\eps}^i\right\|_{L^p(E_i^0; \RR^{3 \times 3})}^p + a_{\eps,i}^{3} d_{\eps,i}^{-p}\left\| \pphi_{1,\eps}^i\right\|_{L^p(E_i^0; \RR^3)}^p.
		\end{align*}
	   Since
	\begin{align*}
		\| \pphi(a_{\eps,i}\cdot+\eps z_i)\|_{L^2(B_2; \RR^3)}^2 &= a_{\eps,i}^3 \| \pphi\|_{L^2(B_{2a_{\eps,i}}(\eps z_i); \RR^3)}^2 \lesssim a_{\eps,i}^3 a_{\eps,i}^{-3} \| \pphi\|_{C^0(\Omega; \RR^3)}^2\lesssim 1
	\end{align*}
and 
\begin{align*}
		\| \Grad(\pphi(a_{\eps,i}\cdot+\eps z_i))\|_{L^2(B_2; \RR^{3 \times 3})}^2
	%&=a_{\eps,i}^2\| \Grad\pphi(a_{\eps,i}\cdot+\eps z_i)\|_{L^2(B_2)}^2\\
	&
	=a_{\eps,i}^2 a_{\eps,i}^3 \| \Grad\pphi\|_{L^2(B_{2a_{\eps,i}}(\eps z_i); \RR^{3 \times 3})}^2
	\lesssim a_{\eps,i}^2 a_{\eps,i}^3 a_{\eps,i}^{-3} \| \pphi\|_{C^1(\Omega; \RR^3)}^2 \lesssim a_{\eps,i}^2,
\end{align*}
	we have $\| \pphi(a_{\eps,i}\cdot+\eps z_i)\|_{W^{1,2}(B_2; \RR^3)}\lesssim 1$.
    
Furthermore, from \cite[Theorem 6.1]{Maremonti1999} we get
	\begin{align*}
		|\Grad \pphi_{1,\eps}^i(x)|\lesssim \| \pphi(a_{\eps,i}\cdot+\eps z_i)\|_{W^{1,2}(B_2; \RR^3)}|x|^{-2}\lesssim |x|^{-2}\quad \text{ for every }|x|\geq 3.
	\end{align*}
	
As $d_{\eps,i}a_{\eps,i}^{-1}\geq \eps^{-\delta}$ we have for $\eps\ll 1$ that $|x|\geq 3$ for each $x\in E_i^0$. In this way
	\begin{align*}
		\left\| \Grad \pphi_{1,\eps}^i\right\|^p_{L^p(E_i^0; \RR^{3 \times 3})}
		\lesssim \int_{E_i^0} |x|^{-2p}
		\lesssim \left(\frac{a_{\eps,i}}{d_{\eps,i}}\right)^{3(p-1)-p}.
	\end{align*}
		Finally, we get
		\begin{align}\label{phi_gD}
			\| \Grad(\pphi^\eps_{\rm g}-\pphi)\|^p_{L^p(D_i; \RR^{3 \times 3})}
			&\lesssim a_{\eps,i}^{3} d_{\eps,i}^{-p}\left\| \pphi_{1,\eps}^i\right\|^p_{L^p(E_i^0; \RR^3)} + a_{\eps,i}^{3-p}\left\| \Grad\pphi_{1,\eps}^i\right\|^p_{L^p(E_i^0; \RR^{3 \times 3})}\nonumber\\
			&\lesssim a_{\eps,i}^{3-p}\left(\frac{a_{\eps,i}}{d_{\eps,i}}\right)^{3(p-1)-p}
			\lesssim a_{\eps,i} d_{\eps,i}^{(2-p)}\left(\frac{a_{\eps,i}}{d_{\eps,i}}\right)^{p-1}
			\lesssim a_{\eps,i} d_{\eps,i}^{2-p}\eps^{\delta(p-1)(d-2)}.
		\end{align}
		On the bad holes, $\phi_{\rm g}^\eps = \phi$. On the good holes we have $a_{\eps, i} \lesssim \eps^{1+2 \delta}$ and thus
		\begin{align}\label{phi_gT}
			\| \Grad(\pphi^\eps_{\rm g}-\pphi)\|^p_{L^p(T_i; \RR^3)}\lesssim \| \pphi\|_{C^1(\Omega, \RR^3)}a_{\eps,i}^{3}\lesssim a_{\eps,i} \eps^{2(1+2\delta)}.
		\end{align}
		Combining \eqref{phi_gC}, \eqref{phi_gD}, and \eqref{phi_gT}, we find
		\begin{align*}
			\| \Grad(\pphi^\eps_{\rm g}-\pphi)\|_{L^p(\Omega; \RR^{3 \times 3})}^p
		&	\lesssim \sum_{z_i\in n^\eps} a_{\eps,i} (a_{\eps,i}^{2-p}+d_{\eps,i}^{2-p}\eps^{(p-1)\delta}+\eps^{2(1+2\delta)}),\\
			&\lesssim\left\{ \begin{array}{ll}
				(\eps^{(1+2\delta)(2-p)}+\eps^{2-p}\eps^{(p-1)\delta}+\eps^{2(1+2\delta)})\sum_{z_i\in n^\eps} a_{\eps,i} & \text{for } 1<p<2\\[1em]
				(1+\eps^{\delta}+\eps^{2(1+2\delta)}) \sum_{z_i\in n^\eps} a_{\eps,i} & \text{for }p=2, \\[1em]
				(\eps^{3(2-p)}+\eps^{(1+\delta)(2-p)}\eps^{(p-1)\delta} + \eps^{2(1+2\delta)})\sum_{z_i\in n^\eps} a_{\eps,i} & \text{for }p>2.
			\end{array} \right.
		\end{align*}
		Lastly, we use the Strong Law of Large Numbers Lemma~\ref{SLLN} to conclude that almost surely
		\begin{align}\label{phi_g}
			\| \Grad(\pphi^\eps_{\rm g}-\pphi)\|_{L^p(\Omega; \RR^{3 \times 3})}^p
			&\lesssim\left\{ \begin{array}{ll}
				\eps^{(1+2\delta)(2-p)} + \eps^{2-p + (p-1)\delta}& \text{for } 1<p \leq 2,\\[1em]
				\eps^{3(2-p)}+\eps^{2-p+\delta}& \text{for }p>2.
			\end{array} \right.
		\end{align}

        Eventually, combining \eqref{phi_b} and \eqref{phi_g}, we find \eqref{estimate-nabla}.
		
		%%%%%%%%%%%%% Beweis (vi) %%%%%%%%%%%%%%%%%%%%%%%%%
		
		\item\paragraph{Step 2: Validity of \eqref{M1-v}.}
		By definition of $\pphi^\eps_{\rm b}$ and $\pphi^\eps_{\rm g}$ we have
		\begin{align*}
			\pphi_\eps=\pphi^\eps_{\rm b}+\pphi^\eps_{\rm g}-\pphi,
		\end{align*}
		therefore we can write
		\begin{align}\label{a}
			\int_0^T\int_\Omega \Grad \pphi_\eps : \Grad \vv_\eps\dx\dt
			=\int_0^T\int_\Omega \Grad (\pphi^\eps_{\rm b}-\pphi) : \Grad \vv_\eps\dx\dt+\int_0^T\int_\Omega \Grad \pphi^\eps_{\rm g} : \Grad \vv_\eps\dx\dt.
		\end{align}
		Using that $\|\vv_\eps\|_{L^2(0,T;W^{1,2}(\Omega;\RR^3))} \lesssim 1$ and \eqref{L2konvphi_b} we deduce  
		\begin{align*}
			\int_0^T\int_\Omega \Grad (\pphi_{\rm b}^\eps-\pphi) : \Grad \vv_\eps\dx\dt
			\lesssim \| \Grad (\pphi_{\rm b}^\eps-\pphi)\|_{L^2(\Omega;\RR^{3 \times 3})} \| \vv_\eps\|_{L^2(0,T;W^{1,2}(\Omega;\RR^3))}
			\xrightarrow{\eps\to 0} 0.
		\end{align*}
		Hence, it is left  to show that
		\begin{align}\label{g1}
		\lim_{\eps\to 0}	\int_0^T\int_\Omega \Grad \pphi^\eps_{\rm g} : \Grad \vv_\eps\dx\dt = \int_0^T\int_\Omega \Grad \pphi : \Grad \vv\dx\dt + \int_0^T\int_\Omega (\bM \pphi) \cdot\vv\dx\dt  .
		\end{align}
	To this end we write
		\begin{equation}\label{konvTerm}
			\begin{split}
				\int_0^T\int_\Omega \Grad \pphi^\eps_{\rm g} : \Grad \vv_\eps\dx\dt
			=\int_0^T\int_\Omega \Grad \pphi : \Grad \vv_\eps\dx\dt
			&	+\int_0^T\int_{\bigcup_{z_i \in n^\eps} D_i} \Grad (\pphi^\eps_{\rm g}-\pphi) : \Grad \vv_\eps\dx\dt\\
			&+
			\int_0^T\int_{\bigcup_{z_i \in n^\eps} C_i} \Grad (\pphi^\eps_{\rm g}-\pphi) : \Grad \vv_\eps\dx\dt.
			\end{split}
		\end{equation}
	Since $\vv_\eps\rightharpoonup\vv$ weakly in $L^2(0,T;W^{1,2}(\Omega;\RR^3))$, the first term on the right-hand side converges to ${\int_0^T\int_\Omega \Grad \pphi : \Grad \vv}\dx\dt$. To control the second term on the right-hand side we use \eqref{phi_gD} and the boundedness of  $\vv_\eps$ in $L^2(0,T;W^{1,2}(\Omega;\RR^3))$ to see that 
		\begin{align*}
			\Bigg| \int_0^T\int_{\bigcup_{z_i \in n^\eps} D_i} &\Grad (\pphi^\eps_{\rm g}-\pphi) : \Grad \vv_\eps \dx\dt\Bigg|
			\lesssim \| \Grad (\pphi^\eps_{\rm g}-\pphi) \|_{L^2(\bigcup_{z_i \in n^\eps} D_i; \RR^{3 \times 3})}\| \vv_\eps\|_{L^2(0,T;W^{1,2}(\Omega; \RR^3))}\\
			&\lesssim \left( \sum_{z_i\in n^\eps} a_{\eps,i}\eps^{\delta}\right)^{\frac{1}{2}}\| \vv_\eps\|_{L^2(0,T;W^{1,2}(\Omega;\RR^3))}
			\lesssim \eps^{\frac{\delta}{2}}\left( \eps^3\sum_{z_i\in n^\eps} r_i\right)^{\frac{1}{2}}.
		\end{align*}
		Due to the Lemma~\ref{SLLN} the right-hand side vanishes almost surely for $\eps\to 0$.
	Setting $\pphi=(\phi_1, \phi_2, \phi_3)$ and letting $\oomega_k^\eps$ be the function given by  Lemma \ref{AllaireFunk} we define a new function $\overline{\pphi}_\eps$ by
		\begin{align*}
				\overline{\pphi}_\eps:=
			\left\{\begin{array}{ll}
			\sum_{k=1}^3 \phi_k(\eps z_i)(\oomega_k^\eps((\cdot -\eps z_i)a_{\eps,i}^{-1}) -\ee_k)&  \text{ in } C_i \text{ for }z_i\in n^\eps,\\
		0& \text{ in } \Omega\setminus(\bigcup_{z_i\in n^\eps}C_i).
			\end{array}\right.     
		\end{align*}
		We write the last term of \eqref{konvTerm} as
		\begin{align}\label{konvTerm1}
			\int_0^T\int_{\bigcup_{z_i \in n^\eps} C_i} \Grad (\pphi^\eps_{\rm g}-\pphi) : \Grad \vv_\eps\dx\dt
			&=\int_0^T\int_{\bigcup_{z_i \in n^\eps}C_i} \Grad (\pphi^\eps_{\rm g}-\pphi- \overline{\pphi}_\eps): \Grad  \vv_\eps\dx\dt \\ 
            &\qquad+ \int_0^T\int_{\bigcup_{z_i \in n^\eps}C_i} \Grad \overline{\pphi}_\eps : \Grad \vv_\eps\dx\dt.
		\end{align}
		Now observe that $-\pphi_{1,\eps}^i- \sum_{k=1}^3 \phi_k(\eps z_i)\left( \oomega_k-\ee_k \right)$ solves the Stokes problem in the exterior domain $\RR^3\setminus B_1$ with boundary datum $-\pphi(a_{\eps,i}\cdot -\eps z_i)+\pphi(\eps z_i)$. As a consequence 
		\begin{align*}
			\| \Grad (\pphi^\eps_{\rm g}-\pphi- \overline{\pphi}_\eps)\|_{L^2(C_i; \RR^{3 \times 3})}^2
			&\lesssim 
			a_{\eps,i}^{-2}
			\| \Grad (-\pphi_{1,\eps}^i - \sum_{k=1}^3 \phi_k(\eps z_i)\left( \oomega_k-\ee_k \right))\left((\cdot -\eps z_i)a_{\eps,i}^{-1}\right)\|_{L^2(C_i; \RR^{3 \times 3})}^2\\
			&\lesssim a_{\eps,i}\| \Grad (-\pphi_{1,\eps}^i- \sum_{k=1}^3 \phi_k(\eps z_i)\left( \oomega_k-\ee_k \right))\|_{L^2(\RR\setminus B_1; \RR^{3 \times 3})}^2\\
			&\lesssim 
			 a_{\eps,i}\|\pphi(a_{\eps,i}\cdot+\eps z_i)-\pphi(\eps z_i)\|^2_{W^{1-\frac12,2}(\partial B_1; \RR^3)}\\
			 &\lesssim 
			 a_{\eps,i}\|
		(	\eta (\pphi-\pphi(\eps z_i)))(a_{\eps,i}\cdot+\eps z_i)\|^2_{W^{1-\frac12,2}(\partial B_1; \RR^3)},
		\end{align*}
	where $\eta$ is a cut-off function in $B_{ 2a_{\eps, i}}(\eps  z_i)\setminus T_i$ with $\eta=0$ on $\partial B_{2a_{\eps , i}}(\eps z_i)$ and $\eta=1$ on $\partial T_i$.
	Now, we use \cite[Theorem II.4.3]{Galdi2011}, $|\Grad\eta|\lesssim a_{\eps,i}^{-1}$, and a Lipschitz estimate to conclude
	\begin{align*}
		a_{\eps,i}\| (\eta(\pphi-\pphi(\eps z_i))) &(a_{\eps,i}\cdot +\eps z_i)\|_{W^{1-\frac{1}{2},2}(\partial B_1; \RR^3)}^2
		\lesssim  a_{\eps,i}\| (\eta(\pphi-\pphi(\eps z_i))) (a_{\eps,i}\cdot +\eps z_i)\|_{W^{1,2}(B_2\setminus B_1; \RR^3)}^2\\
		&\lesssim a_{\eps,i}^{-2}\| (\pphi-\pphi(\eps z_i))\|_{L^2(B_{2a_{\eps,i}}(\eps z_i)\setminus T_i; \RR^3)}^2 + \| \eta\Grad (\pphi-\pphi(\eps z_i)) \|_{L^2(B_{2a_{\eps,i}}(\eps z_i)\setminus T_i; \RR^{3 \times 3})}^2\\
		&\lesssim a_{\eps,i}^{-2}\| \pphi-\pphi(\eps z_i)\|_{L^2(B_{2a_{\eps,i}}(\eps z_i)\setminus T_i; \RR^3)}^2 + \| \Grad \pphi\|_{L^2(B_{2a_{\eps,i}}(\eps z_i)\setminus T_i; \RR^{3 \times 3})}^2\\
		&\lesssim a_{\eps,i}^{-2}\| \pphi\|_{C^1(\Omega; \RR^3)}^2\int_{B_{2a_{\eps,i}}(\eps z_i)\setminus T_i} |x-\eps z_i|^2\dx + a_{\eps,i}^d\| \pphi\|_{C^1(\Omega; \RR^3)}^2\\
		&\lesssim a_{\eps,i}^{-2}a_{\eps,i}^{5}\| \pphi\|_{C^1(\Omega; \RR^3)}^2+a_{\eps,i}^3 \| \pphi\|_{C^1(\Omega; \RR^3)}^2.
	\end{align*}
		Thus
		\begin{align*}
			\Big|\int_0^T\int_{\bigcup_{z_i \in n^\eps} C_i} &\Grad (\pphi^\eps_{\rm g}-\pphi- \overline{\pphi}_\eps): \Grad \vv_\eps\dx\dt\Big|
			\leq \| \Grad \vv_\eps\|_{L^2(0,T;L^2(\bigcup_{z_i\in n^\eps}C_i; \RR^{3 \times 3}))}\| \Grad (\pphi^\eps_{\rm g}-\pphi- \overline{\pphi}_\eps)\|_{L^2(\bigcup_{z_i\in n^\eps}C_i; \RR^{3 \times 3})}\\
			 &\lesssim \| \vv_\eps\|_{L^2(0,T;W^{1,2}(\Omega;\RR^3))}
			\left(\sum_{z_i\in n^\eps} \| \Grad (\pphi^\eps_{\rm g}-\pphi- \overline{\pphi}_\eps)\|_{L^2(C_i; \RR^{3 \times 3})}^2 \right)^{\frac{1}{2}}\\
			&\lesssim \| \vv_\eps\|_{L^2(0,T;W^{1,2}(\Omega;\RR^3))}
			\left(\sum_{z_i\in n^\eps} a_{\eps,i}^{3} \right)^{\frac{1}{2}}
			\lesssim \eps ^{1+2\delta}\| \vv_\eps\|_{L^2(0,T;W^{1,2}(\Omega;\RR^3))} 
			\left(\sum_{z_i\in n^\eps} a_{\eps,i}\right)^{\frac{1}{2}}.
		\end{align*}
	 From $\| \vv_\eps\|_{L^2(0,T;W^{1,2}(\Omega;\RR^3))} \lesssim 1$ and Lemma~\ref{SLLN}, the right-hand side in the above estimate vanishes almost surely for $\eps\to 0$ and so does the first term on the right-hand side of \eqref{konvTerm1}. In addition, the last term of \eqref{konvTerm1} yields
		\begin{align}\label{I2}
		\lim_{\eps\to 0} \int_0^T\int_{\bigcup_{z_i \in n^\eps} C_i} \Grad \overline{\pphi}_\eps : \Grad \vv_\eps \dx\dt= \int_0^T\int_\Omega \bM \pphi\cdot \vv\dx\dt  .
		\end{align}
		Indeed, by \ref{H5} we have
		\begin{align}\label{I3}
		\lim_{\eps\to 0} \sum_{k=1}^3 \int_0^T\int_{\Omega} \Grad \oomega_k^\eps : \Grad (\phi_k\vv_\eps)\dx\dt= \sum_{k=1}^3 \int_0^T\int_\Omega \mmu_k \cdot( \phi_k \vv)\dx\dt=\int_0^T\int_\Omega (\bM \pphi)\cdot \vv\dx\dt.
		\end{align}
		Hence
		\begin{equation}\label{I1}
			\begin{split}
		\int_0^T\int_{\bigcup_{z_i \in n^\eps} C_i} &\Grad \overline{\pphi}_\eps : \Grad \vv_\eps\dx\dt-\sum_{k=1}^3 \int_0^T\int_{\Omega} \Grad \oomega_k^\eps : \Grad (\phi_k\vv_\eps)\dx\dt\\
	&=\sum_{k=1}^3 \int_0^T\sum_{z_i\in n^\eps}\int_{C_i} (\phi_k(\eps z_i)-\phi_k)\Grad \oomega_k^\eps : \Grad \vv_\eps\dx\dt-\sum_{k=1}^3 \int_0^T\int_{\Omega} \Grad \oomega_k^\eps :(\vv_\eps \otimes\Grad\phi_k)\dx\dt\\
	& -\sum_{k=1}^3 \int_0^T\int_{\bigcup_{z_i \in n^\eps} D_i} \phi_k \Grad \oomega_k^\eps : \Grad \vv_\eps\dx\dt.
			\end{split}
		\end{equation}
		As $\pphi$ is smooth, we use a Lipschitz estimate, Lemma~\ref{AllaireFunk}, $\| \vv_\eps \|_{L^2(0,T;W^{1,2}(\Omega,\RR^3))}\lesssim 1$, and $d_{\eps, i} \leq \eps$ to conclude
		\begin{align*}
			&\Big|\sum_{k=1}^3 \int_0^T \sum_{z_i\in n^\eps}\int_{C_i} (\phi_k(\eps z_i)-\pphi)\Grad \oomega_k^\eps : \Grad \vv_\eps\dx\dt\Big| \\
            &\leq \sum_{k=1}^3 \int_0^T\sum_{z_i\in n^\eps}\int_{C_i} \|\phi_k(\eps z_i)-\pphi\|_{C^0(C_i; \RR^3)}|\Grad \oomega_k^\eps : \Grad \vv_\eps|\dx\dt\\
			&\lesssim d_{\eps,i} \sum_{k=1}^3 \int_0^T\int_{\bigcup_{z_i \in n^\eps} C_i} |\Grad \oomega_k^\eps : \Grad \vv_\eps|\dx\dt
			\lesssim d_{\eps,i} \sum_{k=1}^3 \| \Grad\oomega_k^\eps \|_{L^2(\Omega; \RR^{3 \times 3})}\| \vv_\eps
			\|_{L^2(0,T;W^{1,2}(\Omega;\RR^3))}
			\lesssim \eps  .
		\end{align*}
	Consequently, the first term on the right-hand side of \eqref{I1} vanishes as $\eps\to 0$.
		Since $\vv_\eps\rightharpoonup \vv$ weakly in $L^2(0,T; W^{1,2}(\Omega;\RR^3))$ and $\oomega_k^\eps \to \ee_k$ strongly in $L^2(\Omega; \RR^3)$ by \ref{H3}, together with $\div \vv = 0$, we have by partial integration
		\begin{align*}
	&\lim_{\eps\to 0} \sum_{k=1}^3 \int_0^T\int_{\Omega} \Grad \oomega_k^\eps : ( \vv_\eps\otimes\Grad \phi_k )\dx\dt \\
    &= - \lim_{\eps\to 0} \sum_{k=1}^3 \int_0^T\int_{\Omega} \div(\vv_\eps) (\oomega_k^\eps \cdot \Grad \phi_k) +\Delta_x \phi_k ((\oomega_k^\eps - \ee_k) \cdot \vv_\eps )  \dx\dt 
    = 0,
		\end{align*}
	and also the second term on the right-hand side of \eqref{I1} vanishes.
		Eventually, from Lemma \ref{AllaireFunk} and the fact that $\| \vv_\eps \|_{L^2(0,T;W^{1,2}(\Omega;\RR^3))}\lesssim 1$ we get
		\begin{align*}
			\Big| \sum_{k=1}^3 \int_0^T\int_{\bigcup_{z_i \in n^\eps} D_i} \phi_k \Grad \oomega_k^\eps : \Grad \vv_\eps\dx\dt \Big|\lesssim \| \phi_k\|_{C^0(\Omega; \RR^3)}\| \Grad\oomega_k^\eps\|_{L^2(\bigcup_{z_i \in n^\eps} A_i; \RR^{3 \times 3})}\| \vv_\eps\|_{L^2(0,T;W^{1,2}(\Omega;\RR^3))}\lesssim \eps^{\frac{\delta}{2}},
		\end{align*}
	hence also the last term of \eqref{I1} vanishes for $\eps\to 0$. This together with \eqref{I3} proves \eqref{I2}.
	\end{proof}
	
\section*{Acknowledgements}
{\it R. M. has received funding from the European Union's Horizon research and innovation program under the Marie Skłodowska-Curie agreement No 101150549. 
The research that led to the present paper was partially supported by
a grant from the GNAMPA group of INdAM.
F. O. has been supported by the Czech Science Foundation (GA\v CR) project 22-01591S, and the Czech Academy of Sciences project L100192351. The Institute of Mathematics, CAS is supported by RVO:67985840.
}

\bibliographystyle{plain}
\bibliography{ref}

\end{document}